\documentclass[11pt]{article}
\usepackage[a4paper,margin=1in]{geometry}
\usepackage{amsmath,amssymb,amsthm,mathtools}
\usepackage{enumitem}
\usepackage{needspace}
\usepackage{microtype}
\usepackage{cite}
\usepackage[hidelinks]{hyperref}

\newtheorem{theorem}{Theorem}[section]
\newtheorem{lemma}[theorem]{Lemma}
\newtheorem{proposition}[theorem]{Proposition}
\newtheorem{corollary}[theorem]{Corollary}
\newtheorem{example}[theorem]{Example}
\theoremstyle{remark}
\newtheorem{remark}[theorem]{Remark}
\newtheorem*{remarkstar}{Remark}
\theoremstyle{plain}
\newtheorem*{theoremA}{Theorem A}
\newtheorem*{theoremB}{Theorem B}

\numberwithin{equation}{section}

\newcommand{\R}{\mathbb R}

\DeclareMathOperator{\supp}{supp}

\newcommand{\Sph}{\mathbb S^{d-1}}
\newcommand{\essinf}{\operatorname*{ess\,inf}}
\newcommand{\esssup}{\operatorname*{ess\,sup}}

\title{Sharp thresholds and convergence for a nonlocal bistable equation}
\author{Yuanyang Hu\\[0.5ex]
\normalsize School of Mathematics and Statistics, Henan University\\
\normalsize 475004 Kaifeng, China\\[0.3ex]
\normalsize \texttt{yuanyhu@henu.edu.cn}}
\date{}
\hypersetup{
 pdftitle={Sharp thresholds and convergence for a nonlocal bistable equation},
 pdfauthor={Yuanyang Hu}
}

\begin{document}
\maketitle

\begin{abstract}
We study $u_t=J*u-u+f(u)$ on $\R^d$ with a bistable reaction and
radial nonincreasing probability kernels. Under the stated hypotheses,
every solution from compactly supported data between zero and one
converges to zero, locally to one, or to a single nonzero decaying
stationary state. At every nonzero decaying stationary state between
zero and one, we assume either nondegeneracy up to translations or
analyticity of the reaction near the closure of the state's range.
Decaying limits are strictly radially decreasing, with centers in the
initial support's closed convex hull.
For $u_0=\mathbf1_{B_R}$, we give explicit threshold bounds; analyticity near
$[0,1]$ yields a unique critical radius with a nonzero stationary limit.
\end{abstract}



\section{Introduction}\label{sec:introduction}

We consider the Cauchy problem
\begin{equation}\label{eq:main-equation}
 u_t=J*u-u+f(u),\qquad t>0,\quad x\in\R^d,
 \qquad u(0,x)=u_0(x),
\end{equation}
where $d\ge1$ is an integer, $u_0\in L^\infty(\R^d)$ satisfies $0\le u_0\le1$ almost everywhere, and
$(J*v)(x)=\int_{\R^d}J(x-y)v(y)\,dy$.
The dispersal kernel satisfies $J\in L^1(\R^d)$, $J\ge0$ almost everywhere on $\R^d$, and $\int_{\R^d}J(z)\,dz=1$.  Thus $J$ is a probability kernel.  Under assumption \textup{(F1)} below, the scalar equation $z'=f(z)$ has asymptotically stable equilibria $0$ and $1$ and an unstable equilibrium $\vartheta\in(0,1)$.

Under \textup{(F1)--(F2)} and \textup{(KR)} below,
\eqref{eq:main-equation} has a unique global solution
$u\in C^1([0,\infty);L^\infty(\R^d))$ with $0\le u\le1$ and
$u(0,\cdot)=u_0$. It depends continuously on $u_0$ in
$L^\infty(\R^d)$, uniformly on bounded time intervals.
Lemma~\ref{lem:semiflow} proves these assertions under \textup{(K0)},
using the comparison argument of Fife and Wang
\cite[Section~2]{FifeWang1998}.
We use this solution concept for bounded measurable data.
A stationary solution is a bounded measurable function $U$ satisfying
$J*U-U+f(U)=0$ almost everywhere.

For $x\in\R^d$, $|x|$ denotes the Euclidean norm. Set
$B_R(x_0)=\{x:|x-x_0|<R\}$ and $B_R=B_R(0)$, and write
$\mathbf1_A$ for the indicator of a measurable set $A$.
For measurable $A\subset\R^d$, $|A|$ denotes Lebesgue measure; in particular, $\omega_d=|B_1|$ and $|B_R|=\omega_dR^d$.  We write $\Sph=\{e\in\R^d:|e|=1\}$ and $\|\cdot\|_p$ for the norm in $L^p(\R^d)$, unless another domain is indicated.  In particular, $\|v\|_\infty=\operatorname{ess\,sup}_{x\in\R^d}|v(x)|$.  For functions defined on $\R$, the same norm notation refers to $L^p(\R)$.  Essential suprema and infima, and almost-everywhere statements, are with respect to Lebesgue measure.  We use standard Sobolev notation on $\R^d$ and $\R$, with $H^1(\R^d)=W^{1,2}(\R^d)$.
We write $\partial_i=\partial/\partial x_i$ for the derivative
in the $i$th coordinate, $i=1,\ldots,d$, understood in the weak sense
for Sobolev functions. We also write $\partial_{ij}=\partial_i\partial_j$.

Inequalities between measurable functions hold almost everywhere.  Local uniform convergence means convergence in $L^\infty(B_r)$ for every $r>0$; for continuous representatives this agrees with uniform convergence on compact sets.  A solution of \eqref{eq:main-equation} becomes extinct if $\|u(t,\cdot)\|_\infty\to0$ as $t\to\infty$, and it spreads if $u(t,\cdot)\to1$ locally uniformly as $t\to\infty$.  Unless another limit is specified, limits along an evolution are taken as $t\to\infty$.  For stationary solutions taking values in $[0,1]$, limits as $|x|\to\infty$ refer to the continuous representatives supplied by Lemma~\ref{lem:steadyregularity}.  A radial nonincreasing function has a representative of the form $v(x)=V(|x|)$ with $V$ nonincreasing on $[0,\infty)$.  For solutions with measurable initial data, this representative is understood up to equality almost everywhere.

For $0\le u_0\le1$ vanishing outside a ball, we ask whether every
solution becomes extinct, spreads to $1$, or converges to a single
stationary state tending to zero at infinity. For each $R>0$, let
$u^R$ be the unique global solution of \eqref{eq:main-equation}
with $u_0=\mathbf1_{B_R}$. Does a unique radius separate extinction
from spreading for this family?
Theorems~\ref{cs:thm:dynamics} and~\ref{thm:main} answer these
questions under the assumptions stated there.

\subsection{Assumptions on the nonlinearity}

\begin{itemize}[leftmargin=2.2em]
\item[(F1)] $f\in C^1(\R)$.  For some $\vartheta\in(0,1)$,
\begin{equation}\label{eq:bistable-signs}
 \begin{gathered}
 f(0)=f(\vartheta)=f(1)=0,\qquad
 f'(0)<0,\quad f'(\vartheta)>0,\quad f'(1)<0,\\
 f(s)<0\quad(0<s<\vartheta),\qquad
 f(s)>0\quad(\vartheta<s<1).
 \end{gathered}
\end{equation}
\item[(F2)]

$
 \kappa_f:=\sup_{0\le s\le1}f'(s)<1, m_f=1-\kappa_f>0.
$

\item[(F3)] $f$ is $C^2$ on an open interval containing $[0,1]$.
\item[(F4)] $f$ is real analytic on an open interval containing $[0,1]$.
\end{itemize}
Assumptions \textup{(F1)--(F2)} are imposed throughout.  The additional regularity assumptions \textup{(F3)} and \textup{(F4)} are stated where they are used; \textup{(F4)} implies \textup{(F3)}.

Define
\begin{equation}
 H(s)=s-f(s),\qquad F(s)=\int_0^s f(r)\,dr.
\end{equation}
By \textup{(F2)},
\begin{equation}\label{eq:reaction-monotonicity}
 H'(s)\ge m_f>0\qquad(0\le s\le1).
\end{equation} Since $H(0)=0$ and $H(1)=1$,
\eqref{eq:reaction-monotonicity} gives an inverse
$H^{-1}:[0,1]\to[0,1]$ with Lipschitz constant at most $m_f^{-1}$.
Hence every stationary solution with $0\le U\le1$ satisfies
$U=H^{-1}(J*U)$ and has a uniformly continuous representative;
see Lemma~\ref{lem:steadyregularity}.

We call $F(1)=0$ the balanced case. Theorem~\ref{thm:main}
treats $F(1)>0$; for $F(1)<0$, Corollary~\ref{cor:main_negative}
uses $v=1-u$. The relation between $F(1)$ and the wave speed is
stated in Theorem~A below.

\subsection{Assumptions on the kernel}

Each of the following kernel assumptions includes the preceding one.

\begin{itemize}[leftmargin=2.2em]
\item[(K0)] $J\in L^1(\R^d)$, $J\ge0$ almost everywhere, and
\begin{equation}\label{eq:kernel-mass}
 \int_{\R^d}J(z)\,dz=1.
\end{equation}

\item[(K1)] In addition to (K0),
\begin{equation}
 J(z)=J(-z)\quad\text{for a.e. }z,\qquad
 J\in W^{2,1}(\R^d),\qquad
 \int_{\R^d}|z|J(z)\,dz<\infty,
\end{equation}
and there is $\rho>0$ such that $J>0$ almost everywhere on $B_\rho(0)$.  

\item[(KR)] In addition to (K1),  $J$ is radial, and
\begin{equation}
 J\in C^2(\R^d)\cap L^\infty(\R^d),\qquad
 \int_{\R^d}|z|^2J(z)\,dz<\infty.
\end{equation}

\item[(KM)] In addition to (KR), 
\begin{equation}
 J(x)=j(|x|)\quad\text{with }j\text{ nonincreasing on }[0,\infty).
\end{equation}
\end{itemize}

\subsection{Previous results}

For local diffusion, Kanel' \cite{Kanel1964} studied extinction and
spreading in combustion problems, Fife and McLeod
\cite{FifeMcLeod1977} proved convergence to one-dimensional bistable
fronts, and Aronson and Weinberger \cite{AronsonWeinberger1978}
developed multidimensional spreading theory. Sharp thresholds for
compactly supported data were studied by Zlato\v{s}
\cite{Zlatos2006}, Du and Matano \cite{DuMatano2010}, and
Pol\'a\v{c}ik \cite{Polacik2011}. In higher dimensions, Du and
Pol\'a\v{c}ik \cite{DuPolacik2015} proved locally uniform convergence
for bounded solutions from nonnegative compactly supported data under
regularity and nondegeneracy assumptions on the reaction.
Muratov and Zhong \cite{MuratovZhong2013,MuratovZhong2017} used energy
methods for symmetric decreasing data; their higher-dimensional
bistable threshold theorem assumes that the central values of radial
nonincreasing ground states form a totally disconnected set
\cite[Theorem~4]{MuratovZhong2017}. Nadin \cite{Nadin2026} studied
perturbations of local threshold solutions through an adjoint
linearized equation.

For convolution equations, Bates, Fife, Ren and Wang \cite{BatesFifeRenWang1997} and Chen \cite{Chen1997} developed the bistable traveling-wave theory, while Fife and Wang \cite{FifeWang1998} constructed radial barriers for higher-dimensional spreading.  Alberti and Bellettini \cite{AlbertiBellettini1998} studied nonlocal phase-transition energies and optimal profiles.  Extinction for initial data with sufficiently small $L^1(\R^d)$ mass, with an explicit sufficient bound, was proved for bounded radial probability kernels by Lim \cite[Theorem~4.10(ii)]{Lim2017}.

Fang and Zhao \cite{FangZhao2015} established existence results for bistable traveling waves of discrete- and continuous-time monotone semiflows in homogeneous and periodic habitats.  Their existence results assume bistability and a counter-propagation condition; their weak compactness formulation on monotone profiles also covers nonlocal dispersal equations.  Zhang and Zhao \cite{ZhangZhao2021} proved uniqueness and stability results for bistable waves of monotone semiflows.

Zhang, Li and Yang \cite{ZhangLiYang2022} studied sharp thresholds for $u_0=\mathbf1_{(-R,R)}$ in one space dimension.  Theorem~1.1 of their revised author manuscript states subsequential uniform convergence of the critical solution to a stationary solution under a $C^2$ bistable reaction and an even nonnegative $C^1$ probability kernel with integrable derivative and finite first moment.  In the sign convention of \eqref{eq:main-equation}, their
integral and monotonicity conditions become $F(1)>0$ and
\textup{(F2)}, respectively.

Besse, Capel, Faye and Fouilh\'e \cite{BesseCapelFayeFouilhe2023} studied extinction, propagation and pinning as the dispersal strength varies, including regimes with discontinuous stationary states.  Alfaro, Ducrot and Kang \cite{AlfaroDucrotKang2023} obtained quantitative extinction and propagation estimates for $u_0=(\vartheta+\varepsilon)\mathbf1_{(-R,R)}$ as $\varepsilon\downarrow0$, with bounds depending on the decay of the kernel at infinity. Related nonlocal equations admit higher-dimensional stationary
states \cite{BatesChmaj1999} and one-dimensional homoclinic solutions
with instability results \cite{ChmajRen1999}; multilayer patterns,
including formal whole-line constructions, were studied in
\cite{ChmajRen2000}.  

The \L ojasiewicz--Simon inequality has also been used to prove convergence for nonlocal equations.  Feireisl, Issard-Roch and Petzeltov\'a \cite{FeireislIssardRochPetzeltova2004} treated nonlocal phase-field systems on bounded domains; the operator assumptions and convergence theorem are stated in \cite[Section~3.4 and Theorem~3.4]{FeireislPetzeltova2008}.  Hutson and Grinfeld \cite{HutsonGrinfeld2006} proved convergence for sufficiently small dispersal on bounded spatial domains by a perturbation argument.

\subsection{Known results used in the proofs}

The proofs use the following one-dimensional traveling-wave theorem and radial spreading estimates.

For a one-dimensional kernel $K$, we use the conditions
\begin{equation}\label{1.10}
\begin{gathered}
 K\in W^{2,1}(\R),\qquad K\ge0,\qquad K(s)=K(-s)\quad\text{for a.e. }s\in\R,\\
 \int_\R K(s)\,ds=1,\qquad \int_\R |s|K(s)\,ds<\infty,\\
 K>0\quad\text{a.e. on }(-\rho_K,\rho_K)\quad\text{for some }\rho_K>0.
\end{gathered}
\end{equation}
\Needspace{12\baselineskip}
\begin{theoremA}[\cite{BatesFifeRenWang1997}]
Assume \textup{(F1)--(F3)}, and let $K$ satisfy \eqref{1.10}.  Then there exist a unique speed $c\in\R$ and, up to translation, a unique decreasing profile $\Phi$ such that
\begin{equation}\label{eq:planar-front}
 -c\Phi'=K*\Phi-\Phi+f(\Phi),\qquad
 \Phi(-\infty)=1,\qquad
 \Phi(+\infty)=0,
\end{equation}
with $0<\Phi<1$, $\Phi\in C^1(\R)$ and $\Phi'\in L^1(\R)\cap L^2(\R)$.  Moreover, the speed satisfies
\begin{equation}\label{eq:front-speed-identity}
 c\int_\R |\Phi'(\xi)|^2\,d\xi=F(1).
\end{equation}
\end{theoremA}

\begin{remarkstar}
The existence, regularity and uniqueness assertions follow from the traveling-wave results of Bates, Fife, Ren and Wang \cite{BatesFifeRenWang1997}; the change of dependent variable and auxiliary $C^2$ extension in Section~\ref{sec:basic} produce a globally bistable reaction to which their results apply and preserve the original equation for profiles taking values in $[0,1]$.  If $f$ is $C^\infty$ on a neighborhood of $[0,1]$, Chen \cite{Chen1997} gives an alternative existence and uniqueness argument after a smooth auxiliary extension outside $[0,1]$.

By monotonicity and the end states in \eqref{eq:planar-front}, $\|\Phi'\|_1=1$.  If $c\ne0$, \eqref{eq:planar-front} gives $\|\Phi'\|_\infty\le(1+\max_{[0,1]}|f|)/|c|$.  If $c=0$, differentiating $\Phi-f(\Phi)=K*\Phi$ gives $(1-f'(\Phi))\Phi'=K'*\Phi$, so $\|\Phi'\|_\infty\le m_f^{-1}\|K'\|_1$.  Thus $\|\Phi'\|_2^2\le\|\Phi'\|_\infty\|\Phi'\|_1<\infty$.

By \eqref{eq:planar-front}, $u(t,x)=\Phi(x-ct)$ solves
\eqref{eq:main-equation} with $d=1$, $J=K$ and $u_0=\Phi$;
uniqueness follows from Lemma~\ref{lem:semiflow}. By
\eqref{eq:front-speed-identity},
$\operatorname{sgn}c=\operatorname{sgn}F(1)$.

\end{remarkstar}

For $e\in\Sph$, set $e^\perp=\{z\in\R^d:z\cdot e=0\}$ and define $K_e(s)=\int_{e^\perp}J(se+z_\perp)\,dz_\perp$, where $dz_\perp$ is $(d-1)$-dimensional Lebesgue measure on $e^\perp$.  For $d=1$, set $K_e(s)=J(se)$.  Fix $e_1\in\Sph$ and write
\begin{equation}\label{eq:radial-marginal}
 K_*(s)=K_{e_1}(s)=\int_{e_1^\perp}J(se_1+z_\perp)\,dz_\perp,
\end{equation}
with $K_*(s)=J(se_1)$ when $d=1$.
Under \textup{(KR)}, Lemma~\ref{lem:projection} gives
\eqref{1.10} for $K_*$, and radiality makes $K_*$ independent of
$e_1$. If also \textup{(F3)} holds and $F(1)>0$, let
$(c_*,\Phi_*)$ be the speed and a decreasing profile from Theorem~A
for $K_*$. Then $c_*>0$ by \eqref{eq:front-speed-identity}, and
\begin{equation}
	-c_*\Phi_*'=K_* *\Phi_*-\Phi_*+f(\Phi_*),\qquad
	\Phi_*(-\infty)=1,\qquad
	\Phi_*(+\infty)=0.
\end{equation}
\begin{theoremB}[\cite{FifeWang1998}]
Assume \textup{(KR)}, \textup{(F1)--(F3)} and $F(1)>0$.  
For $u_0\in L^\infty(\R^d)$ with $0\le u_0\le1$, let $u$ denote the unique global solution of \eqref{eq:main-equation} with initial datum $u_0$.

\begin{enumerate}[label=\textup{(\roman*)},leftmargin=2.7em]
\item \textit{Inner propagation lower bound.}
For every $\sigma\in(0,c_*)$ and $\eta\in(0,1-\vartheta)$, there exists $L=L(\sigma,\eta)>0$ such that, whenever
\begin{equation}
 0\le u_0\le1,\qquad
 u_0(x)\ge\vartheta+\eta\quad\text{for }|x|\le L,
\end{equation}
the solution $u$ satisfies
\begin{equation}
 \lim_{t\to\infty}\essinf_{|x|\le\sigma t}u(t,x)=1.
\end{equation}
\item \textit{Outer-decay upper bound.}
If $u_0$ is compactly supported and $u(t,\cdot)\to1$ locally uniformly as $t\to\infty$, then
\begin{equation}
 \lim_{t\to\infty}\esssup_{|x|\ge\sigma t}u(t,x)=0
 \qquad\text{for every }\sigma>c_*.
\end{equation}
\end{enumerate}
\end{theoremB}

\begin{remarkstar}
The inner estimate follows from the radial lower barrier of Fife and Wang \cite{FifeWang1998}, and the outer estimate follows from the upper barrier described in their subsequent remark.  In the notation of Fife and Wang \cite{FifeWang1998}, $w=2u-1$ and $g(s)=-2f((1+s)/2)$ for $s\in[-1,1]$.  Their increasing front is $W(\xi)=2\Phi_*(-\xi)-1$, $\xi\in\R$, and has speed $-c_*$.  The time-dependent radial shifts have asymptotic speeds that can be chosen smaller than $c_*-\sigma$ for part \textup{(i)} and smaller than $\sigma-c_*$ for part \textup{(ii)}, giving the respective limits.  For part \textup{(i)}, choose $0<\eta_0<\eta$ sufficiently small for the continuous-data lower estimate, and let $L_0$ be the required radius.  Set $L=L_0+1$ and choose $\zeta\in C_c^\infty(B_L)$ with $0\le\zeta\le1$ and $\zeta=1$ on $B_{L_0}$.  The continuous initial datum $(\vartheta+\eta_0)\zeta$ lies below $u_0$, so comparison gives the estimate for bounded measurable data.  A continuous compactly supported majorant of $u_0$ gives the upper estimate in part \textup{(ii)} in the same way.  In particular, for all sufficiently large $R$, the solution $u^R$ with initial datum $u_0=\mathbf1_{B_R}$ satisfies both estimates in Theorem~B.  Indeed, part \textup{(i)} with one fixed speed gives local convergence to $1$.  For any $0<\sigma<\sigma'<c_*$, applying part \textup{(i)} from a sufficiently large time $T$ at speed $\sigma'$ gives the estimate at speed $\sigma$, since $\sigma t\le\sigma'(t-T)$ for all sufficiently large $t$.  Part \textup{(ii)} gives the outer estimate.
\end{remarkstar}

\subsection{Main results}

We first describe the decaying stationary states and the limits of
solutions with $0\le u_0\le1$ vanishing outside a ball.
When $F(1)>0$, let $\beta$ be the unique zero of $F$ in $(\vartheta,1)$.
Write
\[
 X=L^2(\R^d)\cap L^\infty(\R^d),\qquad
 \|v\|_X=\|v\|_2+\|v\|_\infty.
\]
Using the continuous representatives supplied by
Lemma~\ref{lem:steadyregularity}, define
\begin{equation}\label{cs:eq:stationary-sets}
\begin{split}
 \mathcal S(J,f)&=\{U\in L^\infty(\R^d):0\le U\le1,\ H(U)=J*U\},\\
 \mathcal S_0(J,f)&=\{U\in\mathcal S(J,f):U(x)\to0\text{ as }|x|\to\infty\},\\
 \mathcal S_1(J,f)&=\{U\in\mathcal S(J,f):U(x)\to1\text{ as }|x|\to\infty\}.
\end{split}
\end{equation}
The additional regularity needed below is
\begin{equation}\label{cs:eq:reaction-regularity}
 f\in C^2((0,1)),\qquad
 \sup_{0<s\le1/2}s|f''(s)|<\infty.
\end{equation}
Assumption \textup{(F3)} implies \eqref{cs:eq:reaction-regularity}.

\begin{theorem}\label{cs:thm:stationary-summary}
Assume \textup{(F1)--(F2)}, \textup{(KM)}, $F(1)>0$, and
\eqref{cs:eq:reaction-regularity}.
For each $U\in\mathcal S_0(J,f)\setminus\{0\}$, there is
a pair $(x_*,V)\in\R^d\times C^2([0,\infty))$ such that
\begin{equation}\label{cs:eq:stationary-summary}
 \begin{gathered}
 U(x)=V(|x-x_*|),\qquad
 V'(r)<0\quad(r>0),\qquad \beta<V(0)<1,\\
 U\in L^1(\R^d)\cap H^1(\R^d),\qquad
 \int_{\R^d}|x|^2U(x)\,dx<\infty,\qquad
 \int_{\R^d}f(U(x))\,dx=0.
 \end{gathered}
\end{equation}
The pair $(x_*,V)$ is uniquely determined by $U$.
\end{theorem}

For $U\in\mathcal S(J,f)$, define the bounded operator
$\mathcal L_U$ on $L^2(\R^d)$ by
\begin{equation}\label{eq:main-linearization}
 \mathcal L_U\phi=J*\phi-\phi+f'(U)\phi.
\end{equation}
For $U\in\mathcal S(J,f)$ with $\partial_iU\in L^2(\R^d)$,
$i=1,\ldots,d$, we use the condition
\begin{equation}\label{cs:eq:translation-kernel}
 \ker_{L^2(\R^d)}\mathcal L_U
 =\operatorname{span}\{\partial_1U,\ldots,\partial_dU\}.
\end{equation}
Here
$\ker_{L^2(\R^d)}\mathcal L_U
=\{\phi\in L^2(\R^d):\mathcal L_U\phi=0\text{ in }L^2(\R^d)\}$.
For a closed interval $I\subset[0,1]$, analyticity near $I$ means
that $f$ agrees on $[0,1]\cap O$ with an analytic function defined
on an open neighborhood $O$ of $I$.
For initial data satisfying
\begin{equation}\label{cs:eq:data}
 u_0\in L^\infty(\R^d),~ 0\le u_0\le1,~~~
 u_0=0\ a.e. \text{ on }\R^d\setminus B_{R_0}
 \text{ for some }R_0>0,
\end{equation}
define
\begin{equation}\label{cs:eq:initial-convex-hull}
 \Sigma_0=\Bigl\{x\in\R^d:\int_{B_r(x)}u_0(y)\,dy>0\ \forall r>0\Bigr\},
 \qquad C_0=\overline{\operatorname{co}}(\Sigma_0).
\end{equation}
Here, $C_0$ is the
closed convex hull of $\Sigma_0$.  For $u_0\equiv0$, both sets are empty.

\begin{theorem}\label{cs:thm:dynamics}
Assume \textup{(F1)--(F2)}, \textup{(KM)}, $F(1)>0$, and
\eqref{cs:eq:reaction-regularity}.  Suppose that, for any fixed
$U\in\mathcal S_0(J,f)\setminus\{0\}$, either $f$ is real analytic
near $[0,\|U\|_\infty]$, or \eqref{cs:eq:translation-kernel} holds.
For every $u_0$ satisfying \eqref{cs:eq:data}, the unique global
solution $u$ of \eqref{eq:main-equation} with initial datum $u_0$
satisfies exactly one of
the following alternatives:
\begin{enumerate}[label=\textup{(\roman*)},leftmargin=2.7em]
\item $\|u(t,\cdot)\|_X\to0$ as $t\to\infty$.
\item $\|u(t,\cdot)-1\|_{L^\infty(B_r)}\to0$ as $t\to\infty$
      for every $r>0$.
\item 
 $\|u(t,\cdot)-U\|_X\to0$ as $t\to\infty$ for some $U\in\mathcal S_0(J,f)\setminus\{0\}$.
For the pair $(x_*,V)$ associated with $U$ in
Theorem~\ref{cs:thm:stationary-summary}, we have $x_*\in C_0$.
\end{enumerate}

\end{theorem}

For $F(1)<0$, the reflection $v=1-u$ gives counterparts of
Theorems~\ref{cs:thm:stationary-summary} and~\ref{cs:thm:dynamics}.
Let $\alpha\in(0,\vartheta)$ be the unique number satisfying
$F(\alpha)=F(1)$, and consider the condition
\begin{equation}\label{cs:eq:reaction-regularity-one}
 f\in C^2((0,1)),\qquad
 \sup_{1/2\le s<1}(1-s)|f''(s)|<\infty.
\end{equation}
As with \eqref{cs:eq:reaction-regularity}, assumption \textup{(F3)}
implies \eqref{cs:eq:reaction-regularity-one}.

\begin{corollary}\label{cs:cor:stationary-negative}
Assume \textup{(F1)--(F2)}, \textup{(KM)} and $F(1)<0$.
Then $\mathcal S_0(J,f)=\{0\}$. If also
\eqref{cs:eq:reaction-regularity-one} holds, then for each
$U\in\mathcal S_1(J,f)\setminus\{1\}$ there is a pair
$(x_*,P)\in\R^d\times C^2([0,\infty))$ such that
\begin{equation}\label{cs:eq:stationary-one-summary}
\begin{gathered}
 U(x)=P(|x-x_*|),\qquad P'(r)>0\ (r>0),\qquad
 0<P(0)<\alpha,\\
 1-U\in L^1(\R^d)\cap H^1(\R^d),\qquad
 \int_{\R^d}|x|^2(1-U(x))\,dx<\infty,\qquad
 \int_{\R^d}f(U(x))\,dx=0.
\end{gathered}
\end{equation}
The pair $(x_*,P)$ is uniquely determined by $U$.
\end{corollary}

For initial data satisfying
\begin{equation}\label{cs:eq:one-data}
 u_0\in L^\infty(\R^d),\qquad 0\le u_0\le1,\qquad
 u_0=1\ \text{a.e. on }\R^d\setminus B_{R_0}
 \quad\text{for some }R_0>0,
\end{equation}
define
\begin{equation}\label{cs:eq:one-convex-hull}
 \Sigma_0^\sharp=\Bigl\{x\in\R^d:
 \int_{B_r(x)}(1-u_0(y))\,dy>0\ \forall r>0\Bigr\},
 \qquad C_0^\sharp=\overline{\operatorname{co}}(\Sigma_0^\sharp).
\end{equation}
Thus $C_0^\sharp$ is the closed convex hull of the support of
$(1-u_0(x))\,dx$.

\begin{corollary}\label{cs:cor:dynamics-negative}
Assume \textup{(F1)--(F2)}, \textup{(KM)}, $F(1)<0$, and
\eqref{cs:eq:reaction-regularity-one}. Suppose that, for each
$U\in\mathcal S_1(J,f)\setminus\{1\}$, either $f$ is real analytic
near $[\inf_{\R^d}U,1]$, or \eqref{cs:eq:translation-kernel} holds.
For every $u_0$ satisfying \eqref{cs:eq:one-data}, the unique global
solution $u$ of \eqref{eq:main-equation} with initial datum $u_0$
satisfies exactly one of the following alternatives as $t\to\infty$:
\begin{enumerate}[label=\textup{(\roman*)},leftmargin=2.7em]
\item $\|u(t,\cdot)-1\|_X\to0$.
\item $\|u(t,\cdot)\|_{L^\infty(B_r)}\to0$ for every $r>0$.
\item There is a $U\in\mathcal S_1(J,f)\setminus\{1\}$
such that $\|u(t,\cdot)-U\|_X\to0$.
For the pair $(x_*,P)$ associated with $U$ in
Corollary~\ref{cs:cor:stationary-negative}, we have $x_*\in C_0^\sharp$.
\end{enumerate}
\end{corollary}

\Needspace{10\baselineskip}
For $u_0=\mathbf1_{B_R}$, we now determine the extinction and
spreading intervals and the limit at the critical radius.

\begin{theorem}\label{thm:main}
Assume \textup{(F1)--(F2)} and $F(1)>0$.  For each $R>0$, let $u^R$ denote the unique global solution of \eqref{eq:main-equation} with initial datum $u_0=\mathbf1_{B_R}$.

\begin{enumerate}[label=\textup{(\roman*)},leftmargin=2.7em]
\item Assume \textup{(KR)}. Then there exist radii
\begin{equation}
 0<R_{\rm ext}\le R_{\rm prop}<\infty
\end{equation}
such that $u^R$ becomes extinct if and only if $0<R<R_{\rm ext}$, and spreads to $1$ if and only if $R>R_{\rm prop}$. With
\begin{equation}\label{eq:threshold-constants}
 T_f=m_f^{-1}\log\frac4\vartheta,
 \qquad
 M_1(J)=\int_{\R^d}|z|J(z)\,dz,
 \qquad
 \omega_d=|B_1|,
\end{equation}
the extinction radius satisfies
\begin{equation}\label{eq:main-extinction-bounds}
 \left[\frac{1}{\omega_d\|J\|_\infty(e^{T_f}-1)}\right]^{1/d}
 \le R_{\rm ext}
 \le \frac{dM_1(J)}{4F(1)}.
\end{equation}

\item Assume \textup{(KM)} and \textup{(F4)}. Then
\begin{equation}
 R_{\rm ext}=R_{\rm prop}=:R_*(f,J).
\end{equation}
At the critical radius $R=R_*(f,J)$, there is a radial nonincreasing
$U_*\in\mathcal S_0(J,f)\setminus\{0\}$ such that
\[
 \|u^{R_*}(t,\cdot)-U_*\|_2+\|u^{R_*}(t,\cdot)-U_*\|_\infty\to0
 \qquad\text{as }t\to\infty.
\]
Moreover, for every nonconstant radial nonincreasing
$U\in\mathcal S_0(J,f)\cap L^2(\R^d)$, the operator $\mathcal L_U$ on $L^2(\R^d)$, defined by
\eqref{eq:main-linearization}, has a positive eigenvalue with a
nonzero nonnegative radial eigenfunction.
\end{enumerate}
\end{theorem}

For the case $F(1)<0$, retain \textup{(F1)--(F2)} and \textup{(KR)},
and set $f^\sharp(s)=-f(1-s)$ and
$F^\sharp(s)=\int_0^s f^\sharp(r)\,dr$.
For each $R>0$, let $u_R$ be the unique global solution of
\eqref{eq:main-equation} with $u_0=1-\mathbf1_{B_R}$, and set
$v^R=1-u_R$. By Lemma~\ref{lem:duality}, $v^R$ solves
\eqref{eq:main-equation} with reaction $f^\sharp$ and initial datum
$v^R(0,\cdot)=\mathbf1_{B_R}$, and $F^\sharp(1)=-F(1)>0$.
The sets $\mathcal S(J,f^\sharp)$ and $\mathcal S_0(J,f^\sharp)$
are defined by \eqref{cs:eq:stationary-sets} with $f^\sharp$ in place of $f$.
\begin{corollary}\label{cor:main_negative}
Assume \textup{(F1)--(F2)}, \textup{(KR)} and $F(1)<0$.

\begin{enumerate}[label=\textup{(\roman*)},leftmargin=2.7em]
\item Let $R_{\rm ext}^\sharp$ and $R_{\rm prop}^\sharp$ be the radii
in Theorem~\ref{thm:main}\textup{(i)} for the reaction $f^\sharp$.
Then $0<R_{\rm ext}^\sharp\le R_{\rm prop}^\sharp<\infty$, and
$u_R(t,\cdot)$ converges uniformly to $1$ if and only if
$0<R<R_{\rm ext}^\sharp$, and locally uniformly to $0$ if and only if
$R>R_{\rm prop}^\sharp$.
The bounds \eqref{eq:main-extinction-bounds} hold for $R_{\rm ext}^\sharp$,
with $F(1)$ and $T_f$ replaced by $F^\sharp(1)$ and
$T_{f^\sharp}=m_f^{-1}\log(4/(1-\vartheta))$, respectively.

\item Assume also \textup{(KM)} and \textup{(F4)}. Then $R_{\rm ext}^\sharp=R_{\rm prop}^\sharp=R_*^\sharp$, where $R_*^\sharp=R_*(f^\sharp,J)$.  At the critical radius $R=R_*^\sharp$, there is a radial nonincreasing
$U_*^\sharp\in\mathcal S_0(J,f^\sharp)\setminus\{0\}$ such that
$v^{R_*^\sharp}(t,\cdot)\to U_*^\sharp$ in
$L^2(\R^d)\cap L^\infty(\R^d)$ as $t\to\infty$.
Consequently, $1-U_*^\sharp\in\mathcal S(J,f)$ and
\[
 \|u_{R_*^\sharp}(t,\cdot)-(1-U_*^\sharp)\|_2+
 \|u_{R_*^\sharp}(t,\cdot)-(1-U_*^\sharp)\|_\infty\to0
 \qquad\text{as }t\to\infty.
\]
The spectral conclusion of Theorem~\ref{thm:main}\textup{(ii)} also holds with $f^\sharp$ in place of $f$.
\end{enumerate}
\end{corollary}

In Theorem~\ref{thm:main}\textup{(ii)} and
Corollary~\ref{cor:main_negative}\textup{(ii)}, respectively,
Theorem~\ref{cs:thm:stationary-summary} applies to $U_*$ with reaction $f$
and to $U_*^\sharp$ with reaction $f^\sharp$:
\textup{(F4)} implies \eqref{cs:eq:reaction-regularity}
and is preserved by Lemma~\ref{lem:duality}.
The next corollary records the one-dimensional consequence of
Theorem~\ref{thm:main}\textup{(ii)}.

\Needspace{7\baselineskip}
\begin{corollary}\label{thm:1dcompletefull}
Assume $d=1$, (KM), \textup{(F1)--(F2)}, \textup{(F4)}, and $F(1)>0$.  Set $L_*=R_*(f,J)$ as in Theorem~\ref{thm:main}\textup{(ii)}.  Let $u^{L_*}$ be the unique global solution of \eqref{eq:main-equation} with initial datum $u_0=\mathbf1_{B_{L_*}}$.  Then $u^{L_*}$ satisfies
\begin{equation}\label{eq:display-1-19}
 \|u^{L_*}(t,\cdot)-U_*\|_2+
 \|u^{L_*}(t,\cdot)-U_*\|_\infty\to0
 \qquad\text{as }t\to\infty,
\end{equation}
where $U_*\in\mathcal S_0(J,f)\setminus\{0\}$ is the critical limit
in Theorem~\ref{thm:main}\textup{(ii)}. It is even and nonincreasing
on $[0,\infty)$, with $0<U_*<1$ on $\R$.
\end{corollary}

For the profile in Corollary~\ref{thm:1dcompletefull},
Theorem~\ref{cs:thm:stationary-summary} gives $U_*(0)>\beta$;
Lemma~\ref{lem:peakidentity} gives the integral identity \eqref{eq:peakidentity}.

For the present convolution equation, reflection arguments and the
finite-dimensional gradient reductions of Simon \cite{Simon1983}
and Chill \cite{Chill2003} yield symmetry and convergence as proved
in Sections~\ref{cs:sec:stationary} and~\ref{cs:sec:compact-data}.

For solutions $u^R$ that neither become extinct nor spread,
\eqref{eq:asymptotic-translation-modulus}, the tail bounds
\eqref{eq:transition-pointwise-tail}--\eqref{eq:transition-l2-tail},
and time continuity from Lemma~\ref{lem:semiflow} give relative
compactness in $L^2(\R^d)\cap L^\infty(\R^d)$ under
\textup{(KM)} and $F(1)>0$. Under \textup{(F4)}, the Fredholm
factorization of Lemma~\ref{lem:fredholm} and the gradient inequality
of Lemma~\ref{lem:lojasiewicz} then give convergence by
Proposition~\ref{prop:transitionconverges}.

The paper is organized as follows. Section~\ref{sec:basic} establishes
the semiflow, energy, front and reflection estimates.
In Section~\ref{cs:sec:stationary}, we prove
Theorem~\ref{cs:thm:stationary-summary} on decaying stationary solutions.
Section~\ref{cs:sec:compact-data} proves
Theorem~\ref{cs:thm:dynamics} by compactness and a local gradient inequality,
and derives Corollaries~\ref{cs:cor:stationary-negative}
and~\ref{cs:cor:dynamics-negative} by $v=1-u$.
In Section~\ref{sec:ballthresholds}, we study the family
$u_0=\mathbf1_{B_R}$ and prove Theorem~\ref{thm:main} and its corollaries.
Section~\ref{sec:remarks} discusses the assumptions and gives examples.

\section{Basic estimates and compactness}\label{sec:basic}

Unless stated otherwise, $u$ denotes the solution of
\eqref{eq:main-equation} with $u_0\in L^\infty(\R^d)$ and $0\le u_0\le1$.
We retain \textup{(F1)--(F2)} throughout and state the kernel
assumptions in each result.

\subsection{The bistable nonlinearity and the semiflow}

\begin{lemma}\label{lem:nonlinearity}
Assume \textup{(F1)--(F2)}, and let $H(s)=s-f(s)$ and $F(s)=\int_0^s f(r)\,dr$.  Then $H'(s)\ge m_f>0$ on $[0,1]$.
Since $H(0)=0$ and $H(1)=1$, the restriction $H:[0,1]\to[0,1]$ is a strictly increasing $C^1$ diffeomorphism, and its inverse is Lipschitz with constant at most $m_f^{-1}$.

If $F(1)>0$, there is a unique $\beta\in(\vartheta,1)$ such that $F(\beta)=0$.
Moreover,
\begin{equation}\label{eq:potential-signs}
 F(s)<0\quad(0<s<\beta),
 \qquad
 F(s)>0\quad(\beta<s\le1),
\end{equation}
and $F(s)\le F(1)$ for $0\le s\le1$.
\end{lemma}

\begin{proof}
By \eqref{eq:reaction-monotonicity}, $H'\ge m_f$ on $[0,1]$.
Together with $H(0)=0$ and $H(1)=1$, the inverse-function theorem
gives the diffeomorphism and Lipschitz bound.

If $F(1)>0$, \eqref{eq:bistable-signs} makes $F$ strictly decreasing
from $0$ to $F(\vartheta)<0$ and strictly increasing from
$F(\vartheta)$ to $F(1)>0$. The intermediate value theorem therefore
gives the unique zero $\beta\in(\vartheta,1)$, the signs in
\eqref{eq:potential-signs}, and $F(s)\le F(1)$ on $[0,1]$.
\end{proof}

\begin{remark}
For the equation with a separate dispersal strength,
\begin{equation}
 u_t=\nu(J*u-u)+f(u),\qquad \nu>0,
\end{equation}
the map $s\mapsto\nu s-f(s)$ has derivative at least
$\nu-\kappa_f>0$ on $[0,1]$ when $\nu>\kappa_f$.
Under \textup{(K0)}, a stationary state $0\le U\le1$ satisfies
$\nu U-f(U)=\nu J*U$. Translation continuity of $J$ in $L^1(\R^d)$
makes $J*U$ uniformly continuous; the Lipschitz inverse of
$s\mapsto\nu s-f(s)$ therefore gives a uniformly continuous
representative of $U$. Here $\nu=1>\kappa_f$ by \textup{(F2)}.
For smaller dispersal strengths, nonlocal bistable equations can
exhibit pinning and discontinuous ground states; see
\cite{BesseCapelFayeFouilhe2023}.
\end{remark}

\begin{lemma}\label{lem:semiflow}
Assume (K0) and \textup{(F1)--(F2)}.  For every $u_0\in L^\infty(\R^d)$ with $0\le u_0\le1$ almost everywhere, problem \eqref{eq:main-equation} with initial datum $u_0$ has a unique global solution
\begin{equation}
 u\in C^1([0,\infty);L^\infty(\R^d)).
\end{equation}
The solution satisfies $0\le u(t,\cdot)\le1$ for every $t\ge0$, and the solution map preserves the almost-everywhere order of initial data. Define $S_tu_0=u(t,\cdot)$ and $L_f=\max_{s\in[0,1]}|f'(s)|$.  For every $v_0\in L^\infty(\R^d)$ with $0\le v_0\le1$ almost everywhere, we have
\begin{equation}\label{eq:semiflow-dependence}
 \|S_tu_0-S_tv_0\|_\infty\le e^{L_ft}\|u_0-v_0\|_\infty
 \qquad(t\ge0).
\end{equation}

The following comparison statement also holds.  Let $T>0$, and let
$\underline u,\overline u\in C^1([0,T];L^\infty(\R^d))$ satisfy
\begin{equation}\label{eq:sub-super-inequalities}
 \underline u_t\le J*\underline u-\underline u+f(\underline u),
 \qquad
 \overline u_t\ge J*\overline u-\overline u+f(\overline u).
\end{equation}
Thus $\underline u$ and $\overline u$ are, respectively, a bounded subsolution and a bounded supersolution of \eqref{eq:main-equation} on $[0,T]\times\R^d$.  If $\underline u(0,\cdot)\le\overline u(0,\cdot)$, then $\underline u(t,\cdot)\le\overline u(t,\cdot)$ for every $t\in[0,T]$.  The comparison statement applies to any bounded ranges of $\underline u$ and $\overline u$.  The same conclusion holds if they are continuous in time and $C^1$ on each interval of a finite time partition.

If also $u_0\in L^p(\R^d)$ with $1\le p<\infty$, then $u\in C^1([0,\infty);L^p(\R^d))$ and
\begin{equation}
 \|u(t,\cdot)\|_p\le e^{L_ft}\|u_0\|_p\qquad(t\ge0).
\end{equation}
If $\|u_0\|_\infty<\vartheta$, then $u(t,\cdot)\to0$ uniformly.  If $\essinf_{x\in\R^d}u_0(x)>\vartheta$, then $u(t,\cdot)\to1$ uniformly as $t\to\infty$.
\end{lemma}

\begin{proof}
By \textup{(K0)} and \textup{(F1)}, the map
$\mathcal F:L^\infty(\R^d)\to L^\infty(\R^d)$,
$\mathcal F(u)=J*u-u+f(u)$, is Lipschitz on bounded sets:
$\|J*u\|_\infty\le\|u\|_\infty$, and $f$ is $C^1$.
The Banach-space ODE theorem gives local existence and uniqueness.

The comparison argument of Fife and Wang
\cite[Section~2, Comparison Principle]{FifeWang1998} applies because
$J*$ is positive and bounded, and $s\mapsto f(s)+(M-1)s$ is
nondecreasing on any fixed bounded range for sufficiently large $M$.
Thus $z=e^{Mt}u$ satisfies an order-preserving integral equation.
Successive application gives comparison on a finite time partition.
The argument also allows a bounded positive integral operator in
place of $J*$ and a bounded multiplication coefficient depending
on $(t,x)$.

Comparison gives order preservation and, using the constant
solutions $0$ and $1$, invariance of $[0,1]$. The bounded local
solution therefore extends globally.

On $L^p(\R^d)$, $1\le p\le\infty$, write $J*$ for the bounded
operator $v\mapsto J*v$ and $I$ for the identity. The semigroup
generated by $J*-I$ is
\begin{equation}\label{eq:dispersal-semigroup}
 P_t=e^{t(J*-I)}=e^{-t}\sum_{n=0}^\infty\frac{t^n}{n!}J^{*n}*,
\end{equation}
where $J^{*n}$ is the $n$-fold convolution for $n\ge1$ and
$J^{*0}*v=v$. By unit mass, \eqref{eq:dispersal-semigroup} gives
$\|P_t\|_{L^p(\R^d)\to L^p(\R^d)}\le1$.
Variation of constants and Gronwall's inequality then give
\eqref{eq:semiflow-dependence}. For the $L^p(\R^d)$ assertion,
the extension $\widehat f(s)=f(\min\{1,\max\{0,s\}\})$ is globally
$L_f$-Lipschitz and vanishes at zero. Its integral equation is
well posed in $L^p(\R^d)\cap L^\infty(\R^d)$; uniqueness identifies
its solution with the $[0,1]$-valued solution already constructed.
The contraction bound in \eqref{eq:dispersal-semigroup} gives
$\|u(t,\cdot)\|_p\le e^{L_ft}\|u_0\|_p$, and the evolution equation
gives continuous differentiability in $L^p(\R^d)$.

By \eqref{eq:bistable-signs}, solutions of $z'=f(z)$ converge to $0$
from $z(0)\in[0,\vartheta)$ and to $1$ from $z(0)\in(\vartheta,1]$.
Since spatially constant functions satisfy $J*z=z$, comparison with
$z(0)=\|u_0\|_\infty$ or
$z(0)=\essinf_{x\in\R^d}u_0(x)$, respectively, proves the two
uniform convergence assertions.
\end{proof}

\begin{lemma}\label{lem:strong}
Assume (K0) and $J>0$ almost everywhere on $B_\rho$ for some $\rho>0$.  Let $u_0,v_0\in L^\infty(\R^d)$ satisfy $0\le v_0\le u_0\le1$ almost everywhere, and let $u$ and $v$ be the unique global solutions of \eqref{eq:main-equation} with initial data $u_0$ and $v_0$, respectively.  If $u_0\not\equiv v_0$ in $L^\infty(\R^d)$, then $u$ and $v$ satisfy
\begin{equation}
 u(t,x)>v(t,x)
\end{equation}
for every $t>0$ and almost every $x$.  If the initial data are continuous, the inequality holds at every $x\in\R^d$.
\end{lemma}

\begin{proof}
Apply the strict comparison principle of Rodr\'iguez-Bernal and
Sastre-G\'omez \cite[Proposition~9(ii)]{RodriguezBernalSastreGomez2022}
on $L^\infty(\R^d)$, with kernel $(x,y)\mapsto J(x-y)$,
loss term $-u$, and the globally Lipschitz extension $\widehat f$ used in
Lemma~\ref{lem:semiflow}. Any two points of $\R^d$ can be joined by a finite chain of
steps shorter than $\rho$, and the kernel has a representative
positive on $B_\rho$. For continuous initial data, the same result on
bounded continuous functions gives the pointwise assertion, since
$x\mapsto J(x-\cdot)$ is continuous into $L^1(\R^d)$.
\end{proof}

Under the hypotheses of Lemma~\ref{lem:strong}, local positivity gives
\begin{equation}\label{eq:convolution-positivity}
 J^{*n}>0\quad\text{a.e. on }B_{n\rho},\qquad n\ge1.
\end{equation}
For $n\ge2$ and $|z|<n\rho$, the intersection
$B_\rho\cap(z-B_{(n-1)\rho})$ has positive measure, which proves
\eqref{eq:convolution-positivity} by induction. For $a,r>0$ and
a measurable set $A\subset B_a$ with $|A|>0$, every integer $n$
with $n\rho>r+a$ gives
\begin{equation}\label{eq:display-2-16}
 (J^{*n}*\mathbf1_A)(x)=\int_AJ^{*n}(x-y)\,dy>0
 \qquad(x\in\overline B_r).
\end{equation}
This convolution is uniformly continuous by translation continuity
in $L^1(\R^d)$, so its minimum on $\overline B_r$ is positive.

\begin{lemma}\label{lem:steadyregularity}
Assume (K0).  Every stationary solution $U$ of \eqref{eq:main-equation} satisfying $0\le U\le1$ has a bounded uniformly continuous representative and satisfies
\begin{equation}\label{eq:stationary-inverse}
 U=H^{-1}(J*U),
\end{equation}
where $H^{-1}$ denotes the inverse of $H$ on $[0,1]$.  More precisely,
\begin{equation}\label{eq:stationary-translation-bound}
 |U(x+h)-U(x)|
 \le \frac{\|U\|_\infty}{m_f}\,\|J(\cdot+h)-J\|_{L^1(\R^d)}.
\end{equation}
If, in addition, $J\in W^{k,1}(\R^d)$ and $f$ is $C^k$ on a neighborhood of $[0,1]$, for an integer $k\ge1$, then the uniformly continuous representative of $U$ belongs to $C^k(\R^d)\cap W^{k,\infty}(\R^d)$.  In particular, if $f$ is $C^\infty$ on a neighborhood of $[0,1]$, $J\in C^\infty(\R^d)$, and all derivatives of $J$ belong to $L^1(\R^d)$, then $U$ is smooth.
\end{lemma}

\begin{proof}
By \eqref{eq:reaction-monotonicity}, $U=H^{-1}(J*U)$ almost
everywhere. Young's inequality \cite[Theorem~4.15]{Brezis2011}
and the $m_f^{-1}$-Lipschitz bound for $H^{-1}$ give
\eqref{eq:stationary-translation-bound}; translation continuity
in $L^1(\R^d)$ supplies the uniformly continuous representative,
for which \eqref{eq:stationary-inverse} holds everywhere.
If $J\in W^{k,1}(\R^d)$, then
$D^\alpha(J*U)=(D^\alpha J)*U$ is bounded and uniformly continuous
for every multi-index $\alpha$ with $|\alpha|\le k$.
Thus $J*U\in C^k(\R^d)\cap W^{k,\infty}(\R^d)$, and the
$C^k$ inverse function theorem in \eqref{eq:stationary-inverse}
gives the same regularity for $U$.
\end{proof}

\begin{remark}
We use the uniformly continuous representatives from
Lemma~\ref{lem:steadyregularity} in pointwise statements about
stationary solutions taking values in $[0,1]$.
\end{remark}

\begin{lemma}\label{lem:duality}
Assume (K0) and \textup{(F1)--(F2)}.  Define $f^\sharp:\R\to\R$ by $f^\sharp(s)=-f(1-s)$.
Let $u$ be the unique global solution of \eqref{eq:main-equation} with initial datum $u_0\in L^\infty(\R^d)$, $0\le u_0\le1$, and set $v=1-u$.  Then $v$ is the unique global solution of \eqref{eq:main-equation} with $f$ replaced by $f^\sharp$ and initial datum $v_0=1-u_0$, satisfying
\begin{equation}\label{eq:reflected-evolution}
 v_t=J*v-v+f^\sharp(v).
\end{equation}
Moreover, $f^\sharp$ satisfies \textup{(F1)--(F2)} with unstable zero $1-\vartheta$ and the same constant $m_f$.  Each of \textup{(F3)} and \textup{(F4)} is preserved
under the replacement of $f$ by $f^\sharp$.  Its potential is
\begin{equation}\label{eq:reflected-potential}
 F^\sharp(s):=\int_0^s f^\sharp(r)\,dr=F(1-s)-F(1),\qquad F^\sharp(1)=-F(1).
\end{equation}
\end{lemma}

\begin{proof}
By \eqref{eq:kernel-mass}, $J*1=1$. Substitution of $v=1-u$
therefore gives \eqref{eq:reflected-evolution}, and
Lemma~\ref{lem:semiflow} gives uniqueness for the initial datum
$1-u_0$. The identities $f^\sharp(s)=-f(1-s)$ and
$(f^\sharp)'(s)=f'(1-s)$ give the stated zeros, signs and
regularity, while
\begin{equation}\label{eq:reflected-derivative-bound}
 \sup_{0\le s\le1}(f^\sharp)'(s)=\sup_{0\le s\le1}f'(s)<1
\end{equation}
preserves \textup{(F2)} and $m_f$. Finally, integration of
$f^\sharp(s)=-f(1-s)$ gives \eqref{eq:reflected-potential}.
\end{proof}

\subsection{Translation estimates and local compactness}

\begin{lemma}\label{lem:modulus}
Assume (K0), and let $u$ be the unique global solution of \eqref{eq:main-equation} with initial datum $u_0\in L^\infty(\R^d)$ satisfying $0\le u_0\le1$.  Then for every $h\in\R^d$,
\begin{equation}\label{eq:solution-translation-bound}
 \esssup_{x\in\R^d}|u(t,x+h)-u(t,x)|
 \le e^{-m_ft}\esssup_{x\in\R^d}|u_0(x+h)-u_0(x)|
 +\frac{1-e^{-m_ft}}{m_f}\|J(\cdot+h)-J\|_1.
\end{equation}
Consequently,
\begin{equation}\label{eq:asymptotic-translation-modulus}
 \lim_{r\downarrow0}\limsup_{t\to\infty}
 \sup_{|h|\le r}\|u(t,\cdot+h)-u(t,\cdot)\|_\infty=0.
\end{equation}
Under (K1),
\begin{equation}\label{eq:display-2-24}
 \|J(\cdot+h)-J\|_1\le {\|\nabla J\|_{L^1(\R^d;\R^d)}}|h|.
\end{equation}
\end{lemma}

\begin{proof}
For $h\in\R^d$, write $(\tau_hv)(x)=v(x+h)$ and set $w_h=\tau_hu-u$.  Translation invariance of convolution gives
\begin{equation}
 (w_h)_t+H(\tau_hu)-H(u)=(\tau_hJ-J)*u
\end{equation}
almost everywhere.  Define
\begin{equation}\label{eq:display-2-26}
 a_h(t,x)=
 \begin{cases}
 \displaystyle\frac{H(u(t,x+h))-H(u(t,x))}{u(t,x+h)-u(t,x)},& w_h(t,x)\ne0,\\[1.2ex]
 H'(u(t,x)),& w_h(t,x)=0.
 \end{cases}
\end{equation}
By \textup{(F2)} and \eqref{eq:display-2-26}, $a_h\ge m_f$.
Integrating a jointly measurable representative of $u_t$ from $u_0$
and applying Fubini gives a representative of $u$ that is absolutely
continuous in time for almost every $x$ on bounded time intervals.
For each fixed $h$, $w_h$ therefore satisfies, for almost every $x$
and almost every time,
\begin{equation}\label{eq:translation-difference-equation}
 (w_h)_t+a_hw_h=((\tau_hJ-J)*u)(x).
\end{equation}
Applying the variation-of-constants formula to \eqref{eq:translation-difference-equation} and using the lower bound on the coefficient gives
\begin{equation}\label{eq:translation-integrating-factor}
 |w_h(t,x)|
 \le e^{-m_ft}|w_h(0,x)|
 +\int_0^t e^{-m_f(t-s)}|((\tau_hJ-J)*u(s,\cdot))(x)|\,ds.
\end{equation}
Since $0\le u\le1$, Young's inequality yields
$\|(\tau_hJ-J)*u(s,\cdot)\|_\infty\le\|\tau_hJ-J\|_1$.
Taking the essential supremum in \eqref{eq:translation-integrating-factor} proves \eqref{eq:solution-translation-bound}.  Taking the supremum over $|h|\le r$ in \eqref{eq:solution-translation-bound}, using $\|\tau_hu_0-u_0\|_\infty\le1$, and then letting $t\to\infty$ and $r\downarrow0$ proves \eqref{eq:asymptotic-translation-modulus}.  Under (K1), \eqref{eq:display-2-24} follows from the fundamental theorem of calculus in $L^1(\R^d)$.
\end{proof}

\begin{lemma}\label{lem:asymcompact}
Assume (K0), let $u$ be the unique global solution of \eqref{eq:main-equation} with initial datum $u_0\in L^\infty(\R^d)$ satisfying $0\le u_0\le1$, and let $t_n\to\infty$.  Then a subsequence of $u(t_n,\cdot)$ converges in $L^\infty(B_r)$ for every $r>0$ to a function with a bounded uniformly continuous representative.
\end{lemma}

\begin{proof}
Set $u_n=u(t_n,\cdot)$.  The asymptotic estimate \eqref{eq:asymptotic-translation-modulus} in Lemma~\ref{lem:modulus} gives
\begin{equation}\label{eq:compactness-translation-hypothesis}
 \lim_{r\downarrow0}\limsup_{n\to\infty}
 \sup_{|h|\le r}\|\tau_hu_n-u_n\|_\infty=0.
\end{equation}
Choose $\rho\in C_c^\infty(B_1)$ with $\rho\ge0$ and $\int_{\R^d}\rho(x)\,dx=1$, and put $\rho_r(x)=r^{-d}\rho(x/r)$.  Then
\begin{equation}\label{eq:mollifier-approximation}
 \|u_n-\rho_r*u_n\|_\infty
 \le \sup_{|h|\le r}\|\tau_hu_n-u_n\|_\infty.
\end{equation}
For fixed $r>0$, Young's inequality gives
\begin{equation}\label{eq:mollifier-gradient-bound}
 \|\rho_r*u_n\|_\infty\le1,
 \qquad
 \|\nabla(\rho_r*u_n)\|_\infty\le\|\nabla\rho_r\|_1,
\end{equation}
uniformly in $n$.

Fix a closed cube $Q\subset\R^d$ and let $\varepsilon>0$.  By \eqref{eq:compactness-translation-hypothesis} and \eqref{eq:mollifier-approximation}, choose $r>0$ and then $N$ so that
\begin{equation}\label{eq:display-2-32}
 \|u_n-\rho_r*u_n\|_\infty<\varepsilon
 \qquad(n\ge N).
\end{equation}
By \eqref{eq:mollifier-gradient-bound} and Arzel\`a--Ascoli,
$\{\rho_r*u_n|_Q:n\ge N\}$ is relatively compact in $C(Q)$.
Since $\varepsilon$ is arbitrary, \eqref{eq:display-2-32} makes
$\{u_n|_Q:n\ge1\}$ totally bounded in $L^\infty(Q)$; finitely many
terms with $n<N$ do not affect this conclusion. Completeness then
gives relative compactness. A diagonal extraction over compact
cubes exhausting $\R^d$ yields a subsequence, still indexed by $n$,
and a bounded function $U$ such that
\begin{equation}\label{eq:local-subsequence-limit}
 \|u_n-U\|_{L^\infty(Q)}\to0
\end{equation}
for each cube in the exhaustion, and hence for every bounded ball.  

For fixed $h$, use \eqref{eq:local-subsequence-limit} on a cube containing both $Q$ and $Q+h$.  Passing to the limit in \eqref{eq:solution-translation-bound}, whose initial-data term tends to zero because $t_n\to\infty$, and then exhausting $\R^d$ by cubes gives
\begin{equation}\label{eq:display-2-34}
 \|\tau_hU-U\|_\infty
 \le \frac1{m_f}\|\tau_hJ-J\|_1.
\end{equation}
By translation continuity in $L^1(\R^d)$, \eqref{eq:display-2-34} gives
\begin{equation}\label{eq:display-2-35}
 \|U-\rho_r*U\|_\infty
 \le m_f^{-1}\sup_{|h|\le r}\|\tau_hJ-J\|_1\to0.
\end{equation}
By \eqref{eq:display-2-35}, the uniformly continuous functions
$\rho_r*U$ are Cauchy in $L^\infty(\R^d)$. Their differences have
equal supremum and essential supremum norms, so their uniform
limit is a bounded uniformly continuous representative of $U$.
\end{proof}

\subsection{Extinction for small initial mass and energy estimates}

Lim \cite[Theorem~4.10(ii)]{Lim2017} proves extinction for initial data with sufficiently small $L^1(\R^d)$ mass for radial kernels under $f(s)\le\gamma s$ on $[0,1]$ for some $\gamma<1$.  A related comparison argument for initial data supported on a sufficiently short interval is given in \cite[Proposition~2.1]{BesseCapelFayeFouilhe2023}.  The next lemma gives an explicit sufficient bound on $\|u_0\|_1$ under (K0) and $J\in L^\infty(\R^d)$; for $u_0=\mathbf1_{B_R}$ this yields the lower bound on $R_{\rm ext}$ in \eqref{eq:main-extinction-bounds}.

\begin{lemma}\label{lem:smallseed}
Assume (K0), $J\in L^\infty(\R^d)$, and \textup{(F1)--(F2)}.  Set
\begin{equation}\label{eq:small-mass-constants}
 T_f=\frac1{m_f}\log\frac4\vartheta,
 \qquad
 \delta_f(J)=\frac{1}{\|J\|_\infty(e^{T_f}-1)}.
\end{equation}
For every initial datum $u_0$ satisfying
\begin{equation}
 0\le u_0\le1,\qquad u_0\in L^1(\R^d),\qquad \|u_0\|_1<\delta_f(J),
\end{equation}
the unique global solution $u$ of \eqref{eq:main-equation} with initial datum $u_0$ satisfies
\begin{equation}
 \|u(t,\cdot)\|_\infty\to0.
\end{equation}
In particular, the unique global solution of \eqref{eq:main-equation} with initial datum $u_0=\mathbf1_A$ becomes extinct whenever $A\subset\R^d$ is measurable and $|A|<\delta_f(J)$.
\end{lemma}

\begin{proof}
Since $f(0)=0$ and $f'(s)\le\kappa_f=1-m_f$ on $[0,1]$,
\begin{equation}\label{eq:reaction-linear-bound}
 f(s)=\int_0^s f'(r)\,dr\le(1-m_f)s
 \qquad(0\le s\le1).
\end{equation}
Substitution of \eqref{eq:reaction-linear-bound} into the evolution equation gives
\begin{equation}\label{eq:display-2-40}
 u_t\le J*u-m_fu.
\end{equation}
Let $z\in C^1([0,\infty);L^\infty(\R^d))$ solve
$z_t=J*z-m_fz$ with $z(0,\cdot)=u_0$. Linear comparison from
Lemma~\ref{lem:semiflow}, applied to \eqref{eq:display-2-40}, gives
$u\le z$, hence
\begin{equation}\label{eq:small-mass-semigroup-comparison}
 u(t,\cdot)\le e^{-m_ft}e^{tJ*}u_0.
\end{equation}
Since $e^{tJ*}=e^tP_t$, the series in \eqref{eq:dispersal-semigroup}, together with $0\le u_0\le1$ and $\|J^{*n}\|_\infty\le\|J\|_\infty$ for $n\ge1$, gives
\begin{equation}\label{eq:small-mass-semigroup-bound}
 \|e^{tJ*}u_0\|_\infty
 \le 1+\|u_0\|_1\|J\|_\infty(e^t-1).
\end{equation}
By the definition \eqref{eq:small-mass-constants} of $T_f$, $e^{-m_fT_f}=\vartheta/4$.  If $\|u_0\|_1<\delta_f(J)$, the comparison \eqref{eq:small-mass-semigroup-comparison} together with \eqref{eq:small-mass-semigroup-bound} gives
\begin{equation}\label{eq:small-mass-time-bound}
 \|u(T_f,\cdot)\|_\infty
 \le \frac\vartheta4
 \left[1+\|u_0\|_1\|J\|_\infty(e^{T_f}-1)\right]
 <\frac\vartheta2<\vartheta.
\end{equation}
Since \eqref{eq:small-mass-time-bound} places
$\|u(T_f,\cdot)\|_\infty$ below $\vartheta$, Lemma~\ref{lem:semiflow}
applied from $T_f$ gives extinction.
\end{proof}

\paragraph{Energy dissipation.}
For an even kernel $J$ satisfying (K0) and $v\in L^2(\R^d)\cap L^\infty(\R^d)$, define the energy by
\begin{equation}\label{eq:energy-definition}
 \mathcal E[v]=\frac14\iint_{\R^d\times\R^d}
 J(x-y)(v(x)-v(y))^2\,dx\,dy-\int_{\R^d}F(v(x))\,dx.
\end{equation}

For even $J$, the first variation of \eqref{eq:energy-definition} at $v$ in a direction $h\in L^2(\R^d)\cap L^\infty(\R^d)$ is $\langle v-J*v-f(v),h\rangle_{L^2(\R^d)}$.  Thus \eqref{eq:main-equation} is the $L^2(\R^d)$ gradient flow of $\mathcal E$; see \cite{BatesFifeRenWang1997,ZhangLiYang2022,FifeWang1998}.  For bounded measurable data $0\le u_0\le1$ in $L^2(\R^d)$, let $u$ be the unique global solution of \eqref{eq:main-equation} with initial datum $u_0$.  Then
\begin{equation}\label{eq:energy-dissipation}
 \mathcal E[u(t,\cdot)]-\mathcal E[u(s,\cdot)]
 =-\int_s^t\|u_\tau(\tau,\cdot)\|_2^2\,d\tau,
 \qquad 0\le s\le t.
\end{equation}
By Lemma~\ref{lem:semiflow},
$u\in C^1([0,\infty);L^2(\R^d)\cap L^\infty(\R^d))$.
Since $F(0)=f(0)=0$ and $f'$ is bounded on $[0,1]$,
$|F(s)|\le Cs^2$ and $|f(s)|\le Cs$ there. Taylor's formula,
with remainder bounded by
$C\|u(t+h,\cdot)-u(t,\cdot)\|_2^2=O(h^2)$, therefore gives
$\frac d{dt}\int_{\R^d}F(u(t,x))\,dx
=\langle f(u(t,\cdot)),u_t(t,\cdot)\rangle_{L^2(\R^d)}$.
For the quadratic term, symmetry and unit mass give, with
$u=u(t,\cdot)$,
\begin{equation}\label{eq:energy-quadratic-form}
 \frac14\iint_{\R^d\times\R^d}J(x-y)(u(x)-u(y))^2\,dx\,dy
 =\frac12\langle u-J*u,u\rangle_{L^2(\R^d)}.
\end{equation}
Differentiating \eqref{eq:energy-quadratic-form}, subtracting
the potential derivative, and using \eqref{eq:main-equation} gives
\begin{equation}\label{eq:energy-differential}
 \frac d{dt}\mathcal E[u(t,\cdot)]
 =\langle u(t,\cdot)-J*u(t,\cdot)-f(u(t,\cdot)),u_t(t,\cdot)\rangle_{L^2(\R^d)}
 =-\|u_t(t,\cdot)\|_2^2,
\end{equation}
and integration proves \eqref{eq:energy-dissipation}.

If $\mathcal E[u_0]<0$, \eqref{eq:energy-dissipation} excludes
extinction: otherwise $\|u(T,\cdot)\|_\infty<\vartheta$ for some
$T$, so $F\le0$ on its range and both terms of
\eqref{eq:energy-definition} are nonnegative, contradicting
$\mathcal E[u(T,\cdot)]\le\mathcal E[u_0]<0$.

\begin{lemma}\label{lem:ballenergy}
Assume (K0), $J(z)=J(-z)$ almost everywhere, $M_1(J)<\infty$, and $F(1)>0$, where
\begin{equation}
 M_1(J)=\int_{\R^d}|z|J(z)\,dz,
 \qquad |B_1|=\omega_d.
\end{equation}
For $u_0=\mathbf1_{B_R}$,
\begin{equation}\label{eq:ball-energy-bound}
 \mathcal E[u_0]
 \le \omega_dR^{d-1}
 \left(\frac{dM_1(J)}4-F(1)R\right).
\end{equation}
Consequently
\begin{equation}\label{eq:negative-ball-energy}
 R>\frac{dM_1(J)}{4F(1)}
 \quad\Longrightarrow\quad
 \mathcal E[\mathbf1_{B_R}]<0,
\end{equation}
and, if $\mathcal E[u_0]<0$, the unique global solution $u$ of \eqref{eq:main-equation} with initial datum $u_0=\mathbf1_{B_R}$ cannot become extinct.
\end{lemma}

\begin{proof}
For measurable sets $A$ and $B$, write $A\triangle B=(A\setminus B)\cup(B\setminus A)$ for their symmetric difference and $A+z=\{x+z:x\in A\}$ for a translate.  For $u_0=\mathbf1_{B_R}$, Tonelli's theorem and the change of variables $z=x-y$ give
\begin{equation}\label{eq:ball-symmetric-difference-energy}
 \iint_{\R^d\times\R^d}J(x-y)\bigl(\mathbf1_{B_R}(x)-\mathbf1_{B_R}(y)\bigr)^2\,dx\,dy
 =\int_{\R^d}J(z)|B_R\triangle(B_R+z)|\,dz.
\end{equation}
The translation estimate for functions of bounded variation
\cite[Propositions~8 and~11]{Galerne2011}, applied to
$\mathbf1_{B_R}$, gives
\begin{equation}\label{eq:ball-symmetric-difference-bound}
 |B_R\triangle(B_R+z)|\le d\omega_dR^{d-1}|z|,
 \qquad z\in\R^d,
\end{equation}
since the perimeter of $B_R$ is $d\omega_dR^{d-1}$.
The potential term is $-F(1)|B_R|=-F(1)\omega_dR^d$.
Combining this identity with \eqref{eq:ball-symmetric-difference-energy}
and \eqref{eq:ball-symmetric-difference-bound} gives
\eqref{eq:ball-energy-bound}.

Since $\mathbf1_{B_R}\in L^2(\R^d)$, \eqref{eq:energy-dissipation}
excludes extinction when $\mathcal E[\mathbf1_{B_R}]<0$.
\end{proof}

\subsection{Auxiliary extension and planar fronts}

Unless stated otherwise, solutions of \eqref{eq:main-equation}
in this subsection take values in $[0,1]$.

Under \textup{(F3)}, we extend the reaction to apply Theorem~A
and the radial barriers of Fife and Wang \cite{FifeWang1998}.
In their sign convention, set $\alpha=2\vartheta-1$ and
$g(v)=-2f((1+v)/2)$ on $[-1,1]$. For $a\in\{-1,1\}$, put
$p_a=g'(a)>0$ and $q_a=g''(a)$, and define
\begin{equation}\label{eq:display-3-01}
 g(v)=p_a\int_a^v\exp\!\left(\frac{q_a}{p_a}(r-a)\right)\,dr
\end{equation}
for $v<-1$ when $a=-1$ and $v>1$ when $a=1$.
By \eqref{eq:display-3-01}, the value and first two derivatives
match at each endpoint, and $g'>0$ outside $[-1,1]$. Thus $g$ is
$C^2$ with exactly the zeros $-1,\alpha,1$.
On $[-1,1]$, \textup{(F2)} gives
$1+g'(v)=1-f'((1+v)/2)\ge m_f$; outside, $g'>0$.
Hence $v\mapsto v+g(v)$ is strictly increasing on $\R$.
Setting $v=2u-1$ gives $v_t=J*v-v-g(v)$ and
$v(0,x)=2u_0(x)-1$. Invariance of $[-1,1]$ and agreement of the
reactions there preserve this equation and its front profiles.
Comparison for auxiliary sub- and supersolutions outside $[-1,1]$
uses the extended reaction.

\begin{lemma}\label{lem:projection}
Assume (K1), fix $e\in\Sph$, and let $K_e$ be defined before \eqref{eq:radial-marginal}.  Then $K_e\in W^{2,1}(\R)$ has a $C^1$ representative, which we continue to denote by $K_e$, with $K_e\ge0$, $K_e(s)=K_e(-s)$ for $s\in\R$, and
\begin{equation}
 \int_\R K_e(s)\,ds=1,
 \qquad
 \int_\R |s|K_e(s)\,ds<\infty,
 \qquad
 K_e'\in L^1(\R).
\end{equation}
Moreover, if $\rho$ is as in (K1), then
\begin{equation}\label{eq:display-3-04}
 K_e(s)>0\qquad\text{for a.e. }s\in(-\rho,\rho).
\end{equation}
\end{lemma}

\begin{proof}
By Fubini, $K_e\ge0$ and $\int_\R K_e=1$.
Evenness of $J$ and the change $z_\perp\mapsto-z_\perp$ give,
for almost every $s$,
$K_e(-s)=\int_{e^\perp}J(-se-z_\perp)\,dz_\perp=K_e(s)$,
while the first-moment assumption in (K1) yields
\begin{equation}
 \int_\R |s|K_e(s)\,ds
 \le\int_{\R^d}|z|J(z)\,dz<\infty.
\end{equation}
Write $\partial_e=e\cdot\nabla$ and $\partial_{ee}=\partial_e^2$.
Testing the weak derivative identities for $J$ and using Fubini gives,
for almost every $s\in\R$,
\begin{equation}\label{eq:marginal-weak-derivatives}
 K_e'(s)=\int_{e^\perp}\partial_eJ(se+z_\perp)\,dz_\perp,
 \qquad
 K_e''(s)=\int_{e^\perp}\partial_{ee}J(se+z_\perp)\,dz_\perp.
\end{equation}
By \eqref{eq:marginal-weak-derivatives}, $K_e\in W^{2,1}(\R)$
and $\|K_e'\|_1\le\|\partial_eJ\|_1$.
The one-dimensional Sobolev embedding gives a $C^1$ representative,
for which evenness holds everywhere. Finally, Fubini and positivity
of $J$ on $B_\rho$ give \eqref{eq:display-3-04}, as required in
Theorem~A.
\end{proof}

\paragraph{Planar fronts obtained from Theorem~A.}
Assume \textup{(K1)}, \textup{(F3)}, and fix $e\in\Sph$.
By Lemma~\ref{lem:projection}, $K_e$ satisfies \eqref{1.10};
Theorem~A therefore gives a unique speed $c_e\in\R$ and, up to
translation, a unique decreasing profile $\Phi_e$ satisfying
\begin{equation}\label{eq:directional-front}
 -c_e\Phi_e'=K_e*\Phi_e-\Phi_e+f(\Phi_e),\qquad
 \Phi_e(-\infty)=1,\qquad
 \Phi_e(+\infty)=0,
\end{equation}
with $0<\Phi_e<1$ and $\Phi_e'\in L^1(\R)\cap L^2(\R)$.  By the definition of $K_e$, $u_e(t,x):=\Phi_e(x\cdot e-c_et)$ satisfies \eqref{eq:main-equation} with initial datum $u_e(0,x)=\Phi_e(x\cdot e)$, since $J*u_e(t,x)=(K_e*\Phi_e)(x\cdot e-c_et)$.  The speed identity in Theorem~A gives
\begin{equation}\label{eq:directional-speed-identity}
 c_e\int_\R|\Phi_e'(\xi)|^2\,d\xi=F(1).
\end{equation}
Under \textup{(KR)}, radiality makes $K_e$ independent of $e$.
Denote the common speed by $c_*$, so $c_e=c_*$ for every direction.
For $F(1)>0$, this is the speed
introduced before Theorem~B.

\subsection{A spreading criterion}

Assume (KR) and $F(1)>0$, and let $u$ be the unique global solution of \eqref{eq:main-equation} with initial datum $u_0\in L^\infty(\R^d)$ satisfying $0\le u_0\le1$.  For every $\eta\in(0,1-\vartheta)$, there is $L=L(\eta)>0$ such that
\begin{equation}\label{eq:measurable-seed-spreading}
 u_0(x)\ge\vartheta+\eta\quad (|x|\le L)
 \quad\Longrightarrow\quad
 u(t,\cdot)\to1\quad\text{locally uniformly in }\R^d.
\end{equation}
To prove \eqref{eq:measurable-seed-spreading} under the $C^1$ assumption \textup{(F1)}, fix $\varepsilon>0$ and choose a polynomial $q$ with $\|q-f'\|_{L^\infty(0,1)}\le\varepsilon/2$.  Define
\[
 p(s)=\int_0^s\left(q(r)-\int_0^1q(a)\,da\right)dr,
 \qquad h(s)=p(s)-3\varepsilon s(1-s).
\]
Since $\int_0^1f'(s)\,ds=0$, we have $p(0)=p(1)=0$ and $\|p'-f'\|_{L^\infty(0,1)}\le\varepsilon$.  Integrating this derivative estimate from either endpoint gives
\[
 |p(s)-f(s)|\le\varepsilon\min\{s,1-s\}
 \le2\varepsilon s(1-s),\qquad 0\le s\le1.
\]
Consequently,
\[
 h(s)\le f(s)-\varepsilon s(1-s),\qquad
 \|h'-f'\|_{L^\infty(0,1)}\le4\varepsilon,\qquad
 \int_0^1h(s)\,ds\ge F(1)-\frac56\varepsilon.
\]
For small $\varepsilon$, the derivative signs in
\eqref{eq:bistable-signs} persist on disjoint neighborhoods of
$0$, $\vartheta$ and $1$,
and uniform convergence preserves the signs of $f$ on the compact
complement. Thus $h$ has exactly one zero $\vartheta_h\in(0,1)$,
with $h'(\vartheta_h)>0$ and $h'(0),h'(1)<0$.
Since $h(\vartheta)<0$ and $h$ is increasing near $\vartheta$,
$\vartheta_h>\vartheta$ and $\vartheta_h\to\vartheta$ as
$\varepsilon\downarrow0$.
Choose $\varepsilon$ with $4\varepsilon<m_f$,
$5\varepsilon/6<F(1)$ and $\vartheta_h<\vartheta+\eta/2$.
Then $h$ satisfies \textup{(F1)--(F3)} with positive integral.

By Theorem~B\textup{(i)} for $h$, height $\vartheta+\eta$ and
speed $c_h/2$, where $c_h>0$ is the wave speed for $(h,J)$, there
is $L=L(\eta)$ such that the solution $w$ with reaction $h$ and
initial datum $u_0$ converges locally uniformly to $1$.
Since $0\le w\le1$ and $h\le f$ on $[0,1]$, we have
$w_t\le J*w-w+f(w)$. Lemma~\ref{lem:semiflow} gives $w\le u\le1$,
proving \eqref{eq:measurable-seed-spreading}.

In particular, \eqref{eq:measurable-seed-spreading} shows that $u^R$ spreads for all sufficiently large $R$.  A one-dimensional propagation criterion for $C^1$ bistable reactions admitting a positive-speed traveling wave is given by Alfaro, Ducrot and Kang \cite{AlfaroDucrotKang2023} under their assumptions on the dispersal kernel.

For $b\in(\vartheta,1)$, set $L_b=L(b-\vartheta)$ in
\eqref{eq:measurable-seed-spreading}; thus $u_0\ge b\mathbf1_{B_{L_b}}$
implies spreading. By uniqueness and translation invariance, the criterion
applies from any initial time and on any translated ball.

\subsection{Reflection and spatial bounds}

For this subsection, assume \textup{(KM)}, and let
$u$ solve \eqref{eq:main-equation} with initial datum \eqref{cs:eq:data}.
Let $\widetilde u$ be the unique global solution of
\eqref{eq:main-equation} with initial datum
$\widetilde u_0\in L^\infty(\R^d)$ satisfying $0\le\widetilde u_0\le1$.
For $1\le p\le\infty$ with $u_0-\widetilde u_0\in L^p(\R^d)$,
\eqref{eq:dispersal-semigroup} and Gronwall's inequality give
\begin{equation}\label{cs:eq:flow-dependence}
 \|u(t,\cdot)-\widetilde u(t,\cdot)\|_p
 \le e^{L_ft}\|u_0-\widetilde u_0\|_p,
 \qquad L_f=\max_{[0,1]}|f'|,\quad 1\le p\le\infty.
\end{equation}
The comparison principle in Lemma~\ref{lem:semiflow} also applies to
ordered reactions on $[0,1]$, since a solution for the smaller reaction
is a subsolution for the larger one.  For any datum satisfying \eqref{cs:eq:data} with $u_0\not\equiv0$, write
$h_0(e)=\sup_{x\in\Sigma_0}x\cdot e$, $e\in\Sph$.
We first prove the reflection estimate
\begin{equation}\label{cs:eq:intro-reflection}
 \esssup_{|x|>r+2R_0}u(t,x)\le\essinf_{|y|<r}u(t,y),
 \qquad r>0,\quad t\ge0.
\end{equation}
\begin{lemma}\label{cs:lem:reflection}
Let \textup{(F1)--(F2)}, \textup{(KM)}, and \eqref{cs:eq:data} hold.
Then \eqref{cs:eq:intro-reflection} holds for every $t\ge0$ and $r>0$.
Consequently, whenever $u(t,\cdot)\in L^p(\R^d)$,
\begin{equation}\label{cs:eq:reflection-integral}
 \esssup_{|x|>r+2R_0}u(t,x)
 \le |B_r|^{-1/p}\|u(t,\cdot)\|_p,\qquad 1\le p<\infty .
\end{equation}
\end{lemma}

\begin{proof}
For continuous $u_0$ supported in $\overline B_{R_0}$,
the construction in Lemma~\ref{lem:semiflow} in the space of bounded
uniformly continuous functions gives continuity of $u(t,\cdot)$.
Fix a unit vector $e$ and $\lambda\ge R_0$, and set
\begin{equation}\label{cs:eq:reflection-def}
 D_{e,\lambda}=\{x:x\cdot e>\lambda\},\qquad
 x^{e,\lambda}=x-2(x\cdot e-\lambda)e,\qquad
 w(t,x)=u(t,x^{e,\lambda})-u(t,x).
\end{equation}
For $x,y\in D_{e,\lambda}$,
\begin{equation}\label{cs:eq:reflection-kernel}
 |x-y|^2\le |x-y^{e,\lambda}|^2,\qquad
 K_{e,\lambda}(x,y):=J(x-y)-J(x-y^{e,\lambda})\ge0 .
\end{equation}
By splitting convolution into the two half-spaces and changing
variables by reflection in one of them, \eqref{eq:main-equation} gives
\begin{equation}\label{cs:eq:reflection-evolution}
 \partial_tw(t,x)
 =\int_{D_{e,\lambda}}K_{e,\lambda}(x,y)w(t,y)\,dy
       -a(t,x)w(t,x),
 \qquad
 a(t,x)=\int_0^1H'\bigl(u(t,x)+s w(t,x)\bigr)\,ds .
\end{equation}
By \eqref{cs:eq:reflection-kernel}, the integral operator in
\eqref{cs:eq:reflection-evolution} is positive on
$L^\infty(D_{e,\lambda})$ with norm at most one.
The coefficient $a$ is bounded by continuity of $H'$ on $[0,1]$,
and $u_0=0$ on $D_{e,\lambda}$ gives $w(0,\cdot)\ge0$.
The comparison argument of Fife and Wang
\cite[Section~2, Comparison Principle]{FifeWang1998}, used in
Lemma~\ref{lem:semiflow}, yields
\begin{equation}\label{cs:eq:half-space-order}
 u(t,x)\le u(t,x^{e,\lambda})
 \qquad(x\in D_{e,\lambda},\ t\ge0).
\end{equation}

Now fix $x,y\in\R^d$ with $|x|>|y|+2R_0$.  For
\begin{equation}\label{cs:eq:bisector}
 e=\frac{x-y}{|x-y|},\qquad
 \lambda=\frac{|x|^2-|y|^2}{2|x-y|},
\end{equation}
the reflection sends $x\in D_{e,\lambda}$ to $y$, and
\begin{equation}\label{cs:eq:bisector-bound}
 \lambda\ge\frac{|x|-|y|}{2}>R_0
\end{equation}
because $|x-y|\le |x|+|y|$.
Thus \eqref{cs:eq:half-space-order} gives $u(t,x)\le u(t,y)$.
Taking the supremum over $|x|>r+2R_0$ and the infimum over $|y|<r$
proves \eqref{cs:eq:intro-reflection} for continuous data.

For a datum satisfying \eqref{cs:eq:data}, choose a nonnegative smooth
function $\varrho$ supported in $B_1$ with integral one and set
$\varrho_\varepsilon(x)=\varepsilon^{-d}\varrho(x/\varepsilon)$ and
$u_{0,\varepsilon}=\varrho_\varepsilon*u_0$.
Let $u_\varepsilon$ solve \eqref{eq:main-equation} with datum
$u_{0,\varepsilon}$.
Then $0\le u_{0,\varepsilon}\le1$,
$\supp u_{0,\varepsilon}\subset\overline{B}_{R_0+\varepsilon}$,
and $u_{0,\varepsilon}\to u_0$ in $L^1$ as
$\varepsilon\downarrow0$.
For each fixed $t$, \eqref{cs:eq:flow-dependence} gives
$u_\varepsilon(t,\cdot)\to u(t,\cdot)$ in $L^1$ and, along a
subsequence, almost everywhere. The continuous-data result gives
$u_\varepsilon(t,x)\le u_\varepsilon(t,y)$ when
$|x|>|y|+2(R_0+\varepsilon)$.
Passing to the subsequence yields $u(t,x)\le u(t,y)$ for almost
every pair with $|x|>|y|+2R_0$.
Its restriction to $\{|x|>r+2R_0\}\times B_r$ proves
\eqref{cs:eq:intro-reflection}.
Finally,
$\essinf_{B_r}u(t,\cdot)\le |B_r|^{-1/p}\|u(t,\cdot)\|_p$,
which gives \eqref{cs:eq:reflection-integral}.
\end{proof}

\section{Decaying stationary solutions}\label{cs:sec:stationary}

Throughout this section, assume \textup{(F1)--(F2)}, \textup{(KM)}
and $F(1)>0$. We use $\mathcal S(J,f)$ and $\mathcal S_0(J,f)$
from \eqref{cs:eq:stationary-sets}.

\subsection{A uniform bound for stationary solutions}

Set
\begin{equation}\label{cs:eq:profile-height}
 a(J,f)=\sup_{U\in\mathcal S_0(J,f)}\|U\|_\infty,\qquad
 \bar a(J,f)=\sup_{U\in\mathcal S(J,f),\ U\not\equiv1}\|U\|_\infty.
\end{equation}

\begin{lemma}\label{cs:lem:uniform-height}
Under \textup{(F1)--(F2)} and \textup{(KM)}, the numbers in
\eqref{cs:eq:profile-height} satisfy
$a(J,f)\le\bar a(J,f)<1$.
\end{lemma}

\begin{proof}
By \eqref{eq:stationary-translation-bound}, every
$U\in\mathcal S(J,f)$ satisfies
\begin{equation}\label{cs:eq:uniform-stationary-modulus}
 \|\tau_hU-U\|_\infty
 \le m_f^{-1}\|\tau_hJ-J\|_1,\qquad
 (\tau_hv)(x)=v(x+h).
\end{equation}
If $\bar a(J,f)=1$, choose
$U_n\in\mathcal S(J,f)$ with $U_n\not\equiv1$ and
$\|U_n\|_\infty\to1$, and take $x_n$ such that
$U_n(x_n)>\|U_n\|_\infty-1/n$.
Set $W_n(x)=U_n(x+x_n)$.
By \eqref{cs:eq:uniform-stationary-modulus} and Arzel\`a--Ascoli, a subsequence
converges locally uniformly to a continuous $W$ with $0\le W\le1$ and
$W(0)=1$.  Dominated convergence in
$\int J(x-y)W_n(y)\,dy$ then gives $H(W)=J*W$ for each $x$.

At any point $x$ where $W(x)=1$, the stationary equation gives
$\int J(x-y)(1-W(y))\,dy=0$.
Since $J>0$ almost everywhere on $B_\rho$ and $W$ is continuous,
$W=1$ on $B_\rho(x)$.
Thus $\{W=1\}$ is open and closed, and $W(0)=1$ gives
$W\equiv1$ on $\R^d$.
For fixed $b\in(\vartheta,1)$, local uniform convergence gives
$W_n\ge b$ on $B_{L_b}$ for large $n$.
By \eqref{eq:measurable-seed-spreading}, the solution from $W_n$ spreads;
by Lemma~\ref{lem:semiflow}, it is stationary.
Hence $W_n\equiv1$, contradicting $U_n\not\equiv1$ and proving
$\bar a(J,f)<1$.
Since $\mathcal S_0(J,f)\subset\mathcal S(J,f)\setminus\{1\}$,
the definitions give $a(J,f)\le\bar a(J,f)$.
\end{proof}

\subsection{Symmetry and bounds for stationary solutions}

By the unit mass of $J$ and \textup{(F1)}, the constant members
of $\mathcal S(J,f)$ are precisely $0$, $\vartheta$, and $1$.
We begin with the inverse function argument of Chmaj and Ren
\cite[Proposition~3.1]{ChmajRen1999} and its consequences for
bounded stationary solutions on $\R^d$.

For the next lemmas, let $g:[0,1]\to\R$ satisfy
\begin{equation}\label{cs:eq:general-reaction}
 g\in C^1([0,1]),\qquad g(0)=g(1)=0,\qquad g'(0)<0,\qquad
 m_g:=1-\max_{[0,1]}g'>0.
\end{equation}
We write
\begin{equation}\label{cs:eq:general-functions}
 H_g(s)=s-g(s),\qquad F_g(s)=\int_0^s g(r)\,dr.
\end{equation}
In particular, $g=f$ satisfies \eqref{cs:eq:general-reaction}.
The kernel continues to satisfy \textup{(KM)}.

\begin{lemma}\label{cs:lem:stationary-basic}
Let $g$ satisfy \eqref{cs:eq:general-reaction}, and let
$U\in L^\infty(\R^d)$ satisfy
\begin{equation}\label{cs:eq:general-stationary}
 0\le U\le1,\qquad H_g(U)=J*U.
\end{equation}
Then $U$ has a representative in
$C^1(\R^d)\cap W^{1,\infty}(\R^d)$.
Either $U\equiv0$, $U\equiv1$, or $0<U<1$ everywhere.
If $U(x)\to\ell$ as $|x|\to\infty$, then $g(\ell)=0$.
\end{lemma}

\begin{proof}
By \eqref{cs:eq:general-reaction}, $H_g'\ge m_g>0$ on $[0,1]$.
The inverse-function argument of Lemma~\ref{lem:steadyregularity}
uses only this bound and $g\in C^1([0,1])$, extended $C^1$ beyond
the endpoints if needed. It therefore gives
$U\in C^1(\R^d)\cap W^{1,\infty}(\R^d)$ and
\begin{equation}\label{cs:eq:stationary-modulus}
 U=H_g^{-1}(J*U),\qquad
 \|\tau_hU-U\|_\infty
 \le m_g^{-1}\|U\|_\infty\|\tau_hJ-J\|_1 .
\end{equation}
If $U(x_0)=0$, then \eqref{cs:eq:general-stationary} gives
$(J*U)(x_0)=0$. Nonnegativity, continuity and positivity of $J$ almost everywhere
on $B_\rho$ imply $U=0$ on $B_\rho(x_0)$.
Thus $\{U=0\}$ is open and closed, and connectedness gives $U\equiv0$.
The same argument applied to $1-U$ shows that $U(x_0)=1$ forces
$U\equiv1$. Finally, if $U(x)\to\ell$ at infinity, dominated
convergence gives $J*U\to\ell$; by
\eqref{cs:eq:general-stationary}, $H_g(\ell)=\ell$ and $g(\ell)=0$.
\end{proof}

\begin{lemma}\label{cs:lem:stationary-regularity}
Under \eqref{cs:eq:reaction-regularity}, \textup{(F1)--(F2)}, and
\textup{(KM)}, every $U\in\mathcal S(J,f)$ belongs to
$C^2(\R^d)\cap W^{2,\infty}(\R^d)$.
\end{lemma}

\begin{proof}
For $U\equiv0$ or $U\equiv1$, the assertion is immediate.
Otherwise Lemmas~\ref{cs:lem:stationary-basic} and
\ref{cs:lem:uniform-height} give
$0<U\le\bar a(J,f)<1$.
Since $f\in C^2((0,1))$, the interior inverse function theorem
applied to $U=H^{-1}(J*U)$ gives $U\in C^2(\R^d)$.
Set $Q=J*U$, and let
$B=\sup_x\|D^2Q(x)\|_{\mathrm{op}}$ and
$L_H=\max_{[0,1]}H'$, where $\|\cdot\|_{\mathrm{op}}$ is the
Euclidean operator norm.
Differentiation of $H(U)=Q$ gives
\begin{equation}\label{cs:eq:stationary-second-derivative}
 \partial_iU=\frac{\partial_iQ}{H'(U)},\qquad
 \partial_{ij}U=
 \frac{\partial_{ij}Q+f''(U)\partial_iU\,\partial_jU}{H'(U)}.
\end{equation}
The kernel assumptions bound $B$ and the first derivatives of $Q$.
If $B=0$, then $Q$ is bounded and affine, hence constant.
For $B>0$, Taylor's inequality at
$y=x-B^{-1}\nabla Q(x)$ gives
$0\le Q(y)\le Q(x)-|\nabla Q(x)|^2/(2B)$.
Hence $|\nabla Q(x)|^2\le2BQ(x)$.
Since $Q=H(U)$, $H(0)=0$, and $m_f\le H'\le L_H$, it follows that
\begin{equation}\label{cs:eq:stationary-gradient-bound}
 |\nabla U|^2
 \le2Bm_f^{-2}H(U)
 \le2BL_Hm_f^{-2}U .
\end{equation}
By \eqref{cs:eq:reaction-regularity}, continuity of $f''$ in $(0,1)$,
and $\bar a(J,f)<1$,
\[
 \sup_{0<s\le\bar a(J,f)}s|f''(s)|<\infty .
\]
Together with \eqref{cs:eq:stationary-gradient-bound}, this bounds
$|f''(U)||\nabla U|^2$; by \eqref{cs:eq:stationary-second-derivative}
and $H'\ge m_f$, the first and second derivatives of $U$ are bounded.

\end{proof}

\begin{lemma}\label{cs:lem:stationary-integrability}
Let $U$ satisfy the assumptions of Lemma~\ref{cs:lem:stationary-basic}
and $U(x)\to0$ as $|x|\to\infty$.
Then
\begin{equation}\label{cs:eq:stationary-integrability}
 U\in L^1(\R^d)\cap H^1(\R^d),\qquad
 \int_{\R^d}|x|^2U(x)\,dx<\infty,\qquad
 \int_{\R^d}g(U(x))\,dx=0.
\end{equation}
Moreover,
\begin{equation}\label{cs:eq:stationary-energy}
 \int_{\R^d}U(x)g(U(x))\,dx
 =\frac12\iint_{\R^d\times\R^d}
 J(x-y)(U(x)-U(y))^2\,dx\,dy.
\end{equation}
For $U\not\equiv0$, the quantity in \eqref{cs:eq:stationary-energy}
is positive.
\end{lemma}

\begin{proof}
Since $g'(0)<0$, choose $a,\delta>0$ such that
$g(s)\le-as$ on $[0,\delta]$.
Take $R$ so large that $U\le\delta$ outside $B_R$.
Writing $s_+=\max\{s,0\}$ for $s\in\R$, set
$q=(1+a)^{-1}$ and $b(x)=(aU(x)+g(U(x)))_+$.
Then $b$ is bounded, nonnegative, and supported in $\overline B_R$.
By \eqref{cs:eq:general-stationary},
$(1+a)U\le J*U+b$.
Iteration gives, for every positive integer $N$,
\[
 U\le q^NJ^{*N}*U+\sum_{n=0}^{N-1}q^{n+1}J^{*n}*b,
\]
where $J^{*n}$ is the $n$-fold convolution of $J$ and
$J^{*0}*b=b$.
Since the first term has $L^\infty$ norm at most $q^N\|U\|_\infty$,
letting $N\to\infty$ and applying Tonelli's theorem gives
\begin{equation}\label{cs:eq:stationary-resolvent}
 U\le\sum_{n=0}^\infty q^{n+1}J^{*n}*b,\qquad
 \|U\|_1\le a^{-1}\|b\|_1.
\end{equation}
In particular, $U\in L^2$ because $0\le U\le1$.

For a nonnegative function $h$, write
$M_2(h)=\int_{\R^d}|x|^2h(x)\,dx$.
Evenness and the finite second moment of $J$ give
$\int xJ(x)\,dx=0$.
Expanding $|x+y|^2$ in the convolution integral then yields
\[
 M_2(J^{*n}*b)=M_2(b)+nM_2(J)\|b\|_1.
\]
Multiplication of \eqref{cs:eq:stationary-resolvent} by $|x|^2$
and summation of the geometric series prove
\begin{equation}\label{cs:eq:stationary-moment}
 M_2(U)\le a^{-1}M_2(b)+a^{-2}M_2(J)\|b\|_1<\infty.
\end{equation}

By Young's inequality, $U\in L^2$ and $J\in W^{1,1}$ give
$J*U\in H^1$. Applying the Sobolev chain rule to
$U=H_g^{-1}(J*U)$, with a Lipschitz extension of $H_g^{-1}$, gives $U\in H^1$ and
\begin{equation}\label{cs:eq:stationary-derivative}
 H_g'(U)\partial_iU=J*\partial_iU
 \qquad(1\le i\le d).
\end{equation}
Here $H_g^{-1}(0)=0$ ensures square integrability after composition.
Since $|g(U)|\le C U$, integration of
\eqref{cs:eq:general-stationary} is justified by
\eqref{cs:eq:stationary-resolvent}; the unit mass of $J$ gives
$\int g(U)=0$.
Multiplying the same equation by $U$ and using the symmetry of $J$
proves \eqref{cs:eq:stationary-energy}.
If the right side vanished, continuity of $U$ and positivity of $J$
near zero would give $U(x)=U(y)$ for $|x-y|<\rho$.
Connectedness and decay at infinity would then give $U\equiv0$.
\end{proof}

\begin{lemma}\label{cs:lem:halfspace-positive}
Fix a unit vector $e$ and $\lambda\in\R$, and use
$D_{e,\lambda}$, $x^{e,\lambda}$, and $K_{e,\lambda}$
from \eqref{cs:eq:reflection-def}--\eqref{cs:eq:reflection-kernel}.
Let $w$ be bounded, continuous, and nonnegative on $D_{e,\lambda}$.
Suppose that $a:D_{e,\lambda}\to(0,\infty)$ satisfies
\begin{equation}\label{cs:eq:halfspace-stationary}
 a(x)w(x)=\int_{D_{e,\lambda}}K_{e,\lambda}(x,y)w(y)\,dy
 \qquad(x\in D_{e,\lambda}).
\end{equation}
Then either $w\equiv0$ or $w>0$ throughout $D_{e,\lambda}$.
\end{lemma}

\begin{proof}
After a translation and an orthogonal change of coordinates,
take $\lambda=0$ and $e=(0,\ldots,0,1)$.
Because $J(x)=j(|x|)$ has integral one and $j$ is continuous
and nonincreasing, we have
$j(0)>0$ and $j(r)\to0$ as $r\to\infty$.
Consequently $b_J=\sup\{r\ge0:j(r)=j(0)\}$ is finite.
On the connected set $D=\{x:x_d>b_J/2\}$, we have
\[
 K_{e,0}(x,x)=j(0)-j(2x_d)>0.
\]
If $w(x)=0$ at a point of $D$, equation
\eqref{cs:eq:halfspace-stationary} and continuity of its nonnegative
integrand force $w$ to vanish in a neighborhood of that point.
Thus $\{w=0\}\cap D$ is open and closed in $D$, hence empty or equal to $D$.

For $b_J=0$, this already proves the assertion.
For $b_J>0$ and $0<x_d\le b_J/2$, choose
$y=(x_1,\ldots,x_{d-1},b_J)$.
By the definition of $b_J$,
\begin{equation}\label{cs:eq:plateau-connection}
 K_{e,0}(x,y)=j(b_J-x_d)-j(b_J+x_d)>0.
\end{equation}
The kernel $K_{e,0}$ is symmetric.
Hence, if $w$ is positive at any point of the strip
$0<x_d\le b_J/2$, continuity and
\eqref{cs:eq:plateau-connection}, used in the equation at $y$,
give $w(y)>0$ at a point of $D$.
The zero set in $D$ is then empty.
Conversely, positivity on $D$ and
\eqref{cs:eq:plateau-connection}, used in the equation at $x$,
give positivity at every point of the strip.
These two implications prove the claim.
\end{proof}

\begin{lemma}\label{cs:lem:stationary-symmetry}
Let $g$ satisfy \eqref{cs:eq:general-reaction}.
Every nonzero solution $U$ of \eqref{cs:eq:general-stationary} with
$U(x)\to0$ as $|x|\to\infty$ has the form
\begin{equation}\label{cs:eq:stationary-radial}
 U(x)=V(|x-x_*|),\qquad
 0<V(r)<1,\qquad V'(r)<0\quad(r>0),
\end{equation}
for a unique point $x_*\in\R^d$.
\end{lemma}

\begin{proof}
By Lemmas~\ref{cs:lem:stationary-basic} and
\ref{cs:lem:stationary-integrability},
$0<U<1$, $U\in C^1\cap L^2$, and $U$ is uniformly continuous.
Fix a unit vector $e$.
For $\lambda\in\R$ and $x\in D_{e,\lambda}$, put
\begin{equation}\label{cs:eq:stationary-reflection}
 w_\lambda(x)=U(x^{e,\lambda})-U(x),\qquad
 a_\lambda(x)=\int_0^1
 H_g'\bigl(U(x)+s w_\lambda(x)\bigr)\,ds.
\end{equation}
Splitting the convolution across the reflecting hyperplane, as in
\eqref{cs:eq:reflection-evolution}, gives
\begin{equation}\label{cs:eq:stationary-reflection-equation}
 a_\lambda(x)w_\lambda(x)
 =\int_{D_{e,\lambda}}K_{e,\lambda}(x,y)w_\lambda(y)\,dy,
 \qquad a_\lambda\ge m_g.
\end{equation}

Since $H_g'(0)=1-g'(0)>1$, choose $\gamma,\delta>0$ with
\begin{equation}\label{cs:eq:stationary-tail-coefficient}
 H_g'(s)\ge1+\gamma\qquad(0\le s\le\delta),
\end{equation}
and take $R$ so large that $U\le\delta$ on $\R^d\setminus B_R$.
For $\lambda\ge R$, if $w_\lambda(x)<0$ then
$U(x^{e,\lambda})<U(x)\le\delta$, so
$a_\lambda(x)\ge1+\gamma$.
Let $\zeta_\lambda=(-w_\lambda)_+$ on $D_{e,\lambda}$ and extend
it by zero to $\R^d$.
By \eqref{cs:eq:stationary-reflection-equation} and
$K_{e,\lambda}\ge0$,
\begin{equation}\label{cs:eq:stationary-negative-part}
 a_\lambda\zeta_\lambda
 \le \int_{D_{e,\lambda}}
       K_{e,\lambda}(\,\cdot\,,y)\zeta_\lambda(y)\,dy
 \le J*\zeta_\lambda
 \qquad\text{on }D_{e,\lambda}.
\end{equation}
At points where $\zeta_\lambda>0$, the coefficient is at least
$1+\gamma$.
Taking the supremum in
\eqref{cs:eq:stationary-negative-part} gives
$(1+\gamma)\|\zeta_\lambda\|_\infty
\le\|\zeta_\lambda\|_\infty$.
It follows that $w_\lambda\ge0$ for every $\lambda\ge R$.

Define
\[
 \lambda_*=\inf\{\lambda\in\R:
  w_\mu\ge0\text{ on }D_{e,\mu}\text{ for every }\mu\ge\lambda\}.
\]
The set is nonempty by the preceding argument. If $\lambda_*=-\infty$,
the reflection order would make $U$ nonincreasing on every line
parallel to $e$; its zero limits at both ends would give $U\equiv0$.
Hence $\lambda_*$ is finite.
Taking $\lambda\downarrow\lambda_*$ and using continuity of $U$
gives $w_{\lambda_*}\ge0$.

Suppose $w_{\lambda_*}\not\equiv0$.
Lemma~\ref{cs:lem:halfspace-positive} gives
$w_{\lambda_*}>0$ on $D_{e,\lambda_*}$.
Increase $R$ so that $R>|\lambda_*|+1$ while retaining
$U\le\delta$ outside $B_R$.
Since $g(0)=g(1)$ and $g'(0)<0$, the mean value theorem gives
$\max_{[0,1]}g'>0$, hence $m_g<1$.
Choose $\eta\in(0,1)$ so small that
\begin{equation}\label{cs:eq:stationary-strip}
 \left|B_R\cap\{x:\lambda_*-\eta<x\cdot e<\lambda_*+\eta\}\right|
 <\frac{\gamma m_g^2}
 {(1+\gamma-m_g)\|J\|_2^2}.
\end{equation}
On the compact set
$\overline B_R\cap\{x:x\cdot e\ge\lambda_*+\eta\}$,
the strict positivity of $w_{\lambda_*}$ gives a positive minimum
whenever this set is nonempty.
Uniform continuity of $U$ therefore gives
$\varepsilon\in(0,\eta)$ such that $w_\lambda\ge0$ on this set
whenever $\lambda_*-\varepsilon\le\lambda\le\lambda_*$.

For such a $\lambda$, write $\zeta=\zeta_\lambda$ and
\[
 E=B_R\cap\{x:\lambda<x\cdot e<\lambda_*+\eta\}.
\]
By $U\in L^2$, we have $\zeta\in L^2$.
Multiplication of \eqref{cs:eq:stationary-negative-part} by $\zeta$,
followed by Young's inequality, gives
\[
 \int_{D_{e,\lambda}}a_\lambda\zeta^2
 \le\int_{\R^d}\zeta(J*\zeta)\le\|\zeta\|_2^2.
\]
At points outside $E$ where $\zeta>0$, we have $|x|\ge R$ and
$0\le U(x^{e,\lambda})<U(x)\le\delta$.
Thus \eqref{cs:eq:stationary-tail-coefficient} gives
$a_\lambda\ge1+\gamma$ there, whereas $a_\lambda\ge m_g$
on $E$.
It follows that
\begin{equation}\label{cs:eq:stationary-strip-estimate}
 \gamma\|\zeta\|_2^2
 \le(1+\gamma-m_g)\int_E\zeta^2.
\end{equation}
Also \eqref{cs:eq:stationary-negative-part} and
Cauchy--Schwarz give
$\|\zeta\|_\infty\le m_g^{-1}\|J\|_2\|\zeta\|_2$.
By \eqref{cs:eq:stationary-strip-estimate} and
\eqref{cs:eq:stationary-strip}, $\zeta=0$, so $w_\lambda\ge0$ for
$\lambda\in[\lambda_*-\varepsilon,\lambda_*]$. Together with the
order for $\lambda\ge\lambda_*$, this contradicts the definition
of $\lambda_*$. Hence $w_{\lambda_*}\equiv0$.

For every $\lambda>\lambda_*$ the same order holds.
By Lemma~\ref{cs:lem:halfspace-positive}, it is strict unless
$w_\lambda\equiv0$.
Two distinct parallel symmetry planes would make $U$ invariant
under a nonzero translation; iteration and decay at infinity would
give $U\equiv0$. Hence $w_\lambda>0$ for $\lambda>\lambda_*$.
By choosing the plane midway between two points on the same line,
we deduce that $U$ is symmetric about $x\cdot e=\lambda_*$
and strictly decreasing along each line on the side
$x\cdot e>\lambda_*$.

Since $U>0$ and $U(x)\to0$ as $|x|\to\infty$,
it attains its maximum at some $x_*$.
The strict linewise decrease forces $x_*\cdot e=\lambda_*$.
Repeating the argument for every unit vector $e$ gives symmetry
under every reflection through a hyperplane containing $x_*$.
These reflections send any point of a sphere centered at $x_*$
to any other point of that sphere, so $U$ is radial about $x_*$
and strictly decreasing in radius.

To prove $V'(r)<0$ in \eqref{cs:eq:stationary-radial}, translate
$x_*$ to zero.  Fix $i\in\{1,\ldots,d\}$, let $e_i$ be the
$i$th coordinate vector, and put $w=-\partial_iU$ on $\{x_i>0\}$.
Radial decrease gives $w\ge0$; because the radial profile is
nonconstant, $w\not\equiv0$.
By \eqref{cs:eq:stationary-derivative} and oddness of $\partial_iU$
across $\{x_i=0\}$,
\[
 H_g'(U(x))w(x)
 =\int_{y_i>0}\bigl(J(x-y)-J(x-y^{e_i,0})\bigr)w(y)\,dy.
\]
Lemma~\ref{cs:lem:halfspace-positive} gives $w>0$ there.
Evaluation on the positive $i$th coordinate axis proves $V'(r)<0$
for every $r>0$.
In particular, $x_*$ is the unique maximum point and the unique
center in \eqref{cs:eq:stationary-radial}.
\end{proof}

\begin{lemma}\label{lem:peakidentity}
Assume $d=1$, (KM), \textup{(F1)--(F2)}, and $F(1)>0$. Let $U$ be a nonzero even stationary solution of \eqref{eq:main-equation} that is nonincreasing on $[0,\infty)$, satisfies $U(x)\to0$ as $|x|\to\infty$, and has $0\le U\le1$.  Use the continuous representative of $U$ supplied by Lemma~\ref{lem:steadyregularity}, and set $A=U(0)$.  Let $\beta\in(\vartheta,1)$ be the unique zero of $F$ in that interval.  Then
\begin{equation}\label{eq:peakidentity}
 F(A)=\int_0^\infty J(h)\int_0^h
 [A-U(h-x)]\,[-U'(x)]\,dx\,dh\ge0.
\end{equation}
In particular $A\ge\beta$. If $J(h)>0$ for every $h>0$, as for a Gaussian kernel, then $F(A)>0$ and $A>\beta$.
\end{lemma}

\begin{proof}
By (KM) and Lemma~\ref{lem:steadyregularity}, $U\in C^1(\R)$.
Evenness, monotonicity and decay, together with $U\not\equiv0$, give
\begin{equation}\label{eq:display-5-02}
 A=U(0)>0,\qquad U'(x)\le0\quad(x>0),
\end{equation}
and
\begin{equation}\label{eq:profile-total-variation}
 U'\in L^1(0,\infty),\qquad
 \int_0^\infty[-U'(x)]\,dx=A.
\end{equation}
To relate $F(A)$ to the convolution, integrate
$J*U-U+f(U)=0$ against $U'$ on $(0,\infty)$.
All terms are absolutely integrable by \eqref{eq:profile-total-variation}
and boundedness of $J*U-U$ and $f(U)$.
The chain rule gives $\int_0^M f(U)U'\,dx=F(U(M))-F(A)$;
letting $M\to\infty$ and using $U(M)\to0$, we obtain
\begin{equation}\label{eq:display-5-04}
 \int_0^\infty f(U(x))U'(x)\,dx
 =\int_A^0 f(s)\,ds=-F(A).
\end{equation}
Using \eqref{eq:display-5-04} in the integrated stationary equation gives
\begin{equation}\label{eq:peak-potential-representation}
 F(A)=\int_0^\infty (J*U-U)U'\,dx.
\end{equation}

By evenness of $J$ and its unit mass, for $x>0$,
\begin{equation}\label{eq:even-convolution-second-difference}
 (J*U-U)(x)
 =\int_0^\infty J(h)
 [U(x-h)+U(x+h)-2U(x)]\,dh.
\end{equation}
To interchange the integrals in
\eqref{eq:peak-potential-representation}, use $0\le U\le A$ to obtain
\begin{equation}\label{eq:peak-integrability-bound}
 |U(x-h)+U(x+h)-2U(x)|\,|U'(x)|
 \le 2A|U'(x)|,
\end{equation}
which, together with $J\in L^1(\R)$ and
\eqref{eq:profile-total-variation}, justifies Fubini's theorem.
For $h>0$, set
\begin{equation}\label{eq:peak-shift-integral}
 I(h)=\int_0^\infty
 [U(x-h)+U(x+h)-2U(x)]U'(x)\,dx.
\end{equation}
Splitting the first term of \eqref{eq:peak-shift-integral}
at $x=h$ and using $U(x-h)=U(h-x)$ for $x<h$ gives
\begin{equation}\label{eq:peak-shift-splitting}
 \begin{aligned}
 I(h)
 &=\int_0^h U(h-x)U'(x)\,dx
   +\int_h^\infty U(x-h)U'(x)\,dx\\
 &\quad+\int_0^\infty U(x+h)U'(x)\,dx
   -2\int_0^\infty U(x)U'(x)\,dx.
 \end{aligned}
\end{equation}
With $y=x-h$, the second and third integrals in
\eqref{eq:peak-shift-splitting} combine to give
\begin{equation}\label{eq:peak-tail-cancellation}
 \int_0^\infty
 [U(y)U'(y+h)+U(y+h)U'(y)]\,dy
 =-AU(h),
\end{equation}
because $U(y)U(y+h)\to0$ as $y\to\infty$ and, by
\eqref{eq:profile-total-variation}, its derivative has $L^1(0,\infty)$
norm at most $2A\|U'\|_{L^1(0,\infty)}$.
The chain rule also gives $-2\int_0^\infty U(x)U'(x)\,dx=A^2$.
Substituting both identities into \eqref{eq:peak-shift-splitting}, we obtain
\begin{equation}\label{eq:peak-positive-integrand}
 \begin{aligned}
 I(h)
 &=A^2-AU(h)+\int_0^h U(h-x)U'(x)\,dx\\
 &=\int_0^h[A-U(h-x)][-U'(x)]\,dx.
 \end{aligned}
\end{equation}
The integrand in \eqref{eq:peak-positive-integrand} lies between
$0$ and $A[-U'(x)]$, so \eqref{eq:profile-total-variation} gives
$0\le I(h)\le A^2$. By \eqref{eq:peak-potential-representation}
and \eqref{eq:even-convolution-second-difference}, Fubini's theorem
and \eqref{eq:peak-positive-integrand} now yield
\begin{equation}\label{eq:peak-fubini-identity}
 \begin{aligned}
 F(A)
 &=\int_0^\infty J(h)I(h)\,dh\\
 &=\int_0^\infty J(h)\int_0^h
 [A-U(h-x)][-U'(x)]\,dx\,dh,
 \end{aligned}
\end{equation}
which is \eqref{eq:peakidentity} and proves $F(A)\ge0$.

For strict positivity, choose a bounded interval
$(a,b)\subset(0,\infty)$ with $\int_a^b[-U'(x)]\,dx>0$,
as permitted by \eqref{eq:profile-total-variation}.
Choose $r_0>0$ such that $U(r)\le A/2$ for $r\ge r_0$.
For $h>r_0+b$ and $x\in(a,b)$, we have $x<h$ and $h-x\ge r_0$;
restricting \eqref{eq:peak-positive-integrand} to $(a,b)$ gives
\begin{equation}\label{eq:strict-peak-integral}
 \int_0^h[A-U(h-x)][-U'(x)]\,dx
 \ge \frac A2\int_a^b[-U'(x)]\,dx>0.
\end{equation}
If $J(h)>0$ for every $h>0$, its mass on $(r_0+b,\infty)$
is positive; hence \eqref{eq:strict-peak-integral} and
\eqref{eq:peak-fubini-identity} give $F(A)>0$.
Finally, $A>0$ and the signs $F(s)<0$ on $(0,\beta)$ and
$F(\beta)=0$ from \eqref{eq:potential-signs} imply $A\ge\beta$,
with $A>\beta$ when $F(A)>0$.
\end{proof}

\begin{lemma}\label{cs:lem:stationary-peak}
Let $g$ satisfy \eqref{cs:eq:general-reaction}, and let $U$ be a nonzero
solution of \eqref{cs:eq:general-stationary} with
$U(x)\to0$ as $|x|\to\infty$.
Translate its center to zero and write $U(x)=V(|x|)$ as in
\eqref{cs:eq:stationary-radial}.
Set $A=V(0)$ and extend $V$ evenly to $\R$.
Let $e_1=(1,0,\ldots,0)$ and take $K=K_*$ in \eqref{eq:radial-marginal}
in the direction $e_1$.  Define
\[
 D_U(r)=(K*V)(r)-(J*U)(r e_1),\qquad r\ge0,
\]
where $K*V$ is convolution on $\R$.
Then $D_U\ge0$ and
\begin{equation}\label{cs:eq:peak-identity}
\begin{split}
 F_g(A)
 &=\int_0^\infty K(h)\int_0^h
       [A-V(h-r)][-V'(r)]\,dr\,dh\\
 &\quad+\int_0^\infty D_U(r)[-V'(r)]\,dr>0.
\end{split}
\end{equation}
For $d=1$, $D_U\equiv0$.
\end{lemma}

\begin{proof}
By Lemma~\ref{cs:lem:stationary-symmetry}, $V'(r)<0$ for $r>0$.
For $d\ge2$, write $z=(s,z')\in\R\times\R^{d-1}$.
Then
\begin{equation}\label{cs:eq:radial-marginal-comparison}
\begin{split}
 (J*U)(r e_1)
 &=\int_\R\int_{\R^{d-1}}J(s,z')
       V\!\left(\sqrt{(r-s)^2+|z'|^2}\right)\,dz'\,ds\\
 &\le\int_\R\int_{\R^{d-1}}J(s,z')V(|r-s|)\,dz'\,ds
   =(K*V)(r).
\end{split}
\end{equation}
Here the inequality follows from the decrease of $V$; for $d=1$,
the convolutions agree. Thus the stationary equation gives
\begin{equation}\label{cs:eq:peak-axis-equation}
 K*V-V+g(V)=D_U\ge0.
\end{equation}
Since $\int_0^\infty[-V'(r)]\,dr=A$ and the terms
$K*V$, $V$, $g(V)$, and $D_U$ are bounded, their products with
$V'$ are integrable on $(0,\infty)$.
Multiplying \eqref{cs:eq:peak-axis-equation} by $V'$ and integrating,
with $V(0)=A$ and $V(r)\to0$ as $r\to\infty$, gives
\begin{equation}\label{cs:eq:peak-split}
 F_g(A)=\int_0^\infty(K*V-V)V'\,dr
             +\int_0^\infty D_U(r)[-V'(r)]\,dr.
\end{equation}

By $0\le V\le A$ and $\int_0^\infty|V'|=A$, the majorant
$2AK(h)|V'(r)|$ is integrable. The calculation in
\eqref{eq:even-convolution-second-difference}--\eqref{eq:peak-positive-integrand}
therefore applies to $V$ and $K$, giving
\begin{equation}\label{cs:eq:peak-integrand}
 \int_0^\infty(K*V-V)V'\,dr
 =\int_0^\infty K(h)\int_0^h
 [A-V(h-r)][-V'(r)]\,dr\,dh.
\end{equation}
Together with \eqref{cs:eq:peak-split}, this proves
\eqref{cs:eq:peak-identity}. Since $V'<0$ on $(0,\infty)$,
the inner integral is positive for every $h>0$.
The strict inequality follows from $D_U\ge0$ and
$\int_0^\infty K(h)\,dh=1/2$.
\end{proof}

We now return to $f$ satisfying \textup{(F1)--(F2)}.
By Lemma~\ref{cs:lem:stationary-peak}, $F(\max U)>0$ for every
$U\in\mathcal S_0(J,f)\setminus\{0\}$.  Chmaj and Ren
\cite[Proposition~4.1(2)]{ChmajRen1999} obtained the bound
$\max U\ge\beta$ for their continuous even nonincreasing profiles on
$\R$; for $J(x)=\tfrac12e^{-|x|}$, their Proposition~4.1(3)
gives $F(\max U)=\tfrac12f(\max U)^2>0$.

\begin{lemma}\label{cs:lem:stationary-values}
Every $U\in\mathcal S_0(J,f)\setminus\{0\}$ satisfies
\begin{equation}\label{cs:eq:peak-bound}
 \beta<\max_{\R^d}U<1.
\end{equation}
Every $U\in\mathcal S(J,f)$ with $U(x)\to1$ as $|x|\to\infty$
is identically one.
\end{lemma}

\begin{proof}
Lemma~\ref{cs:lem:stationary-peak} with $g=f$ gives
$F(\max U)>0$.
By \eqref{eq:potential-signs} and the strict upper bound in
Lemma~\ref{cs:lem:stationary-basic}, this is exactly
\eqref{cs:eq:peak-bound}.

For the second assertion, fix $b\in(\vartheta,1)$ and take
$L_b$ from the spreading criterion \eqref{eq:measurable-seed-spreading}.
Since $U(x)\to1$ as $|x|\to\infty$, there is $R>0$ such that
$U\ge b$ on $\R^d\setminus B_R$.
Choose $x_0$ with $|x_0|>R+L_b$.
Then $U\ge b$ on $B_{L_b}(x_0)$, so translation invariance
and the spreading criterion \eqref{eq:measurable-seed-spreading} imply that the solution with
initial value $U$ converges locally uniformly to $1$.
By uniqueness in Lemma~\ref{lem:semiflow}, this solution equals
$U$ at every time.  Consequently $U\equiv1$.
\end{proof}

\begin{proof}[Proof of Theorem~\ref{cs:thm:stationary-summary}]
Fix $U\in\mathcal S_0(J,f)\setminus\{0\}$.
Lemma~\ref{cs:lem:stationary-regularity} gives
$U\in C^2(\R^d)\cap W^{2,\infty}(\R^d)$.
By Lemma~\ref{cs:lem:stationary-symmetry}, there exist
a unique center $x_*\in\R^d$ and an associated function $V$ such that
$U(x)=V(|x-x_*|)$ and $V'(r)<0$ for $r>0$.
The center is unique because it is the unique maximum point of $U$.
For any fixed $e\in\Sph$, the identity $V(r)=U(x_*+re)$ for $r\ge0$
then determines $V$ uniquely.
Thus $V\in C^2([0,\infty))$; Lemma~\ref{cs:lem:stationary-values}
gives $\beta<V(0)<1$.
The remaining assertions in \eqref{cs:eq:stationary-summary} follow
from Lemma~\ref{cs:lem:stationary-integrability} with $g=f$.

\end{proof}

\section{Convergence for compactly supported initial data}\label{cs:sec:compact-data}

We retain \textup{(F1)--(F2)}. Except in
Lemmas~\ref{lem:fredholm} and~\ref{lem:lojasiewicz}, assume
\textup{(KM)} and $F(1)>0$ until
Subsection~\ref{subsec:negative-integral}, where we prove the two
corollaries for $F(1)<0$. Lemmas~\ref{lem:fredholm}
and~\ref{lem:lojasiewicz} use the kernel and reaction hypotheses
in their statements. Until that final subsection,
let $u$ solve \eqref{eq:main-equation} with initial datum
\eqref{cs:eq:data}. We retain $X$, $\Sigma_0$ and $C_0$
from Section~\ref{sec:introduction}; inner products are taken in
$L^2(\R^d)$.

\subsection{Compactness and stationary limits}

We use the energy $\mathcal E$ in \eqref{eq:energy-definition} and set
\begin{equation}\label{cs:eq:energy-gradient}
 G(v)=v-J*v-f(v),\qquad
 D\mathcal E(v)h=\langle G(v),h\rangle_{L^2(\R^d)},\qquad v,h\in X.
\end{equation}
By \textup{(F1)}, $G\in C^1(X;X)$ and $\mathcal E\in C^2(X;\R)$.
When analyticity near $[0,A]$ is used, we choose a $C^1$ extension
agreeing with that analytic function on a neighborhood of $[0,A]$.
The energy and the equation on functions with values in $[0,1]$
are independent of this auxiliary extension.

\begin{lemma}\label{cs:lem:compact-orbit}
Let \textup{(F1)--(F2)}, \textup{(KM)}, and \eqref{cs:eq:data} hold.
Suppose that $u$ does not spread.
Then there are constants $C,A>0$ such that
\begin{equation}\label{cs:eq:mass-bound}
 \sup_{t\ge0}\bigl(\|u(t,\cdot)\|_1+\|u(t,\cdot)\|_2\bigr)\le C
\end{equation}
and, for every $r\ge A$,
\begin{equation}\label{cs:eq:tails}
 \sup_{t\ge0}\esssup_{|x|>r}u(t,x)\le Cr^{-d},
 \qquad
 \sup_{t\ge0}\int_{|x|>r}u(t,x)^2\,dx\le Cr^{-d}.
\end{equation}
The closure of $\{u(t,\cdot):t\ge0\}$ is compact in $X$.
Every limit in $X$ along times tending to infinity belongs to
$\mathcal S_0(J,f)$.  For $0\le s\le t$,
\begin{equation}\label{cs:eq:dissipation}
 \mathcal E(u(t,\cdot))-\mathcal E(u(s,\cdot))
  =-\int_s^t\|\partial_\tau u(\tau,\cdot)\|_2^2\,d\tau,
 \qquad
 \lim_{t\to\infty}\|u_t(t,\cdot)\|_2=0 .
\end{equation}
\end{lemma}

\begin{proof}
Fix $b\in(\vartheta,\beta)$.  Since $u$ does not spread,
\eqref{eq:measurable-seed-spreading} gives
$\essinf_{B_{L_b}}u(t,\cdot)<b$ for every $t\ge0$.
Thus \eqref{cs:eq:intro-reflection} with $r=L_b$ gives
\begin{equation}\label{cs:eq:exterior-b}
 u(t,x)\le b\qquad
 \text{for almost every }|x|>A_0:=L_b+2R_0,\quad t\ge0 .
\end{equation}
By \eqref{eq:potential-signs} and $f'(0)<0$, the number
\[
 c_b=\min_{0\le s\le b}\frac{-F(s)}{s^2}
\]
is positive, where the value at zero is $-f'(0)/2$.
Also $F(s)\le F(1)$ on $[0,1]$.
By \eqref{eq:energy-definition} and \eqref{cs:eq:exterior-b}, the nonnegativity
of the convolution term gives
\begin{equation}\label{cs:eq:energy-lower}
 \mathcal E(u(t,\cdot))
 \ge-F(1)|B_{A_0}|+
 c_b\int_{\R^d\setminus B_{A_0}}u(t,x)^2\,dx .
\end{equation}
By Lemma~\ref{lem:semiflow}, $u$ is $C^1$ in $X$; hence
\eqref{cs:eq:energy-gradient} and \eqref{eq:main-equation} give
\begin{equation}\label{cs:eq:energy-derivative}
 \frac{d}{dt}\mathcal E(u(t,\cdot))=-\|u_t(t,\cdot)\|_2^2 .
\end{equation}
The potential term is differentiable because its Taylor remainder is
bounded by $C\|u(t+h,\cdot)-u(t,\cdot)\|_2^2$.
Combining \eqref{cs:eq:energy-lower} with the upper bound
$\mathcal E(u(t,\cdot))\le\mathcal E(u_0)$ gives a uniform $L^2$ bound.

By \eqref{cs:eq:reflection-integral} with $p=2$, that bound implies
\[
 \esssup_{|x|>r+2R_0}u(t,x)\le C r^{-d/2}.
\]
Choose $A_1>A_0$ so large that $u(t,x)\le\vartheta/2$ outside
$B_{A_1}$ for every $t\ge0$.  The signs in
\textup{(F1)--(F2)} give $\gamma_0>0$ such that
$f(s)\le-\gamma_0s$ on $[0,\vartheta/2]$.
Since $u_0\in L^1$, Lemma~\ref{lem:semiflow} and Fubini's theorem give,
for $M(t)=\int_{\R^d}u(t,x)\,dx$,
\begin{equation}\label{cs:eq:mass-ode}
 M'(t)=\int_{\R^d}f(u(t,x))\,dx
 \le |B_{A_1}|\bigl(\max_{[0,1]}f+\gamma_0\bigr)-\gamma_0M(t).
\end{equation}
The unit mass of $J$ gives $\int_{\R^d}(J*u-u)=0$.
By scalar comparison in \eqref{cs:eq:mass-ode}, the $L^1$ bound in
\eqref{cs:eq:mass-bound} holds.  Thus
\eqref{cs:eq:reflection-integral} with $p=1$ gives the first bound in
\eqref{cs:eq:tails} for large $A$, and the second follows from
\begin{equation}\label{cs:eq:tail-product}
 \int_{|x|>r}u(t,x)^2\,dx
 \le \esssup_{|x|>r}u(t,x)\,\|u(t,\cdot)\|_1 .
\end{equation}

By Lemma~\ref{lem:asymcompact}, every sequence $t_n\to\infty$
has a subsequence for which $u(t_n,\cdot)$ converges in $L^\infty(B_r)$
for every $r>0$.  The bounds \eqref{cs:eq:tails} pass to the limit
by local convergence and Fatou's lemma.  Splitting $\R^d$ into $B_r$
and its complement and then letting $r\to\infty$ upgrades this
convergence to $X$.  For bounded sequences of times, continuity
of $t\mapsto u(t,\cdot)$ in $X$, supplied by Lemma~\ref{lem:semiflow},
gives a convergent subsequence.  Thus the whole orbit has compact
closure in $X$.

By \eqref{cs:eq:energy-lower} and \eqref{cs:eq:energy-derivative},
$\mathcal E(u(t,\cdot))$ has a finite limit as $t\to\infty$ and
$\int_0^\infty\|u_t(t,\cdot)\|_2^2\,dt<\infty$.
Set $v=u_t$.  Since $G\in C^1(X;X)$, differentiation of
$u_t=-G(u)$ gives
\begin{equation}\label{cs:eq:time-linearization}
 v_t=J*v-v+f'(u)v .
\end{equation}
By \eqref{eq:main-equation}, \eqref{cs:eq:time-linearization}, and
$L_f=\max_{[0,1]}|f'|$,
\[
 \|v(t,\cdot)\|_2\le(2+L_f)\|u(t,\cdot)\|_2,\qquad
 \|v_t(t,\cdot)\|_2\le(2+L_f)\|v(t,\cdot)\|_2.
\]
The uniform $L^2$ bound for $u$ therefore bounds $v$ and $v_t$, so
$t\mapsto\|v(t,\cdot)\|_2^2$ is uniformly continuous.
Its integrability then gives the last assertion in
\eqref{cs:eq:dissipation}: a positive limit superior would produce
disjoint intervals of fixed positive length with a common positive lower bound.

Finally, if $u(t_n,\cdot)\to U$ in $X$ along $t_n\to\infty$,
continuity of $G$ and \eqref{cs:eq:dissipation}, with $u_t=-G(u)$,
give $G(U)=0$.  By Lemma~\ref{lem:steadyregularity} and \eqref{cs:eq:tails},
$U$ has a continuous representative tending to zero at infinity;
hence $U\in\mathcal S_0(J,f)$.
\end{proof}

\subsection{Fredholm operators and the analytic gradient inequality}\label{subsec:analytic-tools}

In Lemmas~\ref{lem:fredholm} and~\ref{lem:lojasiewicz},
$U$ denotes a stationary solution of \eqref{eq:main-equation} satisfying
\begin{equation}\label{eq:display-4-18}
 0\le U\le1,\qquad U\in L^2(\R^d),\qquad U(x)\to0\quad(|x|\to\infty).
\end{equation}
Define
\begin{equation}\label{eq:energy-hessian}
 \mathcal A_U h=h-J*h-f'(U)h
 \quad\text{on }L^2(\R^d).
\end{equation}
A bounded linear operator on $L^2(\R^d)$ is Fredholm if its range is closed and its kernel and cokernel are finite dimensional.  Its index is the dimension of the kernel minus that of the cokernel.

\begin{lemma}\label{lem:fredholm}
Assume (KR) and \textup{(F3)}, and let $U$ be as in \eqref{eq:display-4-18}.  Then $\mathcal A_U$ is bounded, self-adjoint and Fredholm of index zero.  Its kernel $K_U$ is finite dimensional and is contained in $X=L^2(\R^d)\cap L^\infty(\R^d)$.  We write $K_U^\perp$ for the orthogonal complement of $K_U$ in $L^2(\R^d)$.  If $Q_U$ denotes the $L^2(\R^d)$-orthogonal projection onto $K_U^\perp$, then
\begin{equation}\label{eq:hessian-complement-isomorphism}
 \mathcal A_U:K_U^\perp\cap X\to K_U^\perp\cap X
\end{equation}
is an isomorphism with bounded inverse when both spaces are equipped with the norm of $X$.
\end{lemma}

\begin{proof}
Put
\begin{equation}
 q(x)=1-f'(U(x)),\qquad q_\infty=1-f'(0)>1.
\end{equation}
Write $M_b$ for multiplication by a bounded measurable $b$ on
$L^2(\R^d)$. By Lemma~\ref{lem:nonlinearity}, $q\ge m_f$;
evenness of $J$ makes convolution self-adjoint. Hence
$\mathcal A_U=M_q-J*$
is bounded and self-adjoint.

Factor
\begin{equation}\label{eq:hessian-factorization}
 \mathcal A_U=M_q(I-T),
 \qquad T=M_{q^{-1}}J*.
\end{equation}
By \textup{(F2)--(F3)}, $r\mapsto(1-f'(r))^{-1}$ is
Lipschitz on $[0,1]$, with value $q_\infty^{-1}$ at zero.
Since $U\in L^2(\R^d)$, the function
\begin{equation}
 p(x)=q(x)^{-1}-q_\infty^{-1}
\end{equation}
belongs to $L^2(\R^d)$ and satisfies $\|p\|_2\le C\|U\|_2$.  Moreover $J\in L^1(\R^d)\cap L^\infty(\R^d)$, hence $J\in L^2(\R^d)$.  Thus
\begin{equation}\label{eq:convolution-compact-decomposition}
 T=q_\infty^{-1}J*+M_pJ*.
\end{equation}
In \eqref{eq:convolution-compact-decomposition}, $q_\infty^{-1}J*$ has operator norm at most $q_\infty^{-1}<1$ on $L^2(\R^d)$.  The operator $M_pJ*$ is Hilbert--Schmidt: its integral kernel is $p(x)J(x-y)$, and
\begin{equation}\label{eq:hilbert-schmidt-bound}
 \iint_{\R^d\times\R^d} |p(x)|^2|J(x-y)|^2\,dx\,dy
 =\|p\|_2^2\|J\|_2^2<\infty.
\end{equation}
The Neumann series makes $I-q_\infty^{-1}J*$ invertible on
$L^2(\R^d)$. By \eqref{eq:convolution-compact-decomposition} and
\eqref{eq:hilbert-schmidt-bound}, $I-T$ is its compact perturbation,
hence Fredholm of index zero. Since $q$ is bounded and
$q\ge m_f>0$, $M_q$ is invertible; \eqref{eq:hessian-factorization}
gives the same Fredholm conclusion for $\mathcal A_U$.

Self-adjointness and closed range give
$\operatorname{Ran}\mathcal A_U=K_U^\perp$; Fredholmness makes
$K_U$ finite dimensional. For $h\in K_U$, $qh=J*h$ and
Cauchy--Schwarz give
$\|h\|_\infty\le m_f^{-1}\|J\|_2\|h\|_2$, so $K_U\subset X$.

Let $N=\dim K_U$, and choose an $L^2(\R^d)$-orthonormal basis $\phi_1,\dots,\phi_N$ of $K_U$.  Since each $\phi_j\in X$, the orthogonal projection
\begin{equation}
 Ph=\sum_{j=1}^N\langle h,\phi_j\rangle_{L^2(\R^d)}\phi_j
\end{equation}
is bounded from $L^2(\R^d)$ to $X$ (with $P=0$ if $N=0$); consequently, the complementary projection $Q=Q_U=I-P$ is bounded on $X$.  We write
$QX=\{h\in X:h\perp K_U\text{ in }L^2(\R^d)\}$.
Self-adjointness gives $\mathcal A_U(QX)\subset QX$.

Since $\operatorname{Ran}\mathcal A_U=K_U^\perp$ and
$K_U\cap K_U^\perp=\{0\}$, the restriction to $K_U^\perp$ is a
bounded bijection. The bounded inverse theorem \cite{Kato1995}
gives, for $g\in QX$, a unique $h\in K_U^\perp$ with
$\mathcal A_Uh=g$ and
\begin{equation}\label{eq:hessian-inverse-l2-bound}
 \|h\|_2\le C\|g\|_2.
\end{equation}
The identity $qh=J*h+g$, the lower bound $q\ge m_f$, and \eqref{eq:hessian-inverse-l2-bound} give
\begin{equation}\label{eq:hessian-inverse-linfty-bound}
 \|h\|_\infty
 \le m_f^{-1}\bigl(\|J\|_2\|h\|_2+\|g\|_\infty\bigr)
 \le C\|g\|_X.
\end{equation}
By \eqref{eq:hessian-inverse-linfty-bound}, $h\in QX$;
combining the two inverse estimates gives
\begin{equation}\label{eq:hessian-inverse-x-bound}
 \|h\|_X\le C\|g\|_X.
\end{equation}
Thus \eqref{eq:hessian-inverse-x-bound} completes the proof of
\eqref{eq:hessian-complement-isomorphism}.
\end{proof}

\begin{lemma}\label{lem:lojasiewicz}
Assume (KR) and \textup{(F4)}, and let $U$ be as in \eqref{eq:display-4-18}.  There exist $\rho>0$, $C>0$ and $\theta\in(0,1/2]$ such that every $v\in X$ satisfying $\|v-U\|_X<\rho$ obeys
\begin{equation}\label{eq:local-gradient-inequality}
 |\mathcal E[v]-\mathcal E[U]|^{1-\theta}
 \le C\|v-J*v-f(v)\|_2.
\end{equation}
\end{lemma}

\begin{proof}
\smallskip
\noindent\textit{Step 1. Analyticity of the energy and its gradient.}
Define $G:X\to X$ by $G(v)=v-J*v-f(v)$.
By \textup{(F4)}, choose a compact interval $I$ with
$[0,1]\subset\operatorname{int}I$ and $f$ analytic near $I$.
A sufficiently small $X$-neighborhood of $U$ has essential ranges
in $I$. Compactness gives $a,C>0$ such that
\begin{equation}\label{eq:analytic-derivative-bound}
 \sup_{s\in I}|f^{(k)}(s)|\le Ck!a^{-k}
 \qquad(k\ge0).
\end{equation}
Since $f(0)=0$, $f(v)\in X$ for $v\in X$ with range in $I$.
For $v$ in the chosen neighborhood and small $h\in X$ with
$\|h\|_\infty<a$, the Taylor series converges in $X$, because
\begin{equation}\label{eq:power-norm-bounds}
 \|h^k\|_2\le\|h\|_\infty^{k-1}\|h\|_2,
 \qquad
 \|h^k\|_\infty\le\|h\|_\infty^k.
\end{equation}
By \eqref{eq:analytic-derivative-bound}--\eqref{eq:power-norm-bounds},
the Nemytskii map $v\mapsto f(v)$ is analytic from this neighborhood
to $X$, with $k$th Fr\'echet derivative, $k\ge1$,
\begin{equation}
 (h_1,\ldots,h_k)\longmapsto f^{(k)}(v)h_1\cdots h_k.
\end{equation}
Since $F(0)=F'(0)=0$, there is a constant $C_I>0$ such that $|F(s)|\le C_Is^2$ and $|f(s)|\le C_I|s|$ on $I$.  Thus $\int_{\R^d}F(v(x))\,dx$ is finite.  Its Taylor expansion is justified in $L^1(\R^d)$: the linear term satisfies $\|f(v)h\|_1\le C_I\|v\|_2\|h\|_2$, and, for $k\ge2$,
\begin{equation}\label{eq:potential-taylor-bound}
 \|F^{(k)}(v)h^k\|_1
 \le \sup_I|F^{(k)}|\,\|h\|_\infty^{k-2}\|h\|_2^2.
\end{equation}
By \eqref{eq:analytic-derivative-bound} and
\eqref{eq:potential-taylor-bound}, the series converges absolutely
for small $\|h\|_\infty$, so the potential integral is analytic near
$U$. The interaction term is the continuous quadratic form
$\frac12\langle v-J*v,v\rangle_{L^2(\R^d)}$; hence $\mathcal E$
is analytic on an $X$-neighborhood of $U$.
Writing $D$ for Fr\'echet differentiation,
\begin{equation}\label{eq:energy-frechet-derivative}
 D\mathcal E(v)h=\langle G(v),h\rangle_{L^2(\R^d)},
\end{equation}
and $DG(U):X\to X$ is the restriction of the operator $\mathcal A_U$ in \eqref{eq:energy-hessian} to $X$.

\smallskip
\noindent\textit{Step 2. Reduction to the kernel of the Hessian.}
Let $K=K_U$, and let $P$ and $Q=I-P$ be the projections from Lemma~\ref{lem:fredholm}.  By \eqref{eq:hessian-complement-isomorphism}, the restriction $\mathcal A_U:QX\to QX$ is an isomorphism.  Since $Q:X\to QX$ is bounded, the map
\begin{equation}
 (\xi,\eta)\longmapsto QG(U+\xi+\eta),
 \qquad \xi\in K,\quad \eta\in QX,
\end{equation}
is analytic near $(0,0)$ in $K\times QX$ with values in $QX$, and its derivative with respect to $\eta$ at $(0,0)$ is $\mathcal A_U|_{QX}$.  The analytic implicit-function theorem (see \cite[Section~2.1.1]{FeehanMaridakis2020}) therefore gives an open neighborhood $V\subset K$ of $0$ and an analytic map $\Psi:V\to QX$.  Since $G(U)=0$ and $\mathcal A_U\xi=0$ for $\xi\in K$, this map satisfies
\begin{equation}\label{eq:display-4-37}
 \Psi(0)=0,\qquad D\Psi(0)=0,
\end{equation}
and
\begin{equation}\label{eq:implicit-graph-equation}
 QG(U+\xi+\Psi(\xi))=0
 \qquad (\xi\in V).
\end{equation}
Define $\varphi:V\to\R$ by
\begin{equation}
 \varphi(\xi)=\mathcal E[U+\xi+\Psi(\xi)].
\end{equation}
Use the inherited $L^2(\R^d)$ inner product for gradients on
$K$, and set $w(\xi)=U+\xi+\Psi(\xi)$. Then $\varphi$ is analytic,
and \eqref{eq:implicit-graph-equation} gives $G(w(\xi))\in K$.
Since $D\Psi(\xi)K\subset QX\subset K^\perp$,
\eqref{eq:energy-frechet-derivative} gives, for $h\in K$,
\begin{equation}\label{eq:reduced-energy-derivative}
 D\varphi(\xi)h
 =\langle G(w(\xi)),h+D\Psi(\xi)h\rangle_{L^2(\R^d)}
 =\langle G(w(\xi)),h\rangle_{L^2(\R^d)}.
\end{equation}
The identity \eqref{eq:reduced-energy-derivative} identifies the gradient as
\begin{equation}\label{eq:reduced-gradient}
 \nabla\varphi(\xi)=G(w(\xi))=PG(w(\xi)).
\end{equation}
Since $\nabla\varphi(0)=G(U)=0$ by \eqref{eq:reduced-gradient}, the finite-dimensional \L ojasiewicz gradient inequality \cite[Theorem~2.2(1)]{FeehanMaridakis2020} gives $C_0>0$ and $\theta\in(0,1/2]$ such that
\begin{equation}\label{eq:finite-dimensional-gradient-inequality}
 |\varphi(\xi)-\varphi(0)|^{1-\theta}
 \le C_0\|G(w(\xi))\|_2
\end{equation}
for $\xi$ sufficiently small.  To ensure $1-\theta\ge1/2$, one may increase the exponent supplied by that theorem after restricting to $|\varphi(\xi)-\varphi(0)|\le1$.  If $K=\{0\}$, \eqref{eq:finite-dimensional-gradient-inequality} is trivial and we take $\theta=1/2$.  See \cite{Simon1983,Chill2003} for the corresponding \L ojasiewicz--Simon convergence argument.

\smallskip
\noindent\textit{Step 3. Comparison with the reduced energy.}
For $v\in X$ sufficiently close to $U$ in the norm of $X$, write
\begin{equation}\label{eq:display-4-43}
 v=U+\xi+\eta,
 \qquad \xi=P(v-U),\quad \eta=Q(v-U).
\end{equation}
Boundedness of $P:X\to K$ from Lemma~\ref{lem:fredholm}
ensures $\xi\in V$ for small $\|v-U\|_X$. Set
\begin{equation}\label{eq:graph-point-and-error}
 w=U+\xi+\Psi(\xi),
 \qquad z=\eta-\Psi(\xi)\in QX.
\end{equation}
By \eqref{eq:display-4-37} and continuity of $\Psi$,
\eqref{eq:display-4-43}--\eqref{eq:graph-point-and-error} make
$\|w-U\|_X$ and $\|z\|_X$ small with $\|v-U\|_X$.
We also use $Q$ as the $L^2(\R^d)$ projection onto $K^\perp$.
Write $\mathcal L(Y)$ for bounded linear operators on $Y$ with the
operator norm. For nearby $y\in X$,
$DG(y)h=h-J*h-f'(y)h$ extends to a bounded operator on $L^2(\R^d)$.
Using this extension, define on $K^\perp$
\begin{equation}\label{eq:averaged-hessian}
 B_{w,z}
 =\int_0^1 QDG(w+sz)\big|_{K^\perp}\,ds.
\end{equation}
The integral is in operator norm; Lipschitz continuity of $f'$
on $I$ makes the integrand continuous. At $(w,z)=(U,0)$ it equals
$\mathcal A_U|_{K^\perp}$, invertible by Lemma~\ref{lem:fredholm}.
Moreover,
\begin{equation}\label{eq:display-4-46}
 DG(y)-DG(U)=-M_{f'(y)-f'(U)},
\end{equation}
so boundedness of $f''$ on $I$ gives
$\|DG(y)-DG(U)\|_{\mathcal L(L^2(\R^d))}\le C\|y-U\|_\infty$.
With $A_Q=\mathcal A_U|_{K^\perp}$, \eqref{eq:averaged-hessian}
therefore yields
\begin{equation}\label{eq:averaged-hessian-perturbation}
 \|B_{w,z}-A_Q\|_{\mathcal L(QL^2(\R^d))}
 \le C\bigl(\|w-U\|_\infty+\|z\|_\infty\bigr).
\end{equation}
After shrinking the $X$-neighborhood of $U$, the right-hand side of \eqref{eq:averaged-hessian-perturbation} is smaller than $\|A_Q^{-1}\|^{-1}/2$.  The Neumann-series argument therefore gives invertibility of $B_{w,z}$ on $QL^2(\R^d)$ and the uniform bound
\begin{equation}\label{eq:averaged-hessian-inverse-bound}
 \|B_{w,z}^{-1}\|_{\mathcal L(K^\perp)}\le C_1.
\end{equation}
By \eqref{eq:implicit-graph-equation}, $QG(w)=0$, and the mean-value formula gives
$QG(v)=B_{w,z}z$,
so the inverse bound \eqref{eq:averaged-hessian-inverse-bound} yields
\begin{equation}\label{eq:graph-error-bound}
 \|z\|_2\le C_1\|QG(v)\|_2
 \le C_1\|G(v)\|_2.
\end{equation}
A second application of the mean-value formula gives
\begin{equation}\label{eq:gradient-mean-value}
 G(v)-G(w)=\int_0^1DG(w+sz)z\,ds.
\end{equation}
The $L^2(\R^d)$ extensions of $DG(y)$ are uniformly bounded
for nearby $y\in X$. Thus \eqref{eq:gradient-mean-value} and
\eqref{eq:graph-error-bound} give
\begin{equation}\label{eq:graph-gradient-bound}
 \|G(w)\|_2
 \le \|G(v)\|_2+C\|z\|_2
 \le C_2\|G(v)\|_2.
\end{equation}

Finally, the derivative formula \eqref{eq:energy-frechet-derivative} gives
\begin{equation}
 \mathcal E[v]-\mathcal E[w]
 =\int_0^1\langle G(w+sz),z\rangle_{L^2(\R^d)}\,ds.
\end{equation}
Since $G(w)\in K$ and $z\in K^\perp$, $\langle G(w),z\rangle_{L^2(\R^d)}=0$.  If $M$ bounds the $L^2(\R^d)$ operator norm of $DG$ on this neighborhood, then
\[
 \begin{aligned}
 |\mathcal E[v]-\mathcal E[w]|
 &\le \int_0^1\|G(w+sz)-G(w)\|_2\,\|z\|_2\,ds\\
 &\le M\int_0^1s\,ds\,\|z\|_2^2
 =\frac M2\|z\|_2^2.
 \end{aligned}
\]
Combining this estimate with \eqref{eq:graph-error-bound} gives
\begin{equation}\label{eq:graph-energy-error-bound}
 |\mathcal E[v]-\mathcal E[w]|
 \le C_3\|z\|_2^2
 \le C_4\|G(v)\|_2^2.
\end{equation}
Combining \eqref{eq:graph-energy-error-bound}, \eqref{eq:finite-dimensional-gradient-inequality}, and \eqref{eq:graph-gradient-bound} yields
\begin{equation}\label{eq:gradient-inequality-combination}
 \begin{aligned}
 |\mathcal E[v]-\mathcal E[U]|^{1-\theta}
 &\le C|\mathcal E[w]-\mathcal E[U]|^{1-\theta}
      +C|\mathcal E[v]-\mathcal E[w]|^{1-\theta}\\
 &\le C\|G(v)\|_2+C\|G(v)\|_2^{2(1-\theta)}.
 \end{aligned}
\end{equation}
Shrink the neighborhood so that $\|G(v)\|_2\le1$.
Since $2(1-\theta)\ge1$, \eqref{eq:gradient-inequality-combination}
then implies \eqref{eq:local-gradient-inequality}.
\end{proof}

\subsection{A local gradient inequality and convergence}

\begin{lemma}\label{cs:lem:fredholm}
Assume \textup{(F1)--(F2)} and \textup{(KM)}, and let
$U\in\mathcal S_0(J,f)$.  Define
\begin{equation}\label{cs:eq:hessian}
 a_U(x)=1-f'(U(x)),\qquad a_\infty=1-f'(0)>1,\qquad
 A_Uh=a_Uh-J*h .
\end{equation}
Then $A_U$ is a bounded self-adjoint Fredholm operator of index zero
on $L^2(\R^d)$.
Its kernel $K_U=\ker A_U$ is contained in $X$.
Let $P_U$ be the $L^2$ orthogonal projection onto $K_U$ and set
$Q_U=I-P_U$.
The restriction $A_U:Q_UX\to Q_UX$ is an isomorphism.
\end{lemma}

\begin{proof}
Evenness of $J$ and the bounds $m_f\le a_U\le\max_{[0,1]}(1-f')$
give boundedness and self-adjointness.  As in Lemma~\ref{lem:fredholm},
write $M_b$ for multiplication by $b$ and factor
\begin{equation}\label{cs:eq:fredholm-factor}
 A_U=M_{a_U}\bigl(I-a_\infty^{-1}J*-M_{p_U}J*\bigr),
 \qquad p_U=a_U^{-1}-a_\infty^{-1}.
\end{equation}
Continuity of $f'$ and $U(x)\to0$ imply $p_U(x)\to0$ at infinity.
For each $R>0$, $M_{p_U\mathbf1_{B_R}}J*$ is Hilbert--Schmidt,
with squared kernel norm $\|p_U\mathbf1_{B_R}\|_2^2\|J\|_2^2$.
Moreover,
\begin{equation}\label{cs:eq:compact-multiplier}
 \|M_{p_U\mathbf1_{\R^d\setminus B_R}}J*\|_{L^2\to L^2}
 \le\sup_{|x|>R}|p_U(x)|\to0\qquad(R\to\infty).
\end{equation}
Thus $M_{p_U}J*$ is compact, while $a_\infty>1$ makes
$I-a_\infty^{-1}J*$ invertible by a Neumann series.
By \eqref{cs:eq:fredholm-factor}, $A_U$ has Fredholm index zero;
self-adjointness then gives $\operatorname{Ran}A_U=K_U^\perp$.
For $h\in K_U$, $a_Uh=J*h$ gives
$\|h\|_\infty\le m_f^{-1}\|J\|_2\|h\|_2$, so $K_U\subset X$.
The bounded inverse on $K_U^\perp\subset L^2(\R^d)$ and
$a_Uh=J*h+g$ yield
\begin{equation}\label{cs:eq:inverse-bounds}
 \|h\|_2\le C\|g\|_2,\qquad
 \|h\|_\infty\le m_f^{-1}(\|J\|_2\|h\|_2+\|g\|_\infty),
 \qquad A_Uh=g\in Q_UX.
\end{equation}
These estimates prove the asserted isomorphism on $Q_UX$.
\end{proof}

\begin{lemma}\label{cs:lem:gradient}
Assume \textup{(F1)--(F2)} and \textup{(KM)}, and let
$U\in\mathcal S_0(J,f)\cap X$.
If $U\not\equiv0$, assume that either $f$ is real analytic near
$[0,\|U\|_\infty]$ or \eqref{cs:eq:translation-kernel} holds.
Then there are $\delta_U>0$, $C_U>0$, and
$\theta_U\in(0,1/2]$ such that, for $v\in X$,
\begin{equation}\label{cs:eq:gradient-inequality}
 |\mathcal E(v)-\mathcal E(U)|^{1-\theta_U}
 \le C_U\|G(v)\|_2
 \qquad\text{if }0\le v\le1,\quad \|v-U\|_X<\delta_U .
\end{equation}
If $U\equiv0$ or \eqref{cs:eq:translation-kernel} holds, one may take
$\theta_U=1/2$.
\end{lemma}

\begin{proof}
Retain the notation of Lemma~\ref{cs:lem:fredholm}.
In the analytic case, use the analytic extension of $f$ on a neighborhood
of $[0,\|U\|_\infty]$ to define $G$ and $\mathcal E$ throughout this proof.
For $0\le v\le1$ sufficiently close to $U$, these maps agree with the
original gradient and energy.
By uniform continuity of $f'$ on compact intervals, the integral
formula for the first-order remainder gives
\begin{equation}\label{cs:eq:composition-derivative}
 \|f(v+h)-f(v)-f'(v)h\|_X=o(\|h\|_X)
 \quad(\|h\|_X\to0),\qquad
 DG(v)h=(1-f'(v))h-J*h .
\end{equation}
The same uniform continuity gives continuity of $DG$ in operator
norm on $X$ and of its extension as an operator on $L^2$.

For $\xi\in K_U$ and $\eta\in Q_UX$, consider
$Q_UG(U+\xi+\eta)=0$.
By \eqref{cs:eq:inverse-bounds}, its derivative in $\eta$ at $(0,0)$,
$A_U|_{Q_UX}$, is invertible.  The $C^1$ implicit-function theorem
therefore gives a $C^1$ map $\Psi$ from a neighborhood of $0$ in $K_U$ to $Q_UX$.
Since $G(U)=0$ and $A_UK_U=0$,
\begin{equation}\label{cs:eq:reduced-map}
 \Psi(0)=0,\qquad D\Psi(0)=0,\qquad
 Q_UG(U+\xi+\Psi(\xi))=0 .
\end{equation}
Set $w(\xi)=U+\xi+\Psi(\xi)$ and
$\varphi(\xi)=\mathcal E(w(\xi))$.
By \eqref{cs:eq:reduced-map}, $G(w(\xi))\in K_U$, whereas
$D\Psi(\xi)$ takes values in $Q_UX$.
Differentiating with \eqref{cs:eq:energy-gradient} gives
\begin{equation}\label{cs:eq:finite-gradient}
 \nabla\varphi(\xi)=G(w(\xi))\in K_U ,
\end{equation}
where the gradient uses the $L^2$ inner product on $K_U$.

If $K_U=\{0\}$, then $w=U$ and the reduced energy is constant.
This includes $U\equiv0$, because $A_0=a_\infty I-J*$ is invertible.
For the remaining cases we prove
\begin{equation}\label{cs:eq:finite-lojasiewicz}
 |\mathcal E(w(\xi))-\mathcal E(U)|^{1-\theta_U}
 \le C\|G(w(\xi))\|_2
\end{equation}
with $\theta_U\in(0,1/2]$.

In the analytic case, the estimates \eqref{eq:analytic-derivative-bound}--\eqref{eq:potential-taylor-bound},
applied on a compact interval whose interior contains $[0,\|U\|_\infty]$,
show that $G$ and $\mathcal E$ are analytic. The analytic implicit-function
theorem therefore makes $\Psi$ analytic, so
\eqref{cs:eq:finite-gradient} and
\cite[Theorem~2.2(1)]{FeehanMaridakis2020} give
\eqref{cs:eq:finite-lojasiewicz} with $\theta_U\in(0,1/2]$,
after increasing the exponent if necessary.
Now suppose that \eqref{cs:eq:translation-kernel} holds and $U\not\equiv0$.
Differentiating $H(U)=J*U$ gives
\[
 \partial_iU=\frac{(\partial_iJ)*U}{1-f'(U)} .
\]
Since $U\in X$, $\partial_iJ\in L^1$, and $1-f'(U)\ge m_f$,
each $\partial_iU$ belongs to $X$ and is uniformly continuous.
Translations are continuous in $L^2$, and uniform continuity gives
continuity of translations of $\partial_iU$ in $L^\infty$.
The integral formula for increments therefore shows that
$a\mapsto U(\cdot-a)$ is $C^1$ from $\R^d$ into $X$.
Its derivative at zero in the direction $a$ is $-a\cdot\nabla U$.
The functions $\partial_1U,\ldots,\partial_dU$ are linearly independent:
$a\cdot\nabla U=0$ with $a\ne0$ would make $U$ constant on lines
parallel to $a$, and decay at infinity would give $U\equiv0$.

Define
\[
 \Gamma(a)=P_U\bigl(U(\cdot-a)-U\bigr),\qquad a\in\R^d .
\]
By \eqref{cs:eq:translation-kernel}, $D\Gamma(0)$ is an isomorphism
from $\R^d$ onto $K_U$.
The finite-dimensional inverse function theorem and uniqueness
in \eqref{cs:eq:reduced-map} imply that every $w(\xi)$ for small $\xi$
is a translate of $U$.
Translation invariance of the equation and of \eqref{eq:energy-definition}
therefore gives
\[
 G(w(\xi))=0,\qquad \mathcal E(w(\xi))=\mathcal E(U).
\]
Hence \eqref{cs:eq:finite-lojasiewicz} holds with $\theta_U=1/2$.
The same choice applies when $K_U=\{0\}$.

For a general $v$ near $U$, set $\xi=P_U(v-U)$,
$w=w(\xi)$ and $z=v-w\in K_U^\perp\cap X$.
By \eqref{cs:eq:composition-derivative} and
Lemma~\ref{cs:lem:fredholm}, the averaged operator in
\eqref{eq:averaged-hessian} is uniformly invertible on
$K_U^\perp\subset L^2(\R^d)$ for nearby $w,z$.
Indeed, the perturbation argument in Step~3 of
Lemma~\ref{lem:lojasiewicz} uses only continuity of $DG$ in the
$L^2$ operator norm, established above.
The estimates \eqref{eq:graph-error-bound},
\eqref{eq:graph-gradient-bound} and
\eqref{eq:graph-energy-error-bound} therefore give
\begin{equation}\label{cs:eq:reduction-errors}
 \|z\|_2\le C\|G(v)\|_2,\qquad
 \|G(w)\|_2\le C\|G(v)\|_2 .
\end{equation}
By \eqref{eq:graph-energy-error-bound},
$|\mathcal E(v)-\mathcal E(w)|\le C\|G(v)\|_2^2$.
Combining these bounds with \eqref{cs:eq:finite-lojasiewicz},
and using $2(1-\theta_U)\ge1$, proves
\eqref{cs:eq:gradient-inequality} after restricting to
$\|G(v)\|_2\le1$.
\end{proof}

\begin{lemma}\label{cs:lem:full-convergence}
Assume the hypotheses of Lemma~\ref{cs:lem:compact-orbit}.
Let $t_n\to\infty$ and suppose that $u(t_n,\cdot)\to U$ in $X$.
If \eqref{cs:eq:gradient-inequality} holds at $U$ for some
$\delta_U,C_U>0$ and $\theta_U\in(0,1/2]$, then
\begin{equation}\label{cs:eq:full-convergence}
 \lim_{t\to\infty}\|u(t,\cdot)-U\|_X=0 .
\end{equation}
\end{lemma}

\begin{proof}
Set $\mathcal K=\overline{\{u(t,\cdot):t\ge0\}}^{\,X}$.
By Lemma~\ref{cs:lem:compact-orbit}, $\mathcal K$ is compact and
$U\in\mathcal K\cap\mathcal S_0(J,f)$.  Energy continuity and
\eqref{cs:eq:dissipation} therefore give
\begin{equation}\label{cs:eq:energy-gap}
 e(t):=\mathcal E(u(t,\cdot))-\mathcal E(U)\ge0,\qquad
 \lim_{t\to\infty}e(t)=0,\qquad
 e'(t)=-\|u_t(t,\cdot)\|_2^2 .
\end{equation}
If $e$ vanishes at a finite time, \eqref{cs:eq:energy-gap}
and continuity of $u_t$ give a stationary trajectory thereafter,
whose value is $U$.

Suppose therefore that $e(t)>0$ for all $t\ge0$, and choose
$\delta_U,C_U,\theta_U$ from the assumed gradient inequality.
As long as $\|u(t,\cdot)-U\|_X<\delta_U$,
\eqref{cs:eq:gradient-inequality} and
\eqref{cs:eq:energy-gap} imply
\begin{equation}\label{cs:eq:length}
 -\frac{d}{dt}e(t)^{\theta_U}
 =\theta_U e(t)^{\theta_U-1}\|u_t(t,\cdot)\|_2^2
 \ge\frac{\theta_U}{C_U}\|u_t(t,\cdot)\|_2 .
\end{equation}
Since $\mathcal K$ is compact in $X$ and the inclusion
$X\hookrightarrow L^2$ is continuous and injective, the $X$ and
$L^2$ topologies agree on $\mathcal K$.
Consequently, for $a=\delta_U/2$, the set
$\{v\in\mathcal K:\|v-U\|_X\ge a\}$ is either empty or has a
positive $L^2$ distance $d_a$ from $U$.
In the empty case the gradient inequality holds for the entire
trajectory, and \eqref{cs:eq:length} directly gives its finite
$L^2$ length on $[T,\infty)$ for every $T\ge0$.

In the other case, choose $n$ so large that
\begin{equation}\label{cs:eq:trapping}
 \|u(t_n,\cdot)-U\|_X<a/2,\qquad
 \|u(t_n,\cdot)-U\|_2+
       \frac{C_U}{\theta_U}e(t_n)^{\theta_U}<d_a .
\end{equation}
If the trajectory first reached $X$ distance $a$ from $U$
at time $T>t_n$, integration of \eqref{cs:eq:length} up to $T$ would
give
\[
 \|u(T,\cdot)-U\|_2
 \le \|u(t_n,\cdot)-U\|_2+
       \frac{C_U}{\theta_U}e(t_n)^{\theta_U}<d_a ,
\]
contrary to the definition of $d_a$.
Thus it stays in the $X$ ball of radius $a$, and
\eqref{cs:eq:length} gives
$\int_{t_n}^{\infty}\|u_t(t,\cdot)\|_2\,dt<\infty$.
In either case, finite $L^2$ length gives an $L^2$ limit, identified as
$U$ by the given subsequence.  By compactness of $\mathcal K$, every
sequence of times tending to infinity has an $X$-convergent subsequence;
its $L^2$ limit is $U$.  This proves \eqref{cs:eq:full-convergence}.
\end{proof}

\begin{proof}[Proof of Theorem~\ref{cs:thm:dynamics}]
Uniqueness gives alternative (i) when $u_0\equiv0$.
Assume $u_0\not\equiv0$, with $C_0$ as in
\eqref{cs:eq:initial-convex-hull}.
Spreading gives (ii). Otherwise Lemma~\ref{cs:lem:compact-orbit}
gives $t_n\to\infty$ and $U\in\mathcal S_0(J,f)\cap X$ with
$u(t_n,\cdot)\to U$ in $X$.
For $U\not\equiv0$, either $f$ is real analytic near
$[0,\|U\|_\infty]$ or \eqref{cs:eq:translation-kernel} holds,
by the hypothesis of Theorem~\ref{cs:thm:dynamics}.
In either case, Lemma~\ref{cs:lem:gradient} gives
\eqref{cs:eq:gradient-inequality}. At $U\equiv0$, the lemma applies directly.
Lemma~\ref{cs:lem:full-convergence} then gives $u(t,\cdot)\to U$ in $X$.
For $U\equiv0$ this is (i); for $U\not\equiv0$,
Theorem~\ref{cs:thm:stationary-summary} supplies the unique
pair $(x_*,V)$ and its stated properties. It remains to prove $x_*\in C_0$.
The alternatives are mutually exclusive, since their local limits
differ and $0<U<1$ in (iii).

To prove $x_*\in C_0$ in (iii), fix a unit vector $e$ and
$\lambda>h_0(e)$. By \eqref{cs:eq:initial-convex-hull},
$u_0=0$ almost everywhere on $D_{e,\lambda}$.
The reflected difference satisfies \eqref{cs:eq:reflection-evolution}
with nonnegative initial data, so the comparison argument in
Lemma~\ref{cs:lem:reflection} gives
\[
 u(t,x^{e,\lambda})\ge u(t,x)
 \quad\text{for almost every }x\in D_{e,\lambda}.
\]
Passing to the uniform limit as $t\to\infty$ in alternative (iii), then using
continuity of $U$, yields
\begin{equation}\label{cs:eq:limit-reflection-order}
 U(x^{e,\lambda})\ge U(x)
 \qquad(x\in D_{e,\lambda}).
\end{equation}
If $x_*\cdot e>h_0(e)$, choose $h_0(e)<\lambda<x_*\cdot e$.
Then \eqref{cs:eq:limit-reflection-order} gives
$U(x_*^{e,\lambda})\ge U(x_*)$, contradicting the unique maximum
in Lemma~\ref{cs:lem:stationary-symmetry}.
Hence $x_*\cdot e\le h_0(e)$ for every unit vector $e$.
The representation of $C_0$ as an intersection of supporting half-spaces
gives $x_*\in C_0$.
\end{proof}

\subsection{Existence of a nonzero decaying stationary solution}

\begin{lemma}\label{cs:lem:stationary-existence}
Under \textup{(F1)--(F2)} and \textup{(KM)},
$\mathcal S_0(J,f)\setminus\{0\}$ is nonempty.
\end{lemma}

\begin{proof}
Chmaj and Ren \cite[Theorem~2.1 and Remark~2.10]{ChmajRen1999}
established existence in dimension one. We use
Lemma~\ref{cs:lem:compact-orbit} to give a proof valid for every $d\ge1$.

By the spreading criterion \eqref{eq:measurable-seed-spreading}, choose $R$ so large that the solution
with initial value $\mathbf1_{B_R}$ spreads.
For $a\in[0,1]$, let $u^a$ solve \eqref{eq:main-equation} with
$u^a(0,\cdot)=a\mathbf1_{B_R}$, and define
\[
 E=\{a\in[0,1]:\lim_{t\to\infty}\|u^a(t,\cdot)\|_\infty=0\},\qquad
 P=\{a\in[0,1]:u^a\text{ spreads}\}.
\]
These sets are disjoint, and $0\in E$, $1\in P$.
They are relatively open in $[0,1]$, as follows.

If $a\in E$, choose $T$ with
$\|u^a(T,\cdot)\|_\infty<\vartheta/2$.
By \eqref{cs:eq:flow-dependence}, the same norm is less than
$3\vartheta/4$ for all nearby parameters.
The scalar solution of $q'=f(q)$ with $q(T)=3\vartheta/4$
decreases to zero as $t\to\infty$ by \textup{(F1)--(F2)}.
Comparison with this spatially constant solution puts every such
parameter in $E$.

If $a\in P$, fix $b\in(\vartheta,1)$ and take $L_b$ from
the spreading criterion \eqref{eq:measurable-seed-spreading}.
At some finite time $T$,
$u^a(T,\cdot)>(1+b)/2$ on $B_{L_b}$.
Finite-time dependence in $L^\infty$ gives
$u^{a'}(T,\cdot)\ge b$ on this ball for $a'$ close to $a$.
The spreading criterion \eqref{eq:measurable-seed-spreading} then puts $a'$ in $P$.

Since $E$ and $P$ are disjoint, nonempty, and relatively open,
connectedness gives $a_*\in(0,1)\setminus(E\cup P)$.
Since $a_*\notin P$, Lemma~\ref{cs:lem:compact-orbit} gives a sequence
$t_n\to\infty$ and $U\in\mathcal S_0(J,f)$ such that
$u^{a_*}(t_n,\cdot)\to U$ in $X$.
If $U\equiv0$, then $\|u^{a_*}(t_n,\cdot)\|_\infty<\vartheta$ for
large $n$.  Lemma~\ref{lem:semiflow}, applied from that finite time,
gives $a_*\in E$, a contradiction.
Hence $U\not\equiv0$, which proves the assertion.
\end{proof}

\subsection{Corollaries for a negative reaction integral}\label{subsec:negative-integral}

In this subsection, retain \textup{(F1)--(F2)} and \textup{(KM)},
and suppose that $F(1)<0$.

\begin{proof}[Proof of Corollary~\ref{cs:cor:stationary-negative}]
By Lemma~\ref{lem:duality}, $f^\sharp(s)=-f(1-s)$ satisfies
\textup{(F1)--(F2)} and $F^\sharp(1)>0$.
If $Z\in\mathcal S_0(J,f)$, then $1-Z$ is stationary for
$f^\sharp$ and tends to $1$ at infinity. The second assertion of
Lemma~\ref{cs:lem:stationary-values}, applied to $f^\sharp$, gives
$1-Z\equiv1$, hence $Z\equiv0$.

Under \eqref{cs:eq:reaction-regularity-one},
$(f^\sharp)''(s)=-f''(1-s)$ shows that $f^\sharp$ satisfies
\eqref{cs:eq:reaction-regularity}. Moreover,
$F(\vartheta)<F(1)<F(0)$ and strict decrease of $F$ on
$(0,\vartheta)$ give the unique $\alpha$, and
\eqref{eq:reflected-potential} gives $\beta_{f^\sharp}=1-\alpha$.
For $U\in\mathcal S_1(J,f)\setminus\{1\}$, apply
Theorem~\ref{cs:thm:stationary-summary} to
$1-U\in\mathcal S_0(J,f^\sharp)\setminus\{0\}$.
Writing $1-U(x)=V(|x-x_*|)$ and $P=1-V$ proves
\eqref{cs:eq:stationary-one-summary} and uniqueness of the pair
$(x_*,P)$. The last integral is absolutely convergent since
$|f(U)|\le C(1-U)$ and $1-U\in L^1(\R^d)$.
\end{proof}

\begin{proof}[Proof of Corollary~\ref{cs:cor:dynamics-negative}]
Set $v=1-u$ and $v_0=1-u_0$. By Lemma~\ref{lem:duality},
$v$ solves \eqref{eq:main-equation} with reaction $f^\sharp$,
positive integral $F^\sharp(1)$, and initial datum $v_0$
satisfying \eqref{cs:eq:data}. The preceding proof gives
\eqref{cs:eq:reaction-regularity} for $f^\sharp$.
For $W\in\mathcal S_0(J,f^\sharp)\setminus\{0\}$, set $U=1-W$.
Reflection takes $[0,\|W\|_\infty]$ to $[\inf_{\R^d}U,1]$, while
$(f^\sharp)'(W)=f'(U)$ and $\partial_iW=-\partial_iU$ show that
the linearized operators and the spans of their translation modes
agree. Thus the hypothesis of Theorem~\ref{cs:thm:dynamics} holds
for every such $W$. Applying that theorem to $v$ gives the three
alternatives, since $u-1=-v$ and $u-U=-(v-W)$.
In the third case, Corollary~\ref{cs:cor:stationary-negative}
gives \eqref{cs:eq:stationary-one-summary}, and
\eqref{cs:eq:one-convex-hull} identifies the center set with the
closed convex hull of the support of $v_0(x)\,dx$.
\end{proof}

\section{Threshold radii and critical solutions}\label{sec:ballthresholds}

We retain \textup{(F1)--(F2)} and the solutions $u^R$ defined in
Section~\ref{sec:introduction}. Kernel and additional reaction assumptions
are stated in each result. Inner products without a subscript are
taken in $L^2(\R^d)$.

\subsection{Dependence on the radius}

\begin{lemma}\label{lem:radialmonotone}
Assume (KM).  Let $u_0\in L^\infty(\R^d)$ satisfy $0\le u_0\le1$, and let $u$ be the unique global solution of \eqref{eq:main-equation} with initial datum $u_0$.  If $u_0$ is radial and nonincreasing as a function of $|x|$, then $u(t,\cdot)$ is radial and nonincreasing as a function of $|x|$ for every $t\ge0$.  In particular, $u^R(t,\cdot)$ is radial nonincreasing for every $R>0$ and $t\ge0$.
\end{lemma}

\begin{proof}
By the layer-cake representation and monotonicity of the
overlap of two balls, convolution with $J$ maps bounded nonnegative
radial nonincreasing functions to radial nonincreasing functions; see
\cite[Chapter~3]{LiebLoss2001}.

Set $a=1+\max_{[0,1]}|f'|$ and $Q(s)=(a-1)s+f(s)$.
Then $Q'\ge0$ on $[0,1]$; since $Q(0)=0$ and $Q(1)=a-1$,
we have $0\le Q\le a-1$ there. The integral equation is
\begin{equation}\label{eq:display-3-14}
 u(t,\cdot)=e^{-at}u_0+\int_0^t e^{-a(t-s)}
 \bigl[J*u(s,\cdot)+Q(u(s,\cdot))\bigr]\,ds.
\end{equation}
Starting from $u^{(0)}(t,\cdot)=u_0$, Picard iteration in
\eqref{eq:display-3-14} stays in $[0,1]$, since its integrand lies
between $0$ and $a$. Convolution with $J$, composition with $Q$,
and integration against the nonnegative time weight preserve radial
monotonicity. Because $J*$ and $Q$ are Lipschitz, the integral map
is a contraction on $[0,1]$-valued functions in
$C([0,T];L^\infty(\R^d))$ for small $T$.
Its iterates converge to the solution of Lemma~\ref{lem:semiflow}.
The class of radial nonincreasing functions is closed in
$L^\infty(\R^d)$, so passage to the limit and iteration in time
prove the assertion.
\end{proof}

\begin{lemma}\label{lem:radiuscontinuity}
Assume (K0) and $J\in L^\infty(\R^d)$.  If $0<R<R'$, then for every $t\ge0$ we have
\begin{equation}
 0\le u^{R'}(t,x)-u^R(t,x)
\end{equation}
and
\begin{equation}\label{eq:radius-comparison-bound}
 \|u^{R'}(t,\cdot)-u^R(t,\cdot)\|_\infty
 \le e^{-m_ft}
 \left[1+\|J\|_\infty |B_{R'}\setminus B_R|(e^t-1)\right].
\end{equation}
Consequently, for every $R>0$ and $\varepsilon>0$ there is $T_0>0$ such that, for every fixed $T\ge T_0$, there is $\delta=\delta(R,\varepsilon,T)>0$ such that, for every $R'>0$,
\begin{equation}\label{eq:radius-large-time-continuity}
 |R'-R|<\delta\quad\Longrightarrow\quad
 \|u^{R'}(T,\cdot)-u^R(T,\cdot)\|_\infty<\varepsilon.
\end{equation}

\end{lemma}

\begin{proof}
Assume first that $R'>R$ and set $w=u^{R'}-u^R$.  Order preservation in Lemma~\ref{lem:semiflow} gives $w\ge0$.  The mean value theorem and the slope bound \textup{(F2)} then give
\begin{equation}\label{eq:display-3-19}
 w_t=J*w-q(t,x)w,
\qquad
 q(t,x)=1-\int_0^1 f'\bigl(u^R(t,x)+\theta w(t,x)\bigr)\,d\theta\ge m_f.
\end{equation}
The solution of $z_t=J*z-m_fz$ with
$z(0,\cdot)=w(0,\cdot)=\mathbf1_{B_{R'}\setminus B_R}$ bounds $w$
from above by Lemma~\ref{lem:semiflow} and
\eqref{eq:display-3-19}. Thus
\begin{equation}\label{eq:radius-semigroup-comparison}
 0\le w(t,\cdot)\le e^{-m_ft}e^{tJ*}w(0,\cdot).
\end{equation}
By \eqref{eq:dispersal-semigroup},
\begin{equation}\label{eq:radius-semigroup-series}
 e^{tJ*}w(0,\cdot)=w(0,\cdot)+\sum_{n=1}^\infty\frac{t^n}{n!}J^{*n}*w(0,\cdot).
\end{equation}
Since $0\le w(0,\cdot)\le1$,
$\|w(0,\cdot)\|_1=|B_{R'}\setminus B_R|$, and unit mass gives
$\|J^{*n}\|_\infty\le\|J\|_\infty$ for $n\ge1$,
\eqref{eq:radius-semigroup-series} yields
\begin{equation}\label{eq:radius-semigroup-bound}
 \|e^{tJ*}w(0,\cdot)\|_\infty
 \le1+\|J\|_\infty |B_{R'}\setminus B_R|(e^t-1).
\end{equation}
Combining \eqref{eq:radius-semigroup-comparison} with \eqref{eq:radius-semigroup-bound} proves \eqref{eq:radius-comparison-bound}.  The case $R'<R$ follows by interchanging the two radii.

Fix $R>0$ and $\varepsilon>0$. Choose $T_0$ with
$e^{-m_fT_0}<\varepsilon/2$, then fix $T\ge T_0$.
Since $|B_{R'}\triangle B_R|\to0$ as $R'\to R$, choose
$\delta=\delta(R,\varepsilon,T)>0$ making the second term in
\eqref{eq:radius-comparison-bound} less than $\varepsilon/2$,
interchanging the radii when $R'<R$. This proves
\eqref{eq:radius-large-time-continuity}. The time must be chosen
first, since $\|\mathbf1_{B_{R'}}-\mathbf1_{B_R}\|_\infty=1$
for $R'\ne R$.
\end{proof}

Under \textup{(KR)} and $F(1)>0$, define $R_{\rm ext}=\sup\{R>0:u^R\text{ becomes extinct}\}$ and $R_{\rm prop}=\inf\{R>0:u^R\text{ spreads}\}$.  We write $R_{\rm ext}(f,J)$ and $R_{\rm prop}(f,J)$ when the dependence on $f$ and $J$ is needed.

\begin{proposition}\label{prop:radialinterval}
Assume (KR), \textup{(F1)--(F2)}, and $F(1)>0$.  Define
\begin{equation}\label{eq:extinction-set}
 \mathcal E_d=\left\{R>0:\ \|u^R(t,\cdot)\|_\infty\to0\right\}
\end{equation}
and
\begin{equation}\label{eq:propagation-set}
 \mathcal P_d=\left\{R>0:\ u^R(t,\cdot)\to1\ \text{locally uniformly in }\R^d\right\}.
\end{equation}
Then there exist finite radii
\begin{equation}
 0<R_{\rm ext}\le R_{\rm prop}<\infty
\end{equation}
such that
\begin{equation}\label{eq:threshold-open-intervals}
 \mathcal E_d=(0,R_{\rm ext}),\qquad
 \mathcal P_d=(R_{\rm prop},\infty).
\end{equation}
Consequently the set of radii $R$ for which $u^R$ neither becomes extinct nor spreads is the nonempty closed interval
\begin{equation}\label{eq:transition-interval}
 [R_{\rm ext},R_{\rm prop}].
\end{equation}
Moreover, with $T_f$, $M_1(J)$ and $\omega_d$ as in \eqref{eq:threshold-constants}, one has
\begin{equation}\label{eq:extinction-radius-bounds}
 \left[\frac{1}{\omega_d\|J\|_\infty(e^{T_f}-1)}\right]^{1/d}
 \le R_{\rm ext}
 \le \frac{dM_1(J)}{4F(1)}.
\end{equation}

\end{proposition}

\begin{proof}
Comparison in Lemma~\ref{lem:semiflow}, applied to $\mathbf1_{B_{R_1}}\le\mathbf1_{B_{R_2}}$ for $R_1<R_2$, shows that $\mathcal E_d$ is downward closed and $\mathcal P_d$ is upward closed.  Lemma~\ref{lem:smallseed} gives $\mathcal E_d\ne\varnothing$, since $|B_R|\to0$ as $R\downarrow0$, while the spreading criterion \eqref{eq:measurable-seed-spreading} gives $\mathcal P_d\ne\varnothing$.  The definitions \eqref{eq:extinction-set} and \eqref{eq:propagation-set} show that the two sets are disjoint.

To prove openness of the extinction set, let $R\in\mathcal E_d$ and choose $T$ so large that
\begin{equation}\label{eq:display-3-31}
 \|u^R(T,\cdot)\|_\infty<\frac\vartheta4,\qquad e^{-m_fT}<\frac\vartheta4.
\end{equation}
By Lemma~\ref{lem:radiuscontinuity}, applied at this $T$ with $\varepsilon=\vartheta/2$, if $R'>R$ is sufficiently close to $R$, then
\begin{equation}\label{eq:display-3-32}
 \|u^{R'}(T,\cdot)-u^R(T,\cdot)\|_\infty<\frac\vartheta2.
\end{equation}
By \eqref{eq:display-3-31}--\eqref{eq:display-3-32},
$\|u^{R'}(T,\cdot)\|_\infty<\vartheta$.
Lemma~\ref{lem:semiflow}, applied from $T$, gives
$R'\in\mathcal E_d$. Downward closure then proves that
$\mathcal E_d$ is open.

For $R\in\mathcal P_d$, fix $\eta\in(0,(1-\vartheta)/4)$
and take $L=L(\eta)$ from \eqref{eq:measurable-seed-spreading}.
By convergence to $1$ in $L^\infty(B_L)$, for large $T$,
\begin{equation}\label{eq:display-3-33}
 u^R(T,x)\ge \vartheta+2\eta\qquad (|x|\le L).
\end{equation}
Increase $T$, if necessary, so that $e^{-m_fT}<\eta/2$.  Applying Lemma~\ref{lem:radiuscontinuity} at this $T$ with $\varepsilon=\eta$, and using \eqref{eq:display-3-33}, gives, for $R'<R$ sufficiently close to $R$,
\begin{equation}\label{eq:display-3-34}
 u^{R'}(T,x)\ge \vartheta+\eta\qquad (|x|\le L).
\end{equation}
By the semiflow property in Lemma~\ref{lem:semiflow},
\eqref{eq:measurable-seed-spreading} applies to
\eqref{eq:display-3-34} from time $T$, so
$u^{R'}(T+s,\cdot)\to1$ locally uniformly as $s\to\infty$.
Thus $R'\in\mathcal P_d$; upward closure proves that
$\mathcal P_d$ is open.

By \eqref{eq:extinction-set} and \eqref{eq:propagation-set}, the threshold radii defined above satisfy
\begin{equation}\label{eq:threshold-radii-definition}
 R_{\rm ext}=\sup\mathcal E_d,\qquad
 R_{\rm prop}=\inf\mathcal P_d.
\end{equation}
Nonemptiness, disjointness and the respective downward and
upward closure give $0<R_{\rm ext}\le R_{\rm prop}<\infty$ in
\eqref{eq:threshold-radii-definition}. Openness yields
\eqref{eq:threshold-open-intervals}, whose complement in
$(0,\infty)$ is \eqref{eq:transition-interval}.

For the lower bound on $R_{\rm ext}$, recall from \eqref{eq:small-mass-constants} that $\delta_f(J)=[\|J\|_\infty(e^{T_f}-1)]^{-1}$.  Lemma~\ref{lem:smallseed} applies because
$|B_R|=\omega_dR^d<\delta_f(J)$
whenever $R<[\omega_d\|J\|_\infty(e^{T_f}-1)]^{-1/d}$.  For the upper bound, the negative-energy criterion \eqref{eq:negative-ball-energy} in Lemma~\ref{lem:ballenergy} shows that no radius larger than $dM_1(J)/(4F(1))$ can belong to $\mathcal E_d$, which gives the upper bound for $R_{\rm ext}$.
\end{proof}

For each $R>0$, let $\omega(u^R)$ denote the set of all functions $U\in L^\infty_{\rm loc}(\R^d)$ for which $u^R(t_n,\cdot)\to U$ in $L^\infty(B_r)$ for every $r>0$ along some sequence $t_n\to\infty$.

\begin{proposition}\label{prop:transitionomega}
Assume (KM), \textup{(F1)--(F2)}, $F(1)>0$, and let
\begin{equation}
 R\in[R_{\rm ext},R_{\rm prop}].
\end{equation}
Then $\{u^R(t,\cdot):t\ge0\}$ is bounded in $L^2(\R^d)$ and
\begin{equation}\label{eq:transition-time-derivative-decay}
 \|u_t^R(t,\cdot)\|_2\to0\qquad\text{as }t\to\infty.
\end{equation}
Moreover, $\omega(u^R)$, with the topology of convergence in $L^\infty(B_r)$ for every $r>0$, is a nonempty compact connected set of nonzero stationary solutions $U$ of \eqref{eq:main-equation} satisfying
\begin{equation}
 0\le U\le1,\qquad U\in L^2(\R^d),
\end{equation}
with $U$ radial nonincreasing and $U(x)\to0$ as $|x|\to\infty$.
\end{proposition}

\begin{proof}
By Lemma~\ref{lem:radialmonotone}, $u^R(t,\cdot)$ is radial
nonincreasing. Take $b=(\vartheta+\beta)/2$, with $\beta$ from
Lemma~\ref{lem:nonlinearity}, and set $\eta=b-\vartheta$ and
$L=L(\eta)$ as in \eqref{eq:measurable-seed-spreading}.
Proposition~\ref{prop:radialinterval} gives
$R\notin\mathcal E_d\cup\mathcal P_d$; in particular,
\begin{equation}\label{eq:transition-exterior-height}
 u^R(t,x)\le b\qquad\text{for a.e. }|x|>L,\quad t\ge0.
\end{equation}
Indeed, if $u^R(t_0,\cdot)>b$ on a positive-measure set outside
$B_L$, radial monotonicity gives $u^R(t_0,\cdot)\ge b$ almost
everywhere on $B_L$. Applying \eqref{eq:measurable-seed-spreading}
from $t_0$ with initial datum $u^R(t_0,\cdot)$ would give
$R\in\mathcal P_d$, a contradiction.

By the potential signs \eqref{eq:potential-signs}, $F(s)<0$ for $0<s<\beta$, so the number
\begin{equation}
 c_b=\min_{0\le s\le b}\frac{-F(s)}{s^2},
\end{equation}
with the quotient at zero defined by its limit $-f'(0)/2$, is positive.  Lemma~\ref{lem:nonlinearity} also gives $F(s)\le F(1)$ on $[0,1]$.  Using these bounds and \eqref{eq:transition-exterior-height}, and dropping the nonnegative interaction part of the energy \eqref{eq:energy-definition}, we obtain
\begin{equation}\label{eq:transition-energy-lower-bound}
 \mathcal E[u^R(t,\cdot)]
 \ge-F(1)|B_L|+c_b\int_{\R^d\setminus B_L}u^R(t,x)^2\,dx.
\end{equation}
The energy identity \eqref{eq:energy-dissipation} gives $\mathcal E[u^R(t,\cdot)]\le \mathcal E[u^R(0,\cdot)]$; combining this with \eqref{eq:transition-energy-lower-bound} yields the explicit bound
\begin{equation}\label{eq:transition-l2-exterior-bound}
 \int_{\R^d\setminus B_L}u^R(t,x)^2\,dx
 \le
 \frac{\mathcal E[u^R(0,\cdot)]+F(1)|B_L|}{c_b}.
\end{equation}
By \eqref{eq:transition-l2-exterior-bound} and $0\le u^R\le1$
on $B_L$, $\|u^R(t,\cdot)\|_2$ is uniformly bounded.
Moreover, \eqref{eq:transition-energy-lower-bound} bounds the energy
below by $-F(1)|B_L|$. Hence \eqref{eq:energy-dissipation} gives
\begin{equation}\label{eq:transition-dissipation-integral}
 \int_0^\infty\|u_t^R(t,\cdot)\|_2^2\,dt<\infty.
\end{equation}
Set $v=u_t^R$.  Lemma~\ref{lem:semiflow} gives $u^R\in C^1([0,\infty);L^2(\R^d)\cap L^\infty(\R^d))$.  Let $\omega_{f'}$ be the modulus of continuity of $f'$ on $[0,1]$, defined by $\omega_{f'}(\delta)=\sup\{|f'(a)-f'(b)|:a,b\in[0,1],\ |a-b|\le\delta\}$.  By the integral form of Taylor's formula, for $a,b\in L^2(\R^d)\cap L^\infty(\R^d)$ with $0\le a,a+b\le1$,
\[
 \|f(a+b)-f(a)-f'(a)b\|_2
 \le\omega_{f'}(\|b\|_\infty)\|b\|_2.
\]
Since $\omega_{f'}(\delta)\to0$ as $\delta\downarrow0$, this estimate and the
$C^1$ time regularity of $u^R$ in $L^2(\R^d)\cap L^\infty(\R^d)$
justify the $L^2(\R^d)$ chain rule. Differentiating
\eqref{eq:main-equation} gives
\begin{equation}\label{eq:transition-linearized-evolution}
 v_t=J*v-v+f'(u^R)v.
\end{equation}
Since $f(0)=0$, $|f(s)|\le L_fs$ on $[0,1]$, with $L_f$
from Lemma~\ref{lem:semiflow}. Here $C>0$ is independent of $t$
and may increase between estimates. The uniform $L^2(\R^d)$ bound
on $u^R$ gives
\begin{equation}\label{eq:display-3-45}
 \|v(t,\cdot)\|_2
 \le \|J*u^R(t,\cdot)\|_2+\|u^R(t,\cdot)\|_2+\|f(u^R(t,\cdot))\|_2
 \le C.
\end{equation}
By \eqref{eq:transition-linearized-evolution}, the bound \eqref{eq:display-3-45} and boundedness of $f'$ on $[0,1]$ give $\|v_t(t,\cdot)\|_2\le C$.  Consequently
\begin{equation}
 \left|\frac d{dt}\|v(t,\cdot)\|_2^2\right|
 \le2\|v(t,\cdot)\|_2\|v_t(t,\cdot)\|_2\le C.
\end{equation}
Thus $t\mapsto\|v(t,\cdot)\|_2^2$ is uniformly continuous and,
by \eqref{eq:transition-dissipation-integral}, integrable.
If it did not tend to zero, uniform continuity would give disjoint
intervals of a fixed positive length with a fixed positive lower
bound, contradicting integrability. This proves
\eqref{eq:transition-time-derivative-decay}.

Let $t_n\to\infty$.  Lemma~\ref{lem:asymcompact} gives a subsequence, still indexed by $n$, and a bounded uniformly continuous function $U$ such that
\begin{equation}\label{eq:transition-local-limit}
 u^R(t_n,\cdot)\to U
\end{equation}
locally uniformly.  By \eqref{eq:transition-local-limit}, a further subsequence converges to $U$ almost everywhere.  Fatou's lemma and the uniform $L^2(\R^d)$ bound therefore give $U\in L^2(\R^d)$.  For every $r,A>0$,
\begin{equation}\label{eq:transition-convolution-limit-bound}
 \|J*u^R(t_n,\cdot)-J*U\|_{L^\infty(B_r)}
 \le \|u^R(t_n,\cdot)-U\|_{L^\infty(B_{r+A})}
      +\int_{|z|>A}J(z)\,dz.
\end{equation}
In \eqref{eq:transition-convolution-limit-bound}, first letting $n\to\infty$ and then $A\to\infty$ gives local uniform convergence of the convolutions.  Thus
\begin{equation}\label{eq:transition-stationary-residual-limit}
 J*u^R(t_n,\cdot)-u^R(t_n,\cdot)+f(u^R(t_n,\cdot))
 \to J*U-U+f(U)
\end{equation}
locally uniformly.  By \eqref{eq:main-equation} and
\eqref{eq:transition-time-derivative-decay}, the left-hand side of
\eqref{eq:transition-stationary-residual-limit} tends to zero in
$L^2(\R^d)$. Its locally uniform limit is therefore zero almost
everywhere; continuity gives
\begin{equation}
 J*U-U+f(U)=0\qquad\text{on }\R^d,
\end{equation}
and $U$ is a stationary solution of \eqref{eq:main-equation}.

By \eqref{eq:transition-local-limit}, $U$ is radial
nonincreasing; since $U\in L^2(\R^d)$, its continuous
representative tends to zero at infinity. If $U\equiv0$,
\eqref{eq:transition-local-limit} gives
$\|u^R(t_n,\cdot)\|_{L^\infty(B_1)}\to0$.
Radial monotonicity identifies this essential supremum with
$\|u^R(t_n,\cdot)\|_\infty$. For large $n$, it is below
$\vartheta$, so Lemma~\ref{lem:semiflow}, applied from $t_n$,
would give $R\in\mathcal E_d$, contradicting
Proposition~\ref{prop:radialinterval}.

On $L^\infty_{\rm loc}(\R^d)$, with functions identified almost everywhere, we use the metric
\begin{equation}\label{eq:local-essential-metric}
 d_{\rm loc}(v,w)=\sum_{k=1}^\infty 2^{-k}\min\{1,\|v-w\|_{L^\infty(B_k)}\}.
\end{equation}
Completeness of $L^\infty(B_k)$ gives compatible limits on
nested balls, so \eqref{eq:local-essential-metric} is complete.
Lemma~\ref{lem:asymcompact} makes $\omega(u^R)$ nonempty.
For each $T>0$, any sequence of times in $[T,\infty)$ has a
subsequence tending to a finite time or to infinity.
Lemma~\ref{lem:semiflow} gives convergence in the first case, and
Lemma~\ref{lem:asymcompact} gives local convergence in the second.
Thus the closure of $\{u^R(t,\cdot):t\ge T\}$ is compact.
It is connected as the closure of the continuous image of
$[T,\infty)$. Intersecting these nested nonempty compact connected
sets over $T>0$ gives $\omega(u^R)$, hence its compactness and
connectedness.
\end{proof}

\subsection{Compactness for radii between the thresholds}\label{sec:convergence}

\begin{lemma}\label{lem:transitiontight}
Assume (KM), \textup{(F1)--(F2)}, $F(1)>0$, and let $R\in[R_{\rm ext},R_{\rm prop}]$.  Then there is a constant $C_R>0$ such that
\begin{equation}
 \sup_{t\ge0}\bigl(\|u^R(t,\cdot)\|_1+\|u^R(t,\cdot)\|_2\bigr)\le C_R.
\end{equation}
Moreover, for every $r\ge1$,
\begin{equation}\label{eq:transition-pointwise-tail}
 \sup_{t\ge0}\esssup_{|x|\ge r}u^R(t,x)\le C_R r^{-d}
\end{equation}
and
\begin{equation}\label{eq:transition-l2-tail}
 \sup_{t\ge0}\int_{|x|\ge r}u^R(t,x)^2\,dx\le C_R r^{-d}.
\end{equation}
Consequently, if $t_n\to\infty$, then a subsequence of $u^R(t_n,\cdot)$ converges in $X$
to a nonzero radial stationary solution of \eqref{eq:main-equation} tending to zero at spatial infinity.
\end{lemma}

\begin{proof}
By Proposition~\ref{prop:transitionomega}, there is a constant $C_2>0$, independent of $t$, such that
\begin{equation}\label{eq:transition-uniform-l2}
 \|u^R(t,\cdot)\|_2\le C_2\qquad (t\ge0).
\end{equation}
By Lemma~\ref{lem:radialmonotone}, write
$u^R(t,x)=\mathsf u^R(t,|x|)$ with $\mathsf u^R(t,\cdot)$
nonincreasing. Radial monotonicity and
\eqref{eq:transition-uniform-l2} give, for $r>0$,
\begin{equation}\label{eq:radial-l2-pointwise-bound}
 |B_r|\,{\mathsf u^R(t,r)}^2
 \le \int_{B_r}u^R(t,x)^2\,dx
 \le C_2^2.
\end{equation}
In this proof, constants denoted by $C$ are independent of $t$ and may increase between estimates.  It follows from \eqref{eq:radial-l2-pointwise-bound} that $\esssup_{|x|\ge r}u^R(t,x)\le Cr^{-d/2}$ uniformly in $t$, so we may choose $r_0$ large enough that
\begin{equation}\label{eq:transition-small-exterior-height}
 u^R(t,x)\le \frac\vartheta2
 \qquad\text{for a.e. }|x|\ge r_0,\ t\ge0.
\end{equation}
Since $f(s)<0$ on $(0,\vartheta)$ and $f'(0)<0$, the number
$\gamma_f:=\min_{0\le s\le\vartheta/2}\frac{-f(s)}s$,
with the quotient at $s=0$ interpreted as $-f'(0)$, is positive.  Hence
\begin{equation}\label{eq:reaction-exterior-decay}
 f(s)\le-\gamma_f s\qquad(0\le s\le\vartheta/2).
\end{equation}
Since $u_0=\mathbf1_{B_R}\in L^1(\R^d)$, Lemma~\ref{lem:semiflow} gives $u^R\in C^1([0,\infty);L^1(\R^d))$.  Since $u^R\ge0$ and
\begin{equation}\label{eq:display-4-09}
 \int_{\R^d}(J*u^R-u^R)\,dx=0,
\end{equation}
integration of \eqref{eq:main-equation} and the cancellation in \eqref{eq:display-4-09} yield
\begin{equation}\label{eq:transition-mass-derivative}
 \frac d{dt}\|u^R(t,\cdot)\|_1=\int_{\R^d}f(u^R(t,x))\,dx.
\end{equation}
Set $M(t)=\|u^R(t,\cdot)\|_1$, $f_+(s)=\max\{f(s),0\}$,
$C_0=|B_{r_0}|\max_{[0,1]}f_+$ and
$C_1=C_0+\gamma_f|B_{r_0}|$.
The integral over $B_{r_0}$ is at most $C_0$; by
\eqref{eq:transition-small-exterior-height} and
\eqref{eq:reaction-exterior-decay}, the exterior contribution is
at most $-\gamma_f\int_{\R^d\setminus B_{r_0}}u^R(t,x)\,dx$.
Thus \eqref{eq:transition-mass-derivative} yields
\begin{equation}\label{eq:transition-mass-inequality}
 M'(t)
 \le C_0-\gamma_f\int_{\R^d\setminus B_{r_0}}u^R(t,x)\,dx
 \le C_1-\gamma_f M(t).
\end{equation}
Comparing \eqref{eq:transition-mass-inequality} with
$y'=C_1-\gamma_fy$, $y(0)=M(0)=|B_R|$, gives
\begin{equation}\label{eq:transition-uniform-l1}
 \sup_{t\ge0}M(t)\le C_1':=\max\{|B_R|,C_1/\gamma_f\}.
\end{equation}

By radial monotonicity and the mass bound \eqref{eq:transition-uniform-l1}, for every $r>0$ and $t\ge0$,
\begin{equation}\label{eq:radial-l1-pointwise-bound}
 |B_r|\,{\mathsf u^R(t,r)}
 \le\int_{B_r}u^R(t,x)\,dx
 \le C_1'.
\end{equation}
The estimate \eqref{eq:radial-l1-pointwise-bound} gives $\esssup_{|x|\ge r}u^R(t,x)\le Cr^{-d}$; combining it with \eqref{eq:transition-uniform-l1}, we obtain
\begin{equation}\label{eq:transition-tail-from-mass}
 \int_{|x|\ge r}u^R(t,x)^2\,dx
 \le \esssup_{|x|\ge r}u^R(t,x)\,\|u^R(t,\cdot)\|_1
 \le Cr^{-d}.
\end{equation}

For $t_n\to\infty$, Lemma~\ref{lem:asymcompact} gives a
subsequence converging locally uniformly to $U$.
By Proposition~\ref{prop:transitionomega}, $U$ is a nonzero radial
nonincreasing stationary solution tending to zero at infinity;
use its uniformly continuous representative from
Lemma~\ref{lem:steadyregularity}. For fixed $r\ge1$ and $M>r$,
local convergence passes the bound in
\eqref{eq:radial-l1-pointwise-bound} to $U$ on $B_M\setminus B_r$.
Letting $M\to\infty$ gives
\begin{equation}\label{eq:display-4-15}
 \esssup_{|x|\ge r}U(x)\le C_Rr^{-d}.
\end{equation}
Combining \eqref{eq:display-4-15} with \eqref{eq:transition-pointwise-tail} for $u^R(t_n,\cdot)$ gives
\begin{equation}\label{eq:global-uniform-convergence-tail}
 \esssup_{|x|\ge r}|u^R(t_n,x)-U(x)|\le2C_Rr^{-d}.
\end{equation}
Local convergence and \eqref{eq:global-uniform-convergence-tail},
first with $n\to\infty$ and then $r\to\infty$, give convergence in
$L^\infty(\R^d)$. Along a further almost-everywhere convergent
subsequence, Fatou and \eqref{eq:transition-tail-from-mass} give
$\int_{|x|\ge r}U(x)^2\,dx\le C_Rr^{-d}$.
Combining the two $L^2$ tail bounds yields
\begin{equation}
 \int_{|x|\ge r}|u^R(t_n,x)-U(x)|^2\,dx\le 4C_Rr^{-d}.
\end{equation}
For each fixed $r\ge1$,
\[
 \|u^R(t_n,\cdot)-U\|_2^2
 \le |B_r|\,\|u^R(t_n,\cdot)-U\|_{L^\infty(B_r)}^2
       +4C_Rr^{-d}.
\]
Taking first $n\to\infty$ and then $r\to\infty$ proves convergence in $L^2(\R^d)$.  Together with the $L^\infty(\R^d)$ convergence, this proves convergence in $X$.  
\end{proof}

\subsection{Convergence for each radius between the thresholds}

\begin{proposition}\label{prop:transitionconverges}
Assume (KM), \textup{(F1)--(F2)}, \textup{(F4)}, $F(1)>0$, and let $R\in[R_{\rm ext},R_{\rm prop}]$.  Then there is a radial nonincreasing
$U_R\in\mathcal S_0(J,f)\setminus\{0\}$ such that
\begin{equation}\label{eq:whole-orbit-convergence}
 \|u^R(t,\cdot)-U_R\|_2+\|u^R(t,\cdot)-U_R\|_\infty\to0
 \qquad\text{as }t\to\infty.
\end{equation}
Moreover, for some $T>0$,
\begin{equation}
 \int_T^\infty
 \bigl(\|u_t^R(t,\cdot)\|_2+\|u_t^R(t,\cdot)\|_\infty\bigr)\,dt<\infty.
\end{equation}
\end{proposition}

\begin{proof}
\smallskip
\noindent\textit{Step 1. A stationary limit and decay of the time derivative.}
Let $u=u^R$ be the solution for the fixed radius $R\in[R_{\rm ext},R_{\rm prop}]$ in the statement.  By Proposition~\ref{prop:transitionomega},
$\|u_t(t,\cdot)\|_2\to0$.
By Lemma~\ref{lem:transitiontight}, there are $U\in X$ and a sequence $t_n\to\infty$ for which
\begin{equation}\label{eq:orbit-subsequence-x-convergence}
 \|u(t_n,\cdot)-U\|_X\to0.
\end{equation}
By Proposition~\ref{prop:transitionomega}, $U$ is a nonzero
radial nonincreasing stationary solution of \eqref{eq:main-equation},
with $0\le U\le1$ and $U(x)\to0$ as $|x|\to\infty$.
By \eqref{eq:energy-dissipation} and
\eqref{eq:transition-energy-lower-bound}, define the finite limit
\begin{equation}\label{eq:limiting-energy}
 E_\infty=\lim_{t\to\infty}\mathcal E[u(t,\cdot)].
\end{equation}
Continuity of $\mathcal E$ near $U$ from Step~1 of
Lemma~\ref{lem:lojasiewicz}, together with
\eqref{eq:orbit-subsequence-x-convergence}--\eqref{eq:limiting-energy},
gives $\mathcal E[U]=E_\infty$.

Put $v=u_t$.  For every $T>0$, the map $w\mapsto J*w-w+f(w)$ is continuously differentiable on an open subset of $X$ containing $\{u(t,\cdot):0\le t\le T\}$.  Thus $u$ is $C^2$ in time with values in $X$ on bounded time intervals, and differentiation gives
\begin{equation}\label{eq:time-derivative-equation}
 v_t=J*v-q(t,x)v,
 \qquad
 q(t,x)=1-f'(u(t,x))\ge m_f.
\end{equation}
Using representatives absolutely continuous in time, as in
the proof of Lemma~\ref{lem:modulus}, apply an integrating factor to
\eqref{eq:time-derivative-equation} for almost every $x$.
Since $q\ge m_f$ and $\|J*v\|_\infty\le\|J\|_2\|v\|_2$,
for $0\le s\le t$,
\begin{equation}\label{eq:time-derivative-linfty-bound}
 \|v(t,\cdot)\|_\infty
 \le e^{-m_f(t-s)}\|v(s,\cdot)\|_\infty
 +\|J\|_2\int_s^t e^{-m_f(t-\tau)}\|v(\tau,\cdot)\|_2\,d\tau.
\end{equation}
Since $v=J*u-u+f(u)$ and $0\le u\le1$, $\|v(t,\cdot)\|_\infty$ is uniformly bounded, while Proposition~\ref{prop:transitionomega} gives $\|v(t,\cdot)\|_2\to0$.  Taking $s=t-1$ for $t\ge1$ and writing
$L=\limsup_{t\to\infty}\|v(t,\cdot)\|_\infty$,
the integral term tends to zero because
\[
 \int_{t-1}^t e^{-m_f(t-\tau)}\|v(\tau,\cdot)\|_2\,d\tau
 \le \frac{1-e^{-m_f}}{m_f}
       \sup_{\tau\ge t-1}\|v(\tau,\cdot)\|_2\to0.
\]
Hence \eqref{eq:time-derivative-linfty-bound} gives $L\le e^{-m_f}L$.  Since $m_f>0$, the inequality forces $L=0$, that is, $\|u_t(t,\cdot)\|_\infty\to0$.

\smallskip
\noindent\textit{Step 2. The orbit remains near the stationary limit.}
Let $\rho,C,\theta$ be given by Lemma~\ref{lem:lojasiewicz} at $U$, and set
$e(t)=\mathcal E[u(t,\cdot)]-E_\infty\ge0$.
If $e(t_0)=0$ at some finite time, then
\begin{equation}
 0=e(t_0)=\int_{t_0}^\infty\|u_t(\tau,\cdot)\|_2^2\,d\tau,
\end{equation}
so continuous differentiability in $L^2(\R^d)$ makes $u$
stationary for $t\ge t_0$. By
\eqref{eq:orbit-subsequence-x-convergence}, its value is $U$,
and both conclusions follow with $U_R=U$ and $T>t_0$.
Hence assume $e(t)>0$ for every $t\ge0$.

By \eqref{eq:orbit-subsequence-x-convergence},
\eqref{eq:limiting-energy} and Step~1, choose $n$ so large that
$\|u(t_n,\cdot)-U\|_X<\rho/4$ and
\begin{equation}\label{eq:orbit-trapping-smallness}
 \frac{\|u_t(t_n,\cdot)\|_\infty}{m_f}
 +\left(1+\frac{\|J\|_2}{m_f}\right)
   \frac C\theta e(t_n)^\theta<\frac\rho2.
\end{equation}
As long as $\|u(t,\cdot)-U\|_X<\rho$, the gradient inequality \eqref{eq:local-gradient-inequality} and the energy identity \eqref{eq:energy-differential} give
\begin{equation}\label{eq:energy-gap-gradient-and-dissipation}
 e(t)^{1-\theta}\le C\|u_t(t,\cdot)\|_2,
 \qquad
 e'(t)=-\|u_t(t,\cdot)\|_2^2.
\end{equation}
It follows from \eqref{eq:energy-gap-gradient-and-dissipation} that
\begin{equation}\label{eq:energy-gap-power-decay}
 -\frac d{dt}e(t)^\theta
 =\theta e(t)^{\theta-1}\|u_t(t,\cdot)\|_2^2
 \ge \frac{\theta}{C}\|u_t(t,\cdot)\|_2.
\end{equation}
Integrating \eqref{eq:energy-gap-power-decay} from $t_n$ to $t$, while $u(\tau,\cdot)$ remains in $\{v\in X:\|v-U\|_X<\rho\}$ for $t_n\le\tau\le t$, yields
\begin{equation}\label{eq:orbit-length-l2-bound}
 \int_{t_n}^t\|u_t(\tau,\cdot)\|_2\,d\tau
 \le \frac{C}{\theta}e(t_n)^\theta.
\end{equation}
Taking $s=t_n$ in \eqref{eq:time-derivative-linfty-bound}, integrating over $[t_n,t]$, and applying Fubini's theorem gives
\begin{equation}\label{eq:orbit-length-linfty-bound}
 \int_{t_n}^t\|u_t(\tau,\cdot)\|_\infty\,d\tau
 \le \frac1{m_f}\|u_t(t_n,\cdot)\|_\infty
 +\frac{\|J\|_2}{m_f}
   \int_{t_n}^t\|u_t(\tau,\cdot)\|_2\,d\tau.
\end{equation}
By \eqref{eq:orbit-length-l2-bound},
\eqref{eq:orbit-length-linfty-bound} and
\eqref{eq:orbit-trapping-smallness}, the length in $X$ is less than
$\rho/2$ as long as the orbit stays in the $\rho$-neighborhood.
The $X$ continuity from Lemma~\ref{lem:semiflow} permits passage
to a hypothetical first exit time $T>t_n$, where
$\|u(T,\cdot)-U\|_X=\rho$. Letting $t\uparrow T$ in the two
length estimates gives
\begin{equation}
 \|u(T,\cdot)-U\|_X
 \le \|u(t_n,\cdot)-U\|_X
   +\int_{t_n}^{T}\bigl(\|u_t(\tau,\cdot)\|_2+\|u_t(\tau,\cdot)\|_\infty\bigr)\,d\tau
 <\frac{3\rho}{4},
\end{equation}
a contradiction.  Thus $\|u(t,\cdot)-U\|_X<\rho$ for every $t\ge t_n$, and \eqref{eq:orbit-length-l2-bound} and \eqref{eq:orbit-length-linfty-bound} hold for all $t\ge t_n$.

\smallskip
\noindent\textit{Step 3. Finite length and convergence.}
Letting $t\to\infty$ in \eqref{eq:orbit-length-l2-bound} and
\eqref{eq:orbit-length-linfty-bound} gives
$u_t\in L^1((t_n,\infty);X)$. Thus the orbit is Cauchy in the
complete space $X$; \eqref{eq:orbit-subsequence-x-convergence}
identifies its limit as $U$. Taking $U_R=U$ and $T=t_n$ proves
both conclusions.
\end{proof}

\subsection{Radial instability and critical radii}

Chmaj and Ren \cite[Theorem~2.1]{ChmajRen2002} used differentiation of the stationary equation and comparison with the absolute value of a derivative in a bounded-interval instability argument.  We prove the existence of a positive eigenvalue on $\R^d$ and of a positive eigenvalue for the operator on a sufficiently large ball.

\begin{lemma}\label{lem:radialunstable}
Assume (KR) and \textup{(F3)}, and let $U$ be as in \eqref{eq:display-4-18}, nonconstant, radial and nonincreasing.  Then the operator $\mathcal L_U$ defined in \eqref{eq:main-linearization} has a positive eigenvalue with a nonzero nonnegative radial eigenfunction in $L^2(\R^d)$. Moreover, there exist $L>0$, $\lambda_L>0$, and a nonzero radial function $\psi_L\ge0$ in $L^2(B_L)$ such that
\begin{equation}\label{eq:finite-ball-eigenfunction}
 \int_{B_L}J(x-y)\psi_L(y)\,dy-\psi_L(x)+f'(U(x))\psi_L(x)
 =\lambda_L\psi_L(x)
\end{equation}
for almost every $x\in B_L$.
\end{lemma}

\begin{proof}
\smallskip
\noindent\textit{Step 1. A sign-changing eigenfunction for eigenvalue zero.}
Since $U\in L^2(\R^d)$ and $\partial_iJ\in L^1(\R^d)$, Young's inequality gives $J*U\in H^1(\R^d)$ with $\partial_i(J*U)=(\partial_iJ)*U\in L^2(\R^d)$.  Recall that $H(s)=s-f(s)$.  By Lemma~\ref{lem:nonlinearity}, $H'\ge m_f>0$, so $H^{-1}$ is Lipschitz on $H([0,1])$.  Applying the Sobolev chain rule (see \cite[Chapter~9]{Brezis2011}) to the stationary identity \eqref{eq:stationary-inverse}, written as $U=H^{-1}(J*U)$, gives $U\in H^1(\R^d)$ and
\begin{equation}\label{eq:stationary-weak-derivative}
 H'(U)\partial_iU=(\partial_iJ)*U
\end{equation}
almost everywhere.  Since $J*\partial_iU=(\partial_iJ)*U$ in $L^2(\R^d)$, \eqref{eq:stationary-weak-derivative} yields
\begin{equation}\label{eq:translation-zero-mode}
 \mathcal L_U(\partial_iU)=0.
\end{equation}
Since $U$ is nonconstant, choose a coordinate $i$ for which $\partial_iU\not\equiv0$, and set $v=\partial_iU$.  Radiality makes $v$ odd with respect to the hyperplane $x_i=0$, so $v$ changes sign.

\smallskip
\noindent\textit{Step 2. Positivity of the spectral bound.}
Let $s_U:=\sup\sigma(\mathcal L_U)$,
where $\sigma(\mathcal L_U)$ denotes the spectrum on $L^2(\R^d)$.  The operator $\mathcal L_U=-\mathcal A_U$ is self-adjoint by Lemma~\ref{lem:fredholm}.  The nonzero function $v$ satisfying $\mathcal L_Uv=0$ in \eqref{eq:translation-zero-mode} gives $s_U\ge0$.  We claim that in fact $s_U>0$.  We first prove that, for every $t>0$, $e^{t\mathcal L_U}$ maps every nonzero nonnegative $h\in L^2(\R^d)$ to a function positive almost everywhere on $\R^d$.  Write
\begin{equation}
 \mathcal L_U=J*+c(x),
 \qquad c(x)=-1+f'(U(x)),
\end{equation}
and choose $M>0$ such that $c+M\ge0$.  For $h\in L^2(\R^d)$ with $h\ge0$, set $z(t)=e^{t\mathcal L_U}h$ for $t\ge0$.  Then $\widetilde z(t)=e^{Mt}z(t)$ satisfies
\begin{equation}\label{eq:display-4-73}
 \widetilde z_t=J*\widetilde z+(c+M)\widetilde z,
 \qquad \widetilde z(0)=h.
\end{equation}
Picard iteration in \eqref{eq:display-4-73}, using the positive
semigroup $e^{tJ*}$ and $c+M\ge0$, gives $\widetilde z(t)\ge0$ and
\begin{equation}\label{eq:display-4-74}
 \widetilde z(t)\ge e^{tJ*}h.
\end{equation}
Returning to the original semigroup, \eqref{eq:display-4-74} gives
\begin{equation}\label{eq:linearized-positive-semigroup-bound}
 e^{t\mathcal L_U}h\ge e^{-Mt}e^{tJ*}h.
\end{equation}
For nonzero $h\in L^2(\R^d)$ with $h\ge0$, choose a bounded
measurable set $A$ of positive measure and $\delta>0$ such that
$h\ge\delta\mathbf1_A$. By \eqref{eq:convolution-positivity},
a sufficiently high convolution power satisfies
$J^{*n}*\mathbf1_A>0$ on any prescribed ball, as in
\eqref{eq:display-2-16}. The positive series in
\eqref{eq:dispersal-semigroup} thus gives
\begin{equation}\label{eq:convolution-semigroup-positive}
 e^{tJ*}h>0\quad\text{a.e. on }\R^d
\end{equation}
for every $t>0$. Together with
\eqref{eq:linearized-positive-semigroup-bound}, this proves that
$e^{t\mathcal L_U}$ is positivity improving.

Suppose for contradiction that $s_U=0$.  Since $\mathcal L_U$ is self-adjoint, $e^{t\mathcal L_U}$ is then a contraction on $L^2(\R^d)$.  Since $\mathcal L_Uv=0$ and the semigroup is positive,
\begin{equation}\label{eq:absolute-zero-mode-inequality}
 |v|=|e^{t\mathcal L_U}v|
 \le e^{t\mathcal L_U}|v|.
\end{equation}
Taking $L^2(\R^d)$ norms in \eqref{eq:absolute-zero-mode-inequality} gives
\begin{equation}\label{eq:absolute-zero-mode-norm}
 \|v\|_2\le\|e^{t\mathcal L_U}|v|\|_2,
\end{equation}
whereas contractivity gives the reverse inequality.
The nonnegative functions in \eqref{eq:absolute-zero-mode-inequality}
are ordered and have equal $L^2(\R^d)$ norms, so they agree almost
everywhere. Therefore
\begin{equation}
 e^{t\mathcal L_U}|v|=|v|.
\end{equation}
Together with \eqref{eq:translation-zero-mode}, which gives $e^{t\mathcal L_U}v=v$, and
\begin{equation}
 v_+=\frac{|v|+v}{2},\qquad
 v_-=\frac{|v|-v}{2},
\end{equation}
this gives
\begin{equation}\label{eq:positive-parts-invariance}
 e^{t\mathcal L_U}v_+=v_+,
 \qquad
 e^{t\mathcal L_U}v_-=v_-.
\end{equation}
Both functions are nonzero because $v$ changes sign.  Their invariance in \eqref{eq:positive-parts-invariance} and positivity improvement force both $v_+$ and $v_-$ to be strictly positive almost everywhere, contradicting $v_+v_-=0$ almost everywhere.  Thus
$s_U>0$.

\smallskip
\noindent\textit{Step 3. Attainment of the spectral bound.}
For every $\lambda\ge0$, set
\begin{equation}
 q_\lambda(x)=1-f'(U(x))+\lambda,
 \qquad q_{\infty,\lambda}=1-f'(0)+\lambda>1.
\end{equation}
Then
\begin{equation}\label{eq:spectral-fredholm-factorization}
 -\mathcal L_U+\lambda I
 =M_{q_\lambda}\bigl(I-M_{q_\lambda^{-1}}J*\bigr).
\end{equation}
The factorization \eqref{eq:spectral-fredholm-factorization} has the form treated in the proof of Lemma~\ref{lem:fredholm}, applied to the present stationary solution $U$ under (KR).  As in Lemma~\ref{lem:fredholm}, the Lipschitz continuity of $s\mapsto(1-f'(s)+\lambda)^{-1}$ on $[0,1]$ and $U\in L^2(\R^d)$ give
$q_\lambda^{-1}-q_{\infty,\lambda}^{-1}\in L^2(\R^d)$,
so $M_{q_\lambda^{-1}}J*$ is the sum of the strict contraction
$q_{\infty,\lambda}^{-1}J*$ and a Hilbert--Schmidt operator.  The factorization \eqref{eq:spectral-fredholm-factorization} therefore makes $\mathcal L_U-\lambda I$ Fredholm of index zero for every $\lambda\ge0$.
Since $\mathcal L_U$ is bounded and self-adjoint,
$s_U\in\sigma(\mathcal L_U)$. The Fredholm operator
$\mathcal L_U-s_UI$ has index zero. A trivial kernel would make it
surjective and boundedly invertible, a contradiction.
Thus $s_U$ is an eigenvalue with finite-dimensional eigenspace.

Let $\phi$ be an eigenfunction for $s_U$.  Replacing $\phi$ by $|\phi|$ does not decrease its Rayleigh quotient because $J\ge0$, while the multiplication part of $\mathcal L_U$ is unchanged.  Since $s_U$ is the spectral supremum, the quotient for $|\phi|$ must still equal $s_U$.  Equivalently,
\begin{equation}\label{eq:display-4-85}
 \langle(s_UI-\mathcal L_U)|\phi|,|\phi|\rangle=0.
\end{equation}
By \eqref{eq:display-4-85} and the spectral theorem for the nonnegative self-adjoint operator $s_UI-\mathcal L_U$, this operator annihilates $|\phi|$, so $|\phi|$ is again an eigenfunction for $s_U$ and we may take $\phi\ge0$.  The identity $e^{t\mathcal L_U}\phi=e^{ts_U}\phi$ and the positivity improvement obtained from \eqref{eq:linearized-positive-semigroup-bound} and \eqref{eq:convolution-semigroup-positive} give $\phi>0$ almost everywhere.  Let $O(d)$ be the group of orthogonal linear maps of $\R^d$, and let $\mu$ be its normalized Haar measure.  Define the average as an $L^2(\R^d)$ integral:
\begin{equation}\label{eq:radial-eigenfunction-average}
 \phi_{\rm rad}(x)=\int_{O(d)}\phi(Ax)\,d\mu(A).
\end{equation}
The map $A\mapsto\phi(A\cdot)$ is continuous in $L^2(\R^d)$, as follows by approximation by continuous compactly supported functions.  Orthogonal transformations are isometries on $L^2(\R^d)$, so this integral exists and $\|\phi_{\rm rad}\|_2\le\|\phi\|_2$.  By \eqref{eq:main-linearization}, radiality of $J$ and $U$ makes orthogonal transformations commute with $\mathcal L_U$, so the average remains an eigenfunction for $s_U$.  Invariance of $\mu$ in \eqref{eq:radial-eigenfunction-average} makes the average radial.  It is nonnegative and nonzero: for any ball on which $\int_{B_r}\phi>0$, Fubini's theorem gives $\int_{B_r}\phi_{\rm rad}=\int_{B_r}\phi>0$.  In dimension one the average is $(\phi(x)+\phi(-x))/2$.  Thus $s_U>0$ is the principal eigenvalue of $\mathcal L_U$: here the principal eigenvalue means the largest eigenvalue on $L^2(\R^d)$, with the sign convention in \eqref{eq:main-linearization}.

\smallskip
\noindent\textit{Step 4. A positive eigenvalue on a bounded ball.}
For $L>0$, let $\chi_L=\mathbf1_{B_L}$ and set $\phi_L=\chi_L\phi_{\rm rad}$.  Since $\phi_L\to\phi_{\rm rad}$ in $L^2(\R^d)$ and $\mathcal L_U$ is bounded,
\begin{equation}\label{eq:finite-ball-rayleigh}
 \frac{\langle\mathcal L_U\phi_L,\phi_L\rangle}{\|\phi_L\|_2^2}
 \to s_U>0.
\end{equation}
By \eqref{eq:finite-ball-rayleigh}, the quotient is positive for
large $L$. Since $\phi_L$ vanishes outside $B_L$, it is also the
Rayleigh quotient of $\phi_L|_{B_L}$ for
\begin{equation}\label{eq:finite-ball-linearization}
 \mathcal L_{U,L}\psi
 =\int_{B_L}J(x-y)\psi(y)\,dy-\psi+f'(U)\psi
\end{equation}
on $L^2(B_L)$.  When comparing \eqref{eq:finite-ball-linearization} with $\mathcal L_U$ on $L^2(\R^d)$, we extend the argument $\psi$ by zero outside $B_L$.  Since $J\in L^\infty(\R^d)$ and $|B_L|<\infty$, the kernel $J(x-y)$ belongs to $L^2(B_L\times B_L)$.  Thus the integral part of \eqref{eq:finite-ball-linearization} is Hilbert--Schmidt, and hence compact, on $L^2(B_L)$.  The multiplication part satisfies
\begin{equation}\label{eq:negative-multiplication-bound}
 -1+f'(U)\le-m_f<0.
\end{equation}
Here the essential spectrum means the spectrum after removal of isolated eigenvalues of finite multiplicity.  In view of \eqref{eq:negative-multiplication-bound}, the spectrum of the multiplication operator, and hence its essential spectrum, is contained in $(-\infty,-m_f]$.  Weyl's theorem for compact self-adjoint perturbations (see \cite{Kato1995}) therefore gives
\begin{equation}\label{eq:finite-ball-essential-spectrum}
 \sigma_{\rm ess}(\mathcal L_{U,L})\subset(-\infty,-m_f].
\end{equation}
Fix any $L$ sufficiently large that the quotient in \eqref{eq:finite-ball-rayleigh} is positive.  The spectral supremum of $\mathcal L_{U,L}$ is then positive and lies above the essential spectrum in \eqref{eq:finite-ball-essential-spectrum}; hence it is an isolated eigenvalue $\lambda_L$ of finite multiplicity.  Let $\psi$ be an eigenfunction for $\lambda_L$.  Replacing $\psi$ by $|\psi|$ does not decrease the Rayleigh quotient, while no quotient can exceed $\lambda_L$; hence $|\psi|$ is again an eigenfunction for $\lambda_L$.  The operator $\mathcal L_{U,L}$ commutes with orthogonal transformations, so averaging $|\psi|$ over $O(d)$ as in \eqref{eq:radial-eigenfunction-average} gives a nonzero radial nonnegative eigenfunction $\psi_L$ with the same eigenvalue.
\end{proof}

\begin{lemma}\label{lem:unorderedstable}
Assume (KM), \textup{(F1)--(F3)}, and $F(1)>0$. Let $0<R_1<R_2$, and assume that $u^{R_1}$ and $u^{R_2}$ neither become extinct nor spread. The solutions $u^{R_1}$ and $u^{R_2}$ cannot both converge locally uniformly to the same nonconstant radial stationary solution $U$ of \eqref{eq:main-equation} tending to zero as $|x|\to\infty$.
\end{lemma}

\begin{proof}
Strong comparison in Lemma~\ref{lem:strong} gives
\begin{equation}\label{eq:ordered-orbit-difference}
 w(t,x):=u^{R_2}(t,x)-u^{R_1}(t,x)>0
\end{equation}
for almost every $x\in\R^d$ and every $t>0$.  By subtracting the two equations \eqref{eq:main-equation}, the difference $w$ satisfies
\begin{equation}\label{eq:display-4-91}
 w_t=J*w-w+c(t,x)w,
\qquad
 c(t,x)=\int_0^1f'\bigl(u^{R_1}(t,x)+\theta w(t,x)\bigr)\,d\theta.
\end{equation}
Suppose both solutions converge locally uniformly to $U$.  Since neither becomes extinct nor spreads, Proposition~\ref{prop:radialinterval} places $R_1,R_2$ in $[R_{\rm ext},R_{\rm prop}]$.  Proposition~\ref{prop:transitionomega} gives $U\in L^2(\R^d)$, $0\le U\le1$, and radial monotonicity, so Lemma~\ref{lem:radialunstable} supplies $L$, $\lambda_L$ and the radial nonnegative eigenfunction $\psi_L$.  By local uniform convergence of both solutions and uniform continuity of $f'$ on $[0,1]$, there is $T>0$ such that, for every $t\ge T$,
\begin{equation}\label{eq:ordered-orbit-coefficient-bound}
 c(t,x)\ge f'(U(x))-\frac{\lambda_L}{2}
 \qquad (x\in B_L).
\end{equation}
Write $\mathcal L_{U,L}$ for the operator in \eqref{eq:finite-ball-linearization}, with zero extension outside $B_L$.  Since $w\ge0$ on $\R^d$, discarding the exterior contribution to $J*w$ in \eqref{eq:display-4-91} and using \eqref{eq:ordered-orbit-coefficient-bound} gives the following inequality for the restriction $w(t,\cdot)|_{B_L}$:
\begin{equation}\label{eq:ordered-orbit-ball-inequality}
 w_t\ge\mathcal L_{U,L}w-\frac{\lambda_L}{2}w.
\end{equation}
Multiplying \eqref{eq:ordered-orbit-ball-inequality} by $\psi_L$ and using self-adjointness and \eqref{eq:finite-ball-eigenfunction}, we obtain
\begin{equation}\label{eq:ordered-orbit-unstable-moment}
 \frac d{dt}\int_{B_L}w(t,x)\psi_L(x)\,dx
 \ge\frac{\lambda_L}{2}
 \int_{B_L}w(t,x)\psi_L(x)\,dx.
\end{equation}
The integral is positive by \eqref{eq:ordered-orbit-difference},
since $\psi_L\ge0$ is nonzero. Hence
\eqref{eq:ordered-orbit-unstable-moment} forces exponential growth
at rate at least $\lambda_L/2$, whereas $0<w\le1$ bounds it by
$\|\psi_L\|_{L^1(B_L)}$. This contradiction proves the claim.
\end{proof}

\begin{lemma}\label{lem:orderedisolation}
Assume (KR) and \textup{(F3)}, and let $U$ be as in \eqref{eq:display-4-18}, nonconstant, radial and nonincreasing. There exist $L>0$ and $\varepsilon_U>0$ such that no stationary solution $V\ne U$ of \eqref{eq:main-equation} with $0\le V\le1$, comparable to $U$ on $\R^d$, satisfies
\begin{equation}
 \|V-U\|_{L^\infty(B_L)}<\varepsilon_U.
\end{equation}
Here comparable means either $V\ge U$ almost everywhere on $\R^d$ or $V\le U$ almost everywhere on $\R^d$.
\end{lemma}

\begin{proof}
Choose $L$, $\lambda_L>0$ and $\psi_L\ge0$ from Lemma~\ref{lem:radialunstable}. Uniform continuity of $f'$ on $[0,1]$ gives $\varepsilon_U>0$ such that $\|V-U\|_{L^\infty(B_L)}<\varepsilon_U$ implies
\begin{equation}\label{eq:display-4-96}
 b(x):=\int_0^1 f'\bigl(U(x)+s(V(x)-U(x))\bigr)\,ds
 \ge f'(U(x))-\lambda_L/2\qquad(x\in B_L).
\end{equation}
If $V\ne U$ is comparable to $U$, set $w=V-U$ when $V\ge U$ and $w=U-V$ when $V\le U$, so that $w=|V-U|\ge0$.  Subtracting the stationary equation for the smaller solution from that for the larger solution gives
\begin{equation}\label{eq:display-4-97}
 0=J*w-w+b(x)w.
\end{equation}
Applying Lemma~\ref{lem:strong} to the time-independent solutions $U$ and $V$ of \eqref{eq:main-equation} gives $w>0$ almost everywhere on $\R^d$.  Since $w\ge0$ outside $B_L$, restricting \eqref{eq:display-4-97} to $B_L$, discarding the nonnegative exterior contribution to $J*w$, and using \eqref{eq:display-4-96} yields the following inequality for $w|_{B_L}$, with $\mathcal L_{U,L}$ as in \eqref{eq:finite-ball-linearization}:
\begin{equation}\label{eq:ordered-equilibrium-ball-inequality}
 0\ge\mathcal L_{U,L}w-\frac{\lambda_L}{2}w.
\end{equation}
Since $w$ is bounded and $\psi_L\in L^2(B_L)\subset L^1(B_L)$,
we may test \eqref{eq:ordered-equilibrium-ball-inequality} by
$\psi_L$. Self-adjointness and \eqref{eq:finite-ball-eigenfunction}
give
\begin{equation}
 0\ge\frac{\lambda_L}{2}\int_{B_L}w(x)\psi_L(x)\,dx>0,
\end{equation}
a contradiction. 
\end{proof}

\begin{proposition}\label{prop:sharpradial}
Assume (KM), \textup{(F1)--(F2)}, \textup{(F4)}, and $F(1)>0$. Then
\begin{equation}
 R_{\rm ext}=R_{\rm prop}=:R_*.
\end{equation}
At $R=R_*$, there is a radial nonincreasing
$U_*\in\mathcal S_0(J,f)\setminus\{0\}$ such that $u^{R_*}(t,\cdot)\to U_*$ in $L^2(\R^d)\cap L^\infty(\R^d)$ as $t\to\infty$.
\end{proposition}

\begin{proof}
Let $I=[R_{\rm ext},R_{\rm prop}]$.  For each $R\in I$, let $U_R$ be the stationary limit supplied by Proposition~\ref{prop:transitionconverges}.  Each $U_R$ belongs to $X$, is nonzero and radial nonincreasing, and tends to zero at infinity; in particular, $U_R$ is nonconstant.  For $R_1,R_2\in I$ with $R_1<R_2$, Lemma~\ref{lem:semiflow} gives $u^{R_1}(t,\cdot)\le u^{R_2}(t,\cdot)$.  Passing to the limits supplied by Proposition~\ref{prop:transitionconverges} gives
\begin{equation}\label{eq:ordered-stationary-limits}
 R_1<R_2\quad\Longrightarrow\quad U_{R_1}\le U_{R_2},
\end{equation}
and Lemma~\ref{lem:unorderedstable} shows that $U_{R_1}\ne U_{R_2}$. Thus $R\mapsto U_R$ is injective on $I$, and \eqref{eq:ordered-stationary-limits} makes its image $\mathcal C:=\{U_R:R\in I\}$ totally ordered by pointwise comparison.

Every element of $\mathcal C$ is continuous by Lemma~\ref{lem:steadyregularity}.  Consider the space $C_{\rm rad}(\R^d)$ of continuous radial functions with the local uniform topology, metrized by
\begin{equation}\label{eq:radial-local-metric}
 d_{\rm loc}(U,V)=\sum_{n=1}^\infty 2^{-n}
 \min\{1,\|U-V\|_{L^\infty(B_n)}\}.
\end{equation}
Radial extensions of continuous piecewise linear profiles on $[0,\infty)$, determined by finitely many rational nodes including $0$, rational values at the nodes, and a constant continuation beyond the last node, form a countable dense subset for the metric \eqref{eq:radial-local-metric}.  Thus $C_{\rm rad}(\R^d)$ is separable and, being metrizable, second countable.

For each $U\in\mathcal C$, choose $L$ and $\varepsilon_U$ from Lemma~\ref{lem:orderedisolation}, and set $\mathcal N_U=\{V\in C_{\rm rad}(\R^d):\|V-U\|_{L^\infty(B_L)}<\varepsilon_U\}$.  This set is open in the local uniform topology.  Since $\mathcal C$ is totally ordered by \eqref{eq:ordered-stationary-limits}, Lemma~\ref{lem:orderedisolation} gives $\mathcal N_U\cap\mathcal C=\{U\}$.  Enumerate a countable basis as $\mathfrak B=\{\mathcal B_k:k\ge1\}$ and let $\mathcal O_U$ be the first basis element satisfying
$U\in\mathcal O_U\subset\mathcal N_U$.
Distinct $U,V\in\mathcal C$ have distinct
$\mathcal O_U,\mathcal O_V$: a common basic open set would contain
$V$ and lie in $\mathcal N_U$, contradicting
$\mathcal N_U\cap\mathcal C=\{U\}$.
Thus $U\mapsto\mathcal O_U$ injects $\mathcal C$ into the countable
basis $\mathfrak B$.

If $R_{\rm ext}<R_{\rm prop}$, injectivity of $R\mapsto U_R$
would map the uncountable interval $I$ into the countable set
$\mathcal C$, a contradiction. Thus $I=\{R_*\}$, and
Proposition~\ref{prop:transitionconverges} gives convergence in $X$
to $U_*:=U_{R_*}\in\mathcal S_0(J,f)\setminus\{0\}$.
\end{proof}

\begin{remark}
Under the assumptions of Proposition~\ref{prop:sharpradial}, the limits $U_R$ of $u^R$ form a totally ordered set.  Together with time continuity from Lemma~\ref{lem:semiflow}, Lemma~\ref{lem:transitiontight} gives relative compactness in $X$ by combining the asymptotic translation estimate with uniform estimates at spatial infinity, and the gradient inequality gives convergence of $u^R(t,\cdot)$ for each $R\in[R_{\rm ext},R_{\rm prop}]$.  Strong comparison and a positive eigenvalue on a bounded ball then isolate each limit within this ordered set, as used in Proposition~\ref{prop:sharpradial}.
\end{remark}

For the following spectral result, assume \textup{(F1)--(F2)},
\textup{(KM)} and $F(1)>0$.

The next lemma gives the spectral conclusion associated with
Chmaj and Ren's instability theorem
\cite[Theorem~5.1 and its proof, pp.~38--40]{ChmajRen1999}.
We give the argument for each positive decaying stationary
solution under \textup{(F1)--(F2)} and \textup{(KM)},
using positivity of convolution powers and the Fredholm
factorization from Lemma~\ref{cs:lem:fredholm}.

For $U\in\mathcal S_0(J,f)$, set
\begin{equation}\label{cs:eq:stationary-linearization}
 \lambda_U=\sup\sigma_{L^2(\R^d)}(\mathcal L_U),
\end{equation}
where $\mathcal L_U$ is defined in \eqref{eq:main-linearization}.

\begin{lemma}\label{cs:lem:stationary-instability}
For every $U\in\mathcal S_0(J,f)\setminus\{0\}$, the number
$\lambda_U$ in \eqref{cs:eq:stationary-linearization} is a positive
simple eigenvalue of $\mathcal L_U$.  It has an eigenfunction
$\phi_U\in C(\R^d)\cap L^2(\R^d)\cap L^\infty(\R^d)$
which is positive everywhere.
\end{lemma}

\begin{proof}
By Lemma~\ref{cs:lem:stationary-integrability} and
\eqref{cs:eq:stationary-derivative} with $g=f$,
$U\in H^1(\R^d)$ and $\mathcal L_U\partial_iU=0$.
At least one derivative $v=\partial_iU$ is nonzero.
It changes sign: if it had one sign, the $C^1$ function $U$
would be monotone on each line parallel to $e_i$, whereas its
limits at both ends are zero.
Thus $\lambda_U\ge0$.

Set $T_t=e^{t\mathcal L_U}$ and choose $M$ so that
$-1+f'(U)+M\ge0$. The positive-operator calculation in
\eqref{eq:display-4-73}--\eqref{eq:linearized-positive-semigroup-bound}
gives
\begin{equation}\label{cs:eq:stationary-semigroup}
 T_th\ge e^{-Mt}\sum_{n=0}^\infty
               \frac{t^n}{n!}J^{*n}*h,\qquad h\in L^2(\R^d),\quad h\ge0.
\end{equation}
By \eqref{eq:convolution-semigroup-positive}, $T_th>0$ almost everywhere
for $t>0$ and nonzero $h\ge0$.

If $\lambda_U=0$, self-adjointness makes $T_t$ an $L^2(\R^d)$ contraction.
Since $T_tv=v$ and $T_t$ is positive, $|v|\le T_t|v|$.
Contractivity forces equality, and hence
$T_tv_+=v_+$ and $T_tv_-=v_-$, where $v_-=(-v)_+$.
Both functions are nonzero, so positivity in
\eqref{cs:eq:stationary-semigroup} contradicts their disjoint supports.
We conclude that $\lambda_U>0$.

Recall $a_U=1-f'(U)$ and $a_\infty=1-f'(0)>1$ from
\eqref{cs:eq:hessian}, so $A_U=-\mathcal L_U$.
For every $\lambda\ge0$, apply the factorization and compactness
argument of Lemma~\ref{cs:lem:fredholm}
with $a_U+\lambda$ and $a_\infty+\lambda$ in place of
$a_U$ and $a_\infty$.  It follows that
$\lambda I-\mathcal L_U$ is Fredholm of index zero.
Indeed, $p_\lambda=(a_U+\lambda)^{-1}-(a_\infty+\lambda)^{-1}$
is bounded and tends to zero at infinity, so the cutoff argument
in \eqref{cs:eq:compact-multiplier} applies; moreover
$a_\infty+\lambda>1$.
Since $\lambda_U\in\sigma(\mathcal L_U)$, the Fredholm property
and self-adjointness give $\lambda_UI-\mathcal L_U$ a nonzero
finite-dimensional kernel. Thus $\lambda_U$ is an eigenvalue.

If $h$ is a real eigenfunction at $\lambda_U$, nonnegativity of $J$
gives
\[
 \langle\mathcal L_U|h|,|h|\rangle
 \ge\langle\mathcal L_Uh,h\rangle
 =\lambda_U\|h\|_2^2.
\]
By the definition of $\lambda_U$, equality holds. The spectral
theorem for $\lambda_UI-\mathcal L_U\ge0$ then makes $|h|$ an
eigenfunction, so we may choose $\phi_U\ge0$.
Equation \eqref{cs:eq:stationary-semigroup} makes it positive almost
everywhere, and
\[
 (a_U+\lambda_U)\phi_U=J*\phi_U,\qquad
 \|\phi_U\|_\infty
 \le(m_f+\lambda_U)^{-1}\|J\|_2\|\phi_U\|_2
\]
gives boundedness and continuity.
Here continuity follows from translations of $J$ in $L^2(\R^d)$.
The convolution $J*\phi_U$ is positive at every point, so the
continuous representative satisfies $\phi_U>0$ everywhere.

If the eigenspace had dimension greater than one, it would contain
a nonzero real $h$ orthogonal to $\phi_U$.
Such an $h$ changes sign.
The preceding absolute-value argument makes $|h|$,
$h_+$, and $h_-$ eigenfunctions at $\lambda_U$.
Their evolution under $T_t$ and
\eqref{cs:eq:stationary-semigroup} again contradict their disjoint
supports.
Hence the eigenvalue is simple.
\end{proof}

\subsection{One-dimensional limits}

In one dimension, $B_R=(-R,R)$, so the critical radius in
Proposition~\ref{prop:sharpradial} is a critical half-length.
Lemma~\ref{lem:peakidentity} gives an identity for the value of the
stationary limit at the origin. In every dimension,
Lemma~\ref{lem:duality} shows that $v^R=1-u_R$ solves
\eqref{eq:main-equation} with reaction $f^\sharp$, as used in
Corollary~\ref{cor:main_negative}.

\begin{proposition}\label{prop:sharp1d}
Assume $d=1$, (KM), \textup{(F1)--(F2)}, \textup{(F4)}, and $F(1)>0$.  Set $L_*=R_*(f,J)$ as in Proposition~\ref{prop:sharpradial}.  Then $u^{L_*}$ satisfies
\begin{equation}\label{eq:one-dimensional-critical-convergence}
 \|u^{L_*}(t,\cdot)-U_*\|_2+
 \|u^{L_*}(t,\cdot)-U_*\|_\infty\to0,
\end{equation}
where $U_*\in\mathcal S_0(J,f)\cap L^2(\R)$ is even and
nonincreasing on $[0,\infty)$, with $0<U_*<1$ on $\R$.
\end{proposition}

\begin{proof}
By Proposition~\ref{prop:sharpradial},
\eqref{eq:one-dimensional-critical-convergence} holds for
$U_*\in\mathcal S_0(J,f)\setminus\{0\}$. In one dimension, radial
monotonicity gives evenness and monotonicity on $[0,\infty)$;
Lemma~\ref{lem:steadyregularity} gives a continuous representative.

Since $U_*\in\mathcal S_0(J,f)\setminus\{0\}$,
Lemma~\ref{cs:lem:stationary-basic} with $g=f$ gives $0<U_*<1$ on $\R$.
\end{proof}

\begin{proof}[Proof of Corollary~\ref{thm:1dcompletefull}]
By Proposition~\ref{prop:sharp1d} with $L_*=R_*$, \eqref{eq:display-1-19} holds for the solution in Corollary~\ref{thm:1dcompletefull}, and the stationary limit $U_*$ satisfies $0<U_*<1$ on $\R$.
\end{proof}

\subsection{Proofs of the threshold results}

\begin{proof}[Proof of Theorem~\ref{thm:main}]
Proposition~\ref{prop:radialinterval} proves \textup{(i)}, including
finiteness of the radii, the exact extinction and spreading intervals,
and \eqref{eq:extinction-radius-bounds}.
The lower bound follows from Lemma~\ref{lem:smallseed}; the upper bound
uses \eqref{eq:negative-ball-energy} and energy dissipation
\eqref{eq:energy-dissipation}. For the $C^1$ reaction in \textup{(F1)},
\eqref{eq:measurable-seed-spreading} gives spreading for large $R$
by comparison with a polynomial reaction.

For \textup{(ii)}, assume \textup{(KM)} and \textup{(F4)}.
Proposition~\ref{prop:transitionconverges} gives
\eqref{eq:whole-orbit-convergence} in
$L^2(\R^d)\cap L^\infty(\R^d)$ for every
$R\in[R_{\rm ext},R_{\rm prop}]$.
Lemma~\ref{lem:radialunstable} gives the asserted positive eigenvalue
and nonnegative radial eigenfunction of $\mathcal L_U$ for each
stationary solution $U$ specified in \textup{(ii)}.
By Proposition~\ref{prop:sharpradial},
$R_{\rm ext}=R_{\rm prop}=:R_*$ and
$U_*=U_{R_*}\in\mathcal S_0(J,f)\setminus\{0\}$ is the critical limit.
\end{proof}

\begin{proof}[Proof of Corollary~\ref{cor:main_negative}]
By Lemma~\ref{lem:duality}, $f^\sharp$ satisfies
\textup{(F1)--(F2)} with unstable zero $1-\vartheta$ and the same $m_f$;
\eqref{eq:reflected-potential} gives $F^\sharp(1)=-F(1)>0$.
The lemma also identifies $v^R=1-u_R$ as the unique global solution of
\eqref{eq:main-equation} with reaction $f^\sharp$ and initial datum
$\mathbf1_{B_R}$. Theorem~\ref{thm:main}\textup{(i)} therefore gives
the two radii and their bounds, with
$T_{f^\sharp}=m_f^{-1}\log(4/(1-\vartheta))$.
Since $u_R=1-v^R$, extinction and spreading of $v^R$ give,
respectively, uniform convergence of $u_R$ to $1$ and local uniform
convergence to $0$; the converse implications follow from the same identity.

Under \textup{(KM)} and \textup{(F4)}, Lemma~\ref{lem:duality}
gives \textup{(F4)} for $f^\sharp$.
Applying Theorem~\ref{thm:main}\textup{(ii)} with this reaction yields
$R_{\rm ext}^\sharp=R_{\rm prop}^\sharp=R_*^\sharp$,
convergence of $v^{R_*^\sharp}(t,\cdot)$ to
$U_*^\sharp\in\mathcal S_0(J,f^\sharp)\setminus\{0\}$ in
$L^2(\R^d)\cap L^\infty(\R^d)$, and the spectral conclusion.
Applying Lemma~\ref{lem:duality} to the stationary solution
$U_*^\sharp$ with reaction $f^\sharp$ gives
$1-U_*^\sharp\in\mathcal S(J,f)$, since $(f^\sharp)^\sharp=f$.
Finally,
$u_{R_*^\sharp}(t,\cdot)-(1-U_*^\sharp)
=-(v^{R_*^\sharp}(t,\cdot)-U_*^\sharp)$,
so this difference tends to zero in
$L^2(\R^d)\cap L^\infty(\R^d)$ as $t\to\infty$.
\end{proof}

\begin{corollary}\label{cor:sandwich}
Assume (KR), \textup{(F1)--(F2)}, and $F(1)>0$.  Let $R_{\rm ext}$ and $R_{\rm prop}$ be as in Proposition~\ref{prop:radialinterval}.  Let $u_0\in L^\infty(\R^d)$ satisfy $0\le u_0\le1$, and let $u$ be the unique global solution of \eqref{eq:main-equation} with initial datum $u_0$.  If
\begin{equation}\label{eq:extinction-sandwich-condition}
 0\le u_0\le \mathbf1_{B_R}
 \quad\text{for some }0<R<R_{\rm ext},
\end{equation}
then $u(t,\cdot)$ converges uniformly to $0$ as $t\to\infty$.  If
\begin{equation}\label{eq:propagation-sandwich-condition}
 \mathbf1_{B_R}\le u_0\le1
 \quad\text{for some }R>R_{\rm prop},
\end{equation}
then $u(t,\cdot)$ converges locally uniformly to $1$ as $t\to\infty$.
\end{corollary}

\begin{proof}
By Lemma~\ref{lem:semiflow}, \eqref{eq:extinction-sandwich-condition}
gives $0\le u(t,\cdot)\le u^R(t,\cdot)$, whereas
\eqref{eq:propagation-sandwich-condition} gives
$u^R(t,\cdot)\le u(t,\cdot)\le1$.
The extinction and spreading alternatives in
Proposition~\ref{prop:radialinterval} give the asserted limits.
\end{proof}

\section{Remarks and examples}\label{sec:remarks}

\subsection{Remarks on the assumptions}

Throughout this subsection, $f$ satisfies \textup{(F1)--(F2)}.

\begin{remark}
Assumption (K1) permits kernels with unbounded support.
The second moment in (KR) is used in the radial barriers of Fife and
Wang \cite{FifeWang1998}; the first moment in (K1), in Theorem~A
and Lemma~\ref{lem:ballenergy}.
The bounds in Theorem~\ref{thm:main}\textup{(i)} depend on
$\|J\|_\infty$ and $M_1(J)$.
\end{remark}

\begin{remark}
When $F(1)=0$, \eqref{eq:directional-speed-identity} gives $c_e=0$
in every direction under (K1) and \textup{(F3)}.
Lemma~\ref{lem:smallseed} still gives extinction for sufficiently
small $L^1(\R^d)$ mass under (K0) and $J\in L^\infty(\R^d)$.
The expanding lower barrier of Fife and Wang \cite{FifeWang1998}
requires positive speed, while the bulk energy term $-F(1)|B_R|$ vanishes.
\end{remark}

\begin{lemma}\label{lem:kernel-monotonicity-necessity}
Assume \textup{(F1)--(F2)} and (KR).  Then (KM) is equivalent to the following property: for every $R>0$ and $t>0$, the function $u^R(t,\cdot)$ has a radial nonincreasing representative.
\end{lemma}

\begin{proof}
Lemma~\ref{lem:radialmonotone} proves the forward implication.
Conversely, suppose that $u^R(t,\cdot)$ has a radial nonincreasing
representative for every $R,t>0$.
Fix $R>0$ and set $G_R=J*\mathbf1_{B_R}$.
By Lemma~\ref{lem:semiflow}, $u^R$ is differentiable at $t=0$
in $L^\infty(\R^d)$; since $f(0)=f(1)=0$,
\eqref{eq:main-equation} gives
\[
 u^R(t,\cdot)=\mathbf1_{B_R}
       +t\bigl(G_R-\mathbf1_{B_R}\bigr)+o(t)
 \quad\text{in }L^\infty(\R^d)
 \qquad(t\downarrow0).
\]
On $\{|x|>R\}$ the indicator vanishes, so this expansion gives
$t^{-1}u^R(t,\cdot)\to G_R$ in $L^\infty(\{|x|>R\})$
as $t\downarrow0$.
Functions admitting radial nonincreasing representatives form a closed subset
of this space:
in polar coordinates, essential uniform convergence of the profiles
on $(R,\infty)$ preserves their ordering outside a null set of radii,
and that ordering extends to a nonincreasing representative.
Thus $G_R$ has such a representative on $\{|x|>R\}$.
Translation continuity of $J$ in $L^1(\R^d)$ makes $G_R$ uniformly
continuous, and radiality of $J$ makes it radial. Its continuous
radial profile is therefore nonincreasing on $(R,\infty)$.

For $e_1=(1,0,\ldots,0)$ and $0<R<r<s$, this monotonicity gives
\[
 \frac1{|B_R|}\int_{B_R}J(re_1-y)\,dy
 \ge
 \frac1{|B_R|}\int_{B_R}J(se_1-y)\,dy.
\]
By continuity of $J$, letting $R\downarrow0$ gives
$J(re_1)\ge J(se_1)$.
Thus its radial profile is nonincreasing on $(0,\infty)$;
continuity at zero extends this to $[0,\infty)$, proving (KM).
\end{proof}

\begin{remark}
For any $a>1$, define
\[
 J_a(x)=\frac{(1+a|x|^2)e^{-|x|^2}}
              {\pi^{d/2}(1+ad/2)},\qquad x\in\R^d.
\]
The Gaussian integral and its second moment give
\[
 \int_{\R^d}e^{-|x|^2}\,dx=\pi^{d/2},
 \qquad
 \int_{\R^d}|x|^2e^{-|x|^2}\,dx=\frac d2\pi^{d/2},
\]
so $\int_{\R^d}J_a=1$.
The kernel is positive, radial and smooth, with integrable derivatives
through order two and finite second moment, hence satisfies (KR).
Define $j_a:[0,\infty)\to(0,\infty)$ by $J_a(x)=j_a(|x|)$.
Differentiation gives
\[
 j_a'(r)=\frac{2r(a-1-ar^2)e^{-r^2}}
                  {\pi^{d/2}(1+ad/2)}>0
 \quad\text{for }0<r<\sqrt{\frac{a-1}{a}},
\]
which shows that (KM) fails.

For $J=J_a$, suitable $R>0$ give solutions $u^R(t,\cdot)$
with no radial nonincreasing representative for all small $t>0$.
Indeed, choose $e_1=(1,0,\ldots,0)$ and
$0<r<s<\sqrt{(a-1)/a}$. The formula for $j_a'$ gives
$J_a(re_1)<J_a(se_1)$.
By continuity, $|B_R|^{-1}(J_a*\mathbf1_{B_R})$ converges to $J_a$
at these two points as $R\downarrow0$.
For all sufficiently small $R<r$, therefore,
$G_R=J_a*\mathbf1_{B_R}$ satisfies $G_R(re_1)<G_R(se_1)$.
Continuity and radiality of $G_R$ then give nonempty bounded open
annuli $A_1,A_2\subset\{|x|>R\}$, with every radius in $A_1$
smaller than every radius in $A_2$, and $\delta>0$ such that
\[
 \inf_{x\in A_2}G_R(x)-\sup_{x\in A_1}G_R(x)>\delta.
\]
By the expansion $u^R(t,\cdot)=tG_R+o(t)$ in
$L^\infty(\{|x|>R\})$ from the proof of
Lemma~\ref{lem:kernel-monotonicity-necessity}, there exists $t_0>0$ such that
\[
 \operatorname*{ess\,inf}_{x\in A_2}u^R(t,x)
 >\operatorname*{ess\,sup}_{x\in A_1}u^R(t,x)
 \qquad(0<t<t_0).
\]
This strict ordering on annuli of positive measure excludes a radial
nonincreasing representative for $0<t<t_0$.
For $f(s)=s(1-s)(s-\tfrac14)$, we have
$\sup_{[0,1]}f'=13/48$ and $F(1)=1/24>0$;
thus $(f,J_a)$ satisfies \textup{(F1)--(F2)} and (KR).

Lemma~\ref{lem:kernel-monotonicity-necessity} characterizes (KM)
by radial monotonicity of every $u^R(t,\cdot)$.
Whether $R_{\rm ext}<R_{\rm prop}$ can occur under
\textup{(F1)--(F2)}, $F(1)>0$ and (KR) remains a separate question.
\end{remark}

\subsection{Stationary states and reaction regularity}

For the remainder of this section, assume \textup{(F1)--(F2)}
and $F(1)>0$. Assume also \textup{(KM)} except where a different
kernel is explicitly specified.

\begin{remark}\label{cs:rem:kernel-shape}
The reflection inequality \eqref{cs:eq:reflection-kernel} uses
radial decrease of $J$. In one dimension, take
\[
 \widehat J(x)=
 \frac{e^{-(x-6)^2}+e^{-(x+6)^2}}{2\sqrt\pi},\qquad
 f(s)=s(s-\tfrac14)(1-s),\qquad
 u_0=\mathbf1_{[-1,1]}.
\]
Here $\widehat J$ is an even smooth strictly positive probability
density with all moments.
The reaction assumptions follow by direct differentiation and
integration; the values are recorded in \eqref{cs:eq:cubic-constants}.
The equation with $J=\widehat J$, interpreted in $L^\infty(\R)$,
gives $u_t(0,\cdot)=\widehat J*u_0$ almost everywhere outside
$[-1,1]$.
Direct integration bounds give
\[
 (\widehat J*u_0)(6)\ge\frac{e^{-1}}{\sqrt\pi},\qquad
 (\widehat J*u_0)(2)
 \le\frac{e^{-9}+e^{-49}}{\sqrt\pi}.
\]
By continuity of $\widehat J*u_0$, there is $\varepsilon\in(0,1)$
such that its infimum on $(6-\varepsilon,6+\varepsilon)$ exceeds
its supremum on $(2-\varepsilon,2+\varepsilon)$.
The expansion $u(t,\cdot)=u_0+t u_t(0,\cdot)+o(t)$ in $L^\infty(\R)$
as $t\downarrow0$ therefore gives
\[
 \essinf_{|x-6|<\varepsilon}u(t,x)
 >\esssup_{|x-2|<\varepsilon}u(t,x)
\]
for all sufficiently small $t>0$.

\end{remark}

\begin{remark}\label{cs:rem:regularity-examples}
The two regularity requirements can be illustrated by explicit
reactions.  Set
\[
 p(s)=s(s-\tfrac14)(1-s),\qquad
 \chi_1(s)=
 \begin{cases}
  e^{-1/(1-s)^2},&0\le s<1,\\
  0,&s=1,
 \end{cases}
 \qquad f_*(s)=p(s)\bigl(1+\tfrac14\chi_1(s)\bigr).
\]
Extend $f_*$ to $\R$ by the same formula near zero and by a
$C^1$ extension beyond one.  Then $f_*\in C^\infty([0,1])$ and is analytic near every
$[0,A]$ with $A<1$.
All derivatives of $f_*-p$ vanish at $1$, while
$f_*(s)>p(s)$ for $1/4<s<1$; thus $f_*$ is nonanalytic at $1$.
The positive factor preserves the three zeros and their derivative signs.
On $[0,1]$,
$|p|\le3/16$, $\max p'=13/48$,
$0\le\chi_1\le e^{-1}$, and $|\chi_1'|\le1$.
Consequently
\[
 \max_{[0,1]}f_*'
 \le \frac{13}{48}
       +\frac14\left(\frac{3}{4e}+\frac3{16}\right)<1,
 \qquad
 \int_0^1f_*(s)\,ds
 \ge\frac1{24}-\frac3{64e}>0.
\]
Here $|p'|\le3/4$ gives the first perturbation bound, and
$\int_0^1p=1/24$ gives the second.
By Lemma~\ref{cs:lem:uniform-height}, $f_*$ satisfies the analytic
alternative of Theorem~\ref{cs:thm:dynamics} for every kernel in
\textup{(KM)}.

For the stationary theorem, fix $\alpha\in(0,1)$ and choose
$\chi_0\in C^\infty([0,1])$ with $0\le\chi_0\le1$,
$\chi_0=1$ on $[0,1/16]$, and $\chi_0=0$ on $[1/8,1]$.
For sufficiently small $\varepsilon>0$, the function
\[
 f_{\varepsilon,\alpha}(s)=p(s)+
       \varepsilon\chi_0(s)s^{1+\alpha}
\]
satisfies \eqref{cs:eq:reaction-regularity} and \textup{(F1)--(F2)}
after linear extension beyond both endpoints with the endpoint slopes.
Indeed, the perturbation tends to zero in $C^1$ as
$\varepsilon\downarrow0$, preserves the zeros and their derivative
signs, and is supported where $p(s)/s$ is strictly negative.
Near zero its second derivative equals
$p''(s)+\varepsilon\alpha(1+\alpha)s^{\alpha-1}$.
Thus the weighted bound in \eqref{cs:eq:reaction-regularity} holds
although $f_{\varepsilon,\alpha}''(s)\to+\infty$ as $s\downarrow0$.
\end{remark}

\begin{remark}\label{cs:rem:reaction-regularity}
Assume also \eqref{cs:eq:reaction-regularity}.
Recall the numbers $a(J,f)$ and $\bar a(J,f)$ defined in
\eqref{cs:eq:profile-height}.
By Lemmas~\ref{cs:lem:uniform-height} and
\ref{cs:lem:stationary-existence}, together with
Theorem~\ref{cs:thm:stationary-summary},
$\beta<a(J,f)\le\bar a(J,f)<1$.
Thus analyticity near $[0,a(J,f)]$ is a sufficient scalar condition
for Theorem~\ref{cs:thm:dynamics}.
In particular, this condition holds when $f$ is analytic at every
point of $[0,1)$, with an analytic extension through zero.

Every $f\in C^2([0,1])$ satisfies \eqref{cs:eq:reaction-regularity},
the additional regularity required in Theorem~\ref{cs:thm:stationary-summary}.
Its role in the second derivative estimate is given explicitly in
\eqref{cs:eq:stationary-second-derivative}--\eqref{cs:eq:stationary-gradient-bound}.
For the dynamics theorem, analyticity and
\eqref{cs:eq:translation-kernel} are two sufficient conditions for
the local estimate \eqref{cs:eq:gradient-inequality}.
The question of whether \textup{(F1)--(F2)} and
\textup{(KM)}, together with $f\in C^2([0,1])$, imply the
same convergence conclusion is therefore the question of obtaining
convergence for every compact orbit furnished by
Lemma~\ref{cs:lem:compact-orbit} under these scalar hypotheses alone.

For $U(x)=V(|x-x_*|)$ in \eqref{cs:eq:stationary-summary},
$V'(r)<0$ makes $V^{-1}$ a $C^2$ function on $(0,V(0))$.
The stationary equation gives
\[
 f(s)=s-(J*U)\bigl(x_*+V^{-1}(s)e_1\bigr),
 \qquad e_1=(1,0,\ldots,0).
\]
Since $J*U\in C^2$, this identity gives
$f\in C^2((0,V(0)))$.
\end{remark}

\begin{remark}\label{cs:rem:profile-question}
Assume also \eqref{cs:eq:reaction-regularity}.
For fixed $J$ and $f$, define
\[
 \mathcal R(J,f)=
 \left\{V\in C^2([0,\infty)):
   \begin{array}{l}
    x\mapsto V(|x|)\text{ belongs to }\mathcal S_0(J,f)\setminus\{0\},\\
    V'(0)=0,\quad V'(r)<0\text{ for }r>0
   \end{array}\right\}.
\]
By Theorem~\ref{cs:thm:stationary-summary} and
Lemma~\ref{cs:lem:stationary-existence},
\begin{equation}\label{cs:eq:stationary-parameterization}
 \mathcal S_0(J,f)
 =\{0\}\cup
   \{x\mapsto V(|x-a|):V\in\mathcal R(J,f),\ a\in\R^d\},
 \qquad \mathcal R(J,f)\ne\varnothing.
\end{equation}
By \eqref{cs:eq:stationary-parameterization}, uniqueness up to
translation of a positive decaying stationary state is equivalent to
$\#\mathcal R(J,f)=1$.
Another question for stationary solutions in $[0,1]$ is to determine
\[
 \mathcal S_\vartheta(J,f)
 =\{U\in\mathcal S(J,f):U(x)\to\vartheta
                   \text{ as }|x|\to\infty\}.
\]
Example~\ref{cs:ex:periodic} constructs members of
$\mathcal S(J,f)$ that are periodic in one coordinate for the
Gaussian kernel and cubic reaction defined there.
\end{remark}

\subsection{Local stationary profiles and a periodic example}

\begin{remark}\label{cs:rem:local-classification}
For the local equation considered by Du and Matano~\cite{DuMatano2010},
let $f$ satisfy \textup{(F1)--(F2)}.
Every $C^2$ solution of
\begin{equation}\label{cs:eq:local-stationary}
 Q''+f(Q)=0\quad\text{on }\R,\qquad 0\le Q\le1,
\end{equation}
is a constant $0$, $\vartheta$, or $1$, a translate of a unique
positive even function $Q_0$ satisfying
\begin{equation}\label{cs:eq:local-zero-limit}
 Q_0(0)=\beta,\qquad Q_0'(x)<0\ (x>0),\qquad
 Q_0(x)\to0\quad(|x|\to\infty),
\end{equation}
or a nonconstant periodic function.
Up to translation, the periodic functions are indexed by
$c\in(F(\vartheta),0)$.
Their minimum $a_c$, maximum $b_c$, and period $T_c$ are given by
\begin{equation}\label{cs:eq:local-period}
\begin{gathered}
 0<a_c<\vartheta<b_c<\beta,\qquad F(a_c)=F(b_c)=c,\\
 T_c=\sqrt2\int_{a_c}^{b_c}\frac{ds}{\sqrt{c-F(s)}}.
\end{gathered}
\end{equation}

Multiplication of \eqref{cs:eq:local-stationary} by $Q'$ gives
the first integral
\begin{equation}\label{cs:eq:local-first-integral}
 \tfrac12(Q')^2+F(Q)=c.
\end{equation}
By the bistable signs, $F$ decreases from $0$ to $F(\vartheta)$
and increases thereafter to $F(1)>0$.
For $F(\vartheta)<c<0$, \eqref{cs:eq:local-first-integral} therefore
gives a closed orbit between the simple roots $a_c,b_c$.
Integration along its two branches gives \eqref{cs:eq:local-period};
initial-value uniqueness gives uniqueness up to translation.
At $c=0$, the orbit through $(\beta,0)$ satisfies
\begin{equation}\label{cs:eq:local-quadrature}
 x=\int_{Q_0(x)}^\beta\frac{ds}{\sqrt{-2F(s)}},\qquad x\ge0.
\end{equation}
Since $F(s)=f'(0)s^2/2+o(s^2)$ at zero, the integral diverges there.
Even reflection therefore gives the unique profile in
\eqref{cs:eq:local-zero-limit}.
For $c<F(\vartheta)$ there is no admissible value of $Q$, while
$c=F(\vartheta)$ gives $Q\equiv\vartheta$.
For $0<c\le F(1)$, every nonconstant orbit crosses $Q=0$ in finite
time in one direction: its only possible turning point lies above
$\beta$, and $|Q'|=\sqrt{2c}$ at zero.
For $c>F(1)$, $|Q'|$ is bounded away from zero on $[0,1]$,
so an endpoint is reached in finite time.
Finally, uniqueness at $(Q,Q')=(0,0)$ and $(1,0)$ gives the
remaining constant solutions.  These cases exhaust
\eqref{cs:eq:local-first-integral}.
\end{remark}

Bhowmik, Duncan, Grinfeld and Lord
\cite{BhowmikDuncanGrinfeldLord2011} studied periodic stationary
bifurcations for convolution equations with a Gaussian kernel and
cubic reaction.  The following example gives an explicit local
branch for the reaction in \eqref{cs:eq:gaussian-cubic}.

\begin{example}\label{cs:ex:periodic}
For every $d\ge1$, set
\begin{equation}\label{cs:eq:gaussian-cubic}
 J(x)=(4\pi)^{-d/2}e^{-|x|^2/4},\qquad
 f(s)=s(s-\tfrac14)(1-s).
\end{equation}
These functions satisfy \eqref{cs:eq:reaction-regularity},
\textup{(F1)--(F2)}, and \textup{(KM)}.
There are $s_0>0$, a real analytic function
$k:(-s_0,s_0)\to(0,\infty)$, and nonconstant smooth members
$U_s\in\mathcal S(J,f)$ for $0<|s|<s_0$ such that
\begin{equation}\label{cs:eq:periodic-branch}
 k(0)=\sqrt{\log(16/13)},\qquad
 U_s(x)=\tfrac14+s\cos(k(s)x_1)+O(s^2).
\end{equation}
The estimate in \eqref{cs:eq:periodic-branch} is uniform on $\R^d$
as $s\to0$.
Each $U_s$ is periodic in $x_1$, independent of the other coordinates,
and satisfies $0<U_s<1$.
\end{example}

\begin{proof}
For the polynomial in \eqref{cs:eq:gaussian-cubic},
direct calculation gives
\begin{equation}\label{cs:eq:cubic-constants}
 \vartheta=\tfrac14,\qquad F(1)=\tfrac1{24},\qquad
 f'(0)=-\tfrac14,\quad f'(\vartheta)=\tfrac3{16},\quad
 f'(1)=-\tfrac34,\qquad \max_{[0,1]}f'=\tfrac{13}{48}.
\end{equation}
Together with the Gaussian formula, these identities verify all
the assumptions.

Let $\mathcal A$ be the real Banach algebra of even
$2\pi$-periodic functions with absolutely summable cosine coefficients:
\[
 w(y)=a_0+\sum_{n=1}^\infty a_n\cos(ny),\qquad
 \|w\|_{\mathcal A}=|a_0|+\sum_{n=1}^\infty|a_n|.
\]
The cosine product identity gives
$\|vw\|_{\mathcal A}\le\|v\|_{\mathcal A}\|w\|_{\mathcal A}$.
For $k>0$, define $T_k:\mathcal A\to\mathcal A$ by
\[
 T_kw=a_0+\sum_{n=1}^\infty e^{-n^2k^2}a_n\cos(ny).
\]
Integration in the transverse coordinates and the Fourier transform
of the Gaussian show that
$U(x)=\tfrac14+w(kx_1)$ is stationary exactly when
\begin{equation}\label{cs:eq:periodic-equation}
 \mathcal F(w,k):=T_kw-w+f(\tfrac14+w)=0.
\end{equation}
Absolute convergence justifies termwise convolution.

For each positive $k$, choose a small complex neighborhood on which
$\operatorname{Re}(k^2)>0$, where $\operatorname{Re}$ denotes the real part.
The multipliers $e^{-n^2k^2}$ are uniformly bounded in $n$ there;
Cauchy's estimates on a smaller neighborhood give uniform bounds
for their Taylor coefficients.
Thus $k\mapsto T_k$ is analytic in operator norm, and
\eqref{cs:eq:periodic-equation} defines an analytic map
$\mathcal F:\mathcal A\times(0,\infty)\to\mathcal A$.
By \eqref{cs:eq:cubic-constants}, its linearization at $w=0$
has multipliers
\begin{equation}\label{cs:eq:periodic-eigenvalues}
 \lambda_n(k)=e^{-n^2k^2}-\mu,\qquad
 \mu=\tfrac{13}{16},\qquad n=0,1,2,\ldots.
\end{equation}
At $k_0=\sqrt{\log(16/13)}$, only $\lambda_1(k_0)$ vanishes:
$\lambda_0=3/16>0$ and
$\lambda_n(k_0)\le\mu^4-\mu<0$ for $n\ge2$.

Set $e(y)=\cos y$ and define the bounded projection by
\[
 P\!\left(a_0+\sum_{n=1}^\infty a_n\cos(ny)\right)=a_1e.
\]
Write $\mathcal A=\R e\oplus Z$, where $Z=\ker P$.
By \eqref{cs:eq:periodic-eigenvalues}, the restriction of
$D_w\mathcal F(0,k_0)$ to $Z$ has a bounded inverse.
The analytic implicit-function theorem applied to
\[
 (I-P)\mathcal F(se+z,k)=0
\]
gives a unique analytic $z=z(s,k)\in Z$ near $(0,k_0)$.
By $\mathcal F(0,k)=0$ and uniqueness, $z(0,k)=0$.
Differentiation in $s$, using the invariant Fourier subspaces in
\eqref{cs:eq:periodic-eigenvalues}, gives $\partial_sz(0,k)=0$.
Hence $z(s,k)=O(s^2)$ in $\mathcal A$, uniformly for $k$ near $k_0$.

Define the real analytic scalar function $r$ by
$P\mathcal F(se+z(s,k),k)=r(s,k)e$.
Since $r(0,k)=0$, it factors as
$r(s,k)=s\widetilde r(s,k)$, where
\[
 \widetilde r(s,k)=\int_0^1\partial_sr(\tau s,k)\,d\tau.
\]
By \eqref{cs:eq:periodic-eigenvalues},
\[
 \widetilde r(0,k)=e^{-k^2}-\mu,\qquad
 \partial_k\widetilde r(0,k_0)=-2k_0\mu\ne0.
\]
A second application of the analytic implicit-function theorem
gives $s_0>0$ and an analytic
$k:(-s_0,s_0)\to(0,\infty)$ with $\widetilde r(s,k(s))=0$.
Thus $w_s=se+z(s,k(s))$ solves \eqref{cs:eq:periodic-equation};
its first cosine coefficient $s\ne0$ makes it nonconstant.
Since $\|w_s\|_\infty\le\|w_s\|_{\mathcal A}\le|s|+O(s^2)$,
decreasing $s_0$ gives $0<\tfrac14+w_s<1$.
Setting $U_s(x)=\tfrac14+w_s(k(s)x_1)$ yields
\eqref{cs:eq:periodic-branch}.
Repeated differentiation of $U_s=H^{-1}(J*U_s)$, using integrability
of every Gaussian derivative, gives smoothness.
\end{proof}

\begingroup
\small
\linespread{0.96}\selectfont
\setlength{\parskip}{0pt}
\setlength{\itemsep}{0pt}

\endgroup

\end{document}